\documentclass[11pt]{article}

\usepackage[margin=1.25in]{geometry}

\usepackage{amssymb,amsmath,bm}
\usepackage{xcolor}
\usepackage{amsthm}
\usepackage{hyperref}
\usepackage{mathrsfs}
\usepackage{textcomp}

\allowdisplaybreaks[1]

\usepackage{soul}
\soulregister\cite7 
\soulregister\citep7 
\soulregister\citet7 
\soulregister\ref7 
\soulregister\pageref7 
\soulregister\eqref7

\hypersetup{
    colorlinks,%
    citecolor=blue,%
    filecolor=blue,%
    linkcolor=blue,%
    urlcolor=black
}

\usepackage{verbatim}

\usepackage{sectsty}

\makeatletter

\newdimen\bibspace
\renewenvironment{thebibliography}[1]{%
 \section*{\refname 
       \@mkboth{\MakeUppercase\refname}{\MakeUppercase\refname}}%
     \list{\@biblabel{\@arabic\c@enumiv}}%
          {\settowidth\labelwidth{\@biblabel{#1}}%
           \leftmargin\labelwidth
           \advance\leftmargin\labelsep
           \itemsep\bibspace
           \parsep\z@skip     %
           \@openbib@code
           \usecounter{enumiv}%
           \let\p@enumiv\@empty
           \renewcommand\theenumiv{\@arabic\c@enumiv}}%
     \sloppy\clubpenalty4000\widowpenalty4000%
     \sfcode`\.\@m}
    {\def\@noitemerr
      {\@latex@warning{Empty `thebibliography' environment}}%
     \endlist}

\makeatother

\makeatletter

\newtheorem{thm}{Theorem}[section]
\newtheorem{lem}[thm]{Lemma}
\newtheorem{prop}[thm]{Proposition}

\def\XXint#1#2#3{{\setbox0=\hbox{$#1{#2#3}{\int}$}
  \vcenter{\hbox{$#2#3$}}\kern-.5\wd0}}

\newcommand{\al}{\alpha}                \newcommand{\lda}{\lambda}
\newcommand{\om}{\Omega}                \newcommand{\pa}{\partial}
\newcommand{\va}{\varepsilon}           \newcommand{\ud}{\mathrm{d}}
\newcommand{\be}{\begin{equation}}      \newcommand{\ee}{\end{equation}}

\newcommand{\Lda}{\Lambda}              \newcommand{\V}{\mathscr{V}}
\newcommand{\R}{\mathbb{R}}              
  \newcommand{\A}{\mathscr{A}}

\newcommand{\loc}{\text{loc}}

\begin{document}

\title{\textbf{{\L}ojasiewicz--Simon inequalities near bubbling
configurations  for the Yamabe functional  on bounded domains}
\bigskip}

\author{\medskip  Tianling Jin\footnote{T. Jin is partially supported by NSFC grant 12122120, and Hong Kong RGC grants GRF 16304125, GRF 16303624  and GRF 16303822.}, \  \
Jingang Xiong\footnote{J. Xiong is partially supported by NSFC grant 12325104.}, \  \ Ning Zhou\footnote{N. Zhou is partially supported by NSFC grant 12601198.}}

\date{\today}

\maketitle

\begin{abstract}
We establish {\L}ojasiewicz--Simon type gradient inequalities for the Yamabe functional on bounded smooth domains near configurations consisting of finitely many concentrating bubbles, with or without a regular component. A weighted decomposition separates the finite-dimensional spectral and bubble parameters from an infinite-dimensional coercive remainder. We obtain quantitative estimates for the parameters, the remainder, and the corresponding energy
gaps, in both the regular-plus-bubbling and the pure-bubbling regimes. These inequalities quantify the deviation from such surfaces and are essential for studying the dynamics of related parabolic flows.

\medskip


\end{abstract}

\section{Introduction}

Let $\Omega\subset\mathbb{R}^n$ be an open set, let
$f:\Omega\to\mathbb{R}$ be real analytic, and suppose that
$x_0\in\Omega$ is a critical point of $f$. In his seminal work
\cite{L}, {\L}ojasiewicz proved that there exist constants
$\gamma\in(0,1]$, $C>0$, and $\delta>0$ such that
$B_\delta(x_0)\subset\Omega$ and
\[
 |f(x)-f(x_0)|
 \leq C|\nabla f(x)|^{1+\gamma},
 \qquad x\in B_\delta(x_0).
\]
Simon \cite{Simon} developed an infinite-dimensional analogue of
this gradient inequality for a class of analytic functionals arising
from elliptic variational problems. The resulting
{\L}ojasiewicz--Simon inequality has proved to be a powerful tool for understanding the asymptotic behavior of nonlinear evolution
equations. In particular, such inequalities play an important role
in determining convergence rates for Yamabe-type flows; see, for
example, Brendle \cite{Br05} and
Carlotto--Chodosh--Rubinstein \cite{CCR}. Part of the motivation for the present work comes from the the study of Sobolev-critical fast diffusion in bounded domains, where related bubbling dynamics were studied in \cite{JX20-a,JX24}.

We consider the Sobolev quotient
\[
 Y_\Omega(\xi)
 :=
 \frac{\displaystyle\int_\Omega |\nabla\xi|^2\,\ud x}
 {\left(\displaystyle\int_\Omega|\xi|^{\frac{2n}{n-2}}\,\ud x
 \right)^{\frac{n-2}{n}}},
 \qquad
 \xi\in H_0^1(\Omega)\setminus\{0\},
\]
where $\Omega\subset\mathbb{R}^n$ is a bounded smooth domain,
$n\geq3$. Here $H_0^1(\Omega)$ denotes the closure of
$C_c^\infty(\Omega)$ in $H^1(\Omega)$, equipped with the equivalent
norm
\[
 \|\xi\|
 :=
 \left(\int_\Omega|\nabla\xi|^2\,\ud x\right)^{1/2}.
\]

The classical abstract {\L}ojasiewicz--Simon theory cannot be applied
directly to $Y_\Omega$ for two reasons. First, the functional
$Y_\Omega$ is not, in general, real analytic on
$H_0^1(\Omega)\setminus\{0\}$. Second, the functional may have no
nontrivial critical point. Indeed, a critical point on the normalized
slice satisfies an equation of the form
\be\label{eq:EL}
 \begin{cases}
 -\Delta S
 =
 r_\infty |S|^{\frac{4}{n-2}}S
 &\mbox{in }\Omega,\\
 S=0
 &\mbox{on }\partial\Omega,
 \end{cases}
\ee
where $r_\infty>0$. If $\Omega$ is star-shaped, the Pohozaev
identity rules out nontrivial $H_0^1(\Omega)$ weak solutions of
\eqref{eq:EL}; see \cite{P,BrezisKato}.

To overcome the first difficulty, fix $p>n$ and restrict $Y_\Omega$ to the open cone
\[
 \mathcal D
 :=
 \left\{
 v\in  W^{2,p}(\Omega)\cap W_0^{1,p}(\Omega):
 0<\inf_{x\in\Omega}\frac{v(x)}{d(x)}
 \leq
 \sup_{x\in\Omega}\frac{v(x)}{d(x)}
 <\infty
 \right\},
\]
where
\[
 d(x):=\operatorname{dist}(x,\partial\Omega).
\]
Following the argument of Feireisl-Simondon \cite{FS}, for every
$v\in\mathcal D$ there exists an open neighborhood
$U(v)\subset W^{2,p}(\Omega)\cap W_0^{1,p}(\Omega)$ such that $Y_\Omega$ is real analytic on $U(v)$.

The second difficulty is related to the loss of compactness at the
critical Sobolev exponent. Define
\[
 U_{a,\lambda}(x)
 :=
 [n(n-2)]^{\frac{n-2}{4}}
 \left(
 \frac{\lambda}
 {1+\lambda^2|x-a|^2}
 \right)^{\frac{n-2}{2}},
 \qquad
 a\in\mathbb{R}^n,\quad \lambda>0.
\]
Then
\[
 -\Delta U_{a,\lambda}=U_{a,\lambda}^{\frac{n+2}{n-2}}
 \qquad\mbox{in }\mathbb{R}^n.
\]
Let
\[
 PU_{a,\lambda}:=U_{a,\lambda}-h_{a,\lambda},
\]
where $h_{a,\lambda}$ is the harmonic extension of the boundary
values of $U_{a,\lambda}$:
\[
 \begin{cases}
 \Delta h_{a,\lambda}=0
 &\mbox{in }\Omega,\\
 h_{a,\lambda}=U_{a,\lambda}
 &\mbox{on }\partial\Omega.
 \end{cases}
\]
Equivalently,
\[
 \begin{cases}
 -\Delta PU_{a,\lambda}=U_{a,\lambda}^{\frac{n+2}{n-2}}
 &\mbox{in }\Omega,\\
 PU_{a,\lambda}=0
 &\mbox{on }\partial\Omega.
 \end{cases}
\]
By the classification theorem of Caffarelli-Gidas-Spruck
\cite{CGS}, the functions
\[
 r_\infty^{-\frac{n-2}{4}}U_{a,\lambda},
 \qquad a\in\mathbb{R}^n,\quad\lambda>0,
\]
are precisely all the positive solutions of
\[
 -\Delta U=r_\infty U^{\frac{n+2}{n-2}}
 \qquad\mbox{in }\mathbb{R}^n.
\]
We refer to $U_{a,\lambda}$ and $PU_{a,\lambda}$ as bubbles centered
at $a$, with concentration parameter $\lambda$. Let
\[
 S_n
 :=
 \inf\left\{
 \int_{\mathbb{R}^n}|\nabla\eta|^2\,\ud x:
 \eta\in C_c^\infty(\mathbb{R}^n),\quad
 \int_{\mathbb{R}^n}|\eta|^{\frac{2n}{n-2}}\,\ud x=1
 \right\}
\]
be the optimal Sobolev constant. With the normalization above,
\[
 \int_{\mathbb{R}^n}|\nabla U_{a,\lambda}|^2\,\ud x
 =
 \int_{\mathbb{R}^n}U_{a,\lambda}^{\frac{2n}{n-2}}\,\ud x
 =
 S_n^{\frac n2}.
\]
Motivated by Struwe's global compactness theorem \cite{St}, we study
functions close to configurations of the form
\[
 u_\infty
 +
 r_\infty^{-\frac{n-2}{4}}
 \sum_{i=1}^{\ell}PU_{a_i,\lambda_i},
\]
where $\ell\geq0$, and either $u_\infty\equiv0$ or
$u_\infty\in\mathcal D$ is a solution of \eqref{eq:EL}. We exclude
the case $u_\infty\equiv0$ and $\ell=0$. As usual, when $\ell=0$,
the bubble part is absent. Quantitative versions of Struwe's decomposition have 
been obtained by Ciraolo--Figalli--Maggi \cite{CFM}, Figalli--Glaudo \cite{FG2020} and Deng--Sun--Wei \cite{DSW}. A related result was obtained by Malchiodi--Rupflin--Sharp
\cite{MRS}, who established {\L}ojasiewicz inequalities near simple
bubble trees for the $H$-functional, whereas we study the Dirichlet
Yamabe functional near finitely many interacting bubbles, possibly
with a nontrivial regular component.

Throughout the paper, $u_\infty$, $\ell$, and $r_\infty$ are assumed
to satisfy the compatibility relation
\be\label{eq:r-infty-level}
 r_\infty
 =
 \left(
 Y_\Omega(u_\infty)^{\frac n2}
 +
 \ell S_n^{\frac n2}
 \right)^{\frac2n},
\ee
where we use the convention $Y_\Omega(u_\infty)=0$ when
$u_\infty\equiv0$. 

For
\[
 \boldsymbol\alpha=(\alpha_1,\ldots,\alpha_\ell)\in\mathbb{R}^\ell,
 \qquad
 \boldsymbol\lambda=(\lambda_1,\ldots,\lambda_\ell)
 \in(0,\infty)^\ell, \qquad \boldsymbol a=(a_1,\ldots,a_\ell)\in\Omega^\ell,
\]
define the interaction parameters (see Br\'ezis-Coron \cite{BC})
\[
 \Lambda_{i,j}(\boldsymbol a,\boldsymbol\lambda)
 :=
 \left(
 \frac{\lambda_i}{\lambda_j}
 +
 \frac{\lambda_j}{\lambda_i}
 +
 \lambda_i\lambda_j|a_i-a_j|^2
 \right)^{\frac{2-n}{2}},
 \qquad i\neq j,
\]
and set $\Lambda_{i,i}=0$. We use the notation
\[
\bm\Lambda
:=
\sum_{i,j=1}^{\ell}
\Lambda_{i,j}(\bm a,\bm\lambda).
\]

We restrict attention to interior bubbling configurations for which
\be\label{eq:center-boundary}
 d(a_i)\geq\delta_0>0,
 \qquad 1\leq i\leq\ell,
\ee
where $\delta_0$ is fixed. As
\[
 \sum_{i=1}^{\ell}\lambda_i^{-1}
 +
 \sum_{i=1}^{\ell}|\alpha_i-1|
 +
\bm{\Lambda}
 \longrightarrow0,
\]
one has
\[
 Y_\Omega\left(
 u_\infty
 +
 r_\infty^{-\frac{n-2}{4}}
 \sum_{i=1}^{\ell}\alpha_iPU_{a_i,\lambda_i}
 \right)
 \longrightarrow
 \left(
 Y_\Omega(u_\infty)^{\frac n2}
 +
 \ell S_n^{\frac n2}
 \right)^{\frac2n}
 =
 r_\infty.
\]

Since $Y_\Omega$ is invariant under multiplication by nonzero
constants, we work on the normalized slice
\be\label{eq:normalizedmass}
 \int_\Omega u^{\frac{2n}{n-2}}\,\ud x=1.
\ee

Suppose that $u\in\mathcal D$ lies in a sufficiently small
neighborhood of one of the above configurations. In Section~\ref{sec:2}, we
construct a finite-dimensional surface $\Sigma\subset H_0^1(\Omega)$
adapted to the bubble parameters and to the relevant spectral
directions of the linearized operator. This is a standard
finite-dimensional reduction framework; see, for example,
\cite{Bahri,Bahri-C,St,DruetHebeyRobertBlow-up2004}.

Let $\sigma\in\Sigma$ denote the projection of $u$ onto
$\Sigma$; see \eqref{eq:sigma11} and \eqref{eq:decomposition-1} for the
precise definitions. We write
\[
 u=\sigma+w,
 \qquad
 \sigma\in\Sigma,
\]
where $w$ satisfies the corresponding orthogonality conditions. The
surface is chosen so that the second variation is coercive in the
$w$-directions. The bubble component of $\sigma$ is parametrized by
$\boldsymbol\alpha$, $\boldsymbol\lambda$, and $\boldsymbol a$, and
\[
 \left\|
 \sigma-u_\infty
 -
 r_\infty^{-\frac{n-2}{4}}
 \sum_{i=1}^{\ell}\alpha_iPU_{a_i,\lambda_i}
 \right\|
\]
is small.

For a positive function $u\in\mathcal D$, define
\[
 \mathcal R_u
 :=
 -u^{-\frac{n+2}{n-2}}\Delta u
\]
and introduce the Euler--Lagrange residual at the limiting level
$r_\infty$,
\[
 \mathfrak G(u)
 :=
 \left\|
 \Delta u+r_\infty u^{\frac{n+2}{n-2}}
 \right\|_{L^{\frac{2n}{n+2}}(\Omega)}.
\]
Since
\[
 \Delta u+r_\infty u^{\frac{n+2}{n-2}}
 =
 -(\mathcal R_u-r_\infty)u^{\frac{n+2}{n-2}},
\]
we also have
\be\label{eq:Gu}
 \mathfrak G(u)
 =
 \left(
 \int_\Omega
 |\mathcal R_u-r_\infty|^{\frac{2n}{n+2}}
 u^{\frac{2n}{n-2}}\,\ud x
 \right)^{\frac{n+2}{2n}}.
\ee

The main purpose of this paper is to establish gradient and
energy-gap estimates near the almost-critical surfaces described
above. The precise projection estimates are stated in
Theorem~\ref{thm:LS}, and their energy consequences are collected in
Section~\ref{sec:energy-gap}. We record here the principal
conclusions.

If $u_\infty\equiv0$, then
\[
\begin{aligned}
-C\mathfrak G(u)^2
-C\bm\Lambda
-C\sum_{k=1}^{\ell}\lambda_k^{2-n}
\leq{}&Y_\Omega(u)-r_\infty
\leq{}
C\mathfrak G(u)^2
+C\sum_{k=1}^{\ell}\lambda_k^{2-n}.
\end{aligned}
\]
If $u_\infty>0$, then
\[
\begin{aligned}
Y_\Omega(u)-r_\infty
\geq
-C\mathfrak G(u)-
\begin{cases}
\displaystyle
C\sum_{k=1}^{\ell}
\lambda_k^{-\frac{n+2}{4}}
(\ln\lambda_k)^{\frac{n+2}{2n}}+C\bm\Lambda^{\frac{n+2}{2(n-2)}}
\left(
\ln\bm\Lambda^{-1}
\right)^{\frac{n+2}{2n}},
\ n\geq6,\\[3mm]
\displaystyle
C\sum_{k=1}^{\ell}\lambda_k^{\frac{2-n}{2}}+C\bm\Lambda,
\quad n=3,4,5,
\end{cases}
\end{aligned}
\]
and
\[
Y_\Omega(u)-r_\infty
\leq C\mathfrak G(u)^{1+\gamma},
\]
where $\gamma\in(0,1]$ is a constant.

These results are of {\L}ojasiewicz--Simon type. The norm
$\mathfrak G(u)$ plays the role of the dual norm of the derivative
of $Y_\Omega$, but remains well adapted to the critical Sobolev
nonlinearity and to bubbling configurations.

The parameters and the orthogonal remainder are also
controlled by the residual $\mathfrak G(u)$. The precise estimates for the amplitude
parameters, the remainder $w$, and the regular-part parameters are
stated in Theorem~\ref{thm:LS}.

The paper is organized as follows. Section~\ref{sec:2} develops the
spectral decomposition, constructs the finite-dimensional regular
and bubble surfaces, and proves the raw bubble-energy expansion.
Section~\ref{sec:modulation} selects the modulation parameters,
records the orthogonality and coercivity properties, and establishes
auxiliary estimates and preliminary energy identities. Section~\ref{sec:gradient} proves the
coupled amplitude and remainder estimates of
Theorem~\ref{thm:LS}. Finally, Section~\ref{sec:energy-gap} derives the
energy-gap estimates from the preceding results.

\medskip

\textbf{Declaration on the Use of AI Tools.}
An initial version of this manuscript was prepared by the authors without the use of AI-assisted tools. The authors subsequently used AI-assisted tools to help check certain calculations, including some integral computations, and to polish the language. All mathematical statements, proofs, logical arguments, citations, and conclusions in the final manuscript were independently reviewed and verified by the authors. The authors take full responsibility for the content of the manuscript.

\section{Construction of finite-dimensional surfaces}\label{sec:2}

For $\xi,\eta\in H_0^1(\Omega)$, we define
\be\label{eq:innerproduct}
 \langle \xi,\eta\rangle
 :=
 \int_\Omega \nabla\xi\cdot\nabla\eta\,\ud x.
\ee
For every $u\in\mathcal D$, let
\[
 \mathcal L_u(\Omega)
 :=
 L^2\left(\Omega,u^{\frac4{n-2}}\,\ud x\right)
 =
 \left\{
 f\in L^2_{\loc}(\Omega):
 \int_\Omega f^2u^{\frac4{n-2}}\,\ud x<\infty
 \right\},
\]
equipped with the inner product
\[
 \langle f,g\rangle_{\mathcal L_u}
 :=
 \int_\Omega fgu^{\frac4{n-2}}\,\ud x.
\]

\subsection{Surface generated by a nontrivial critical point}

Suppose that $u_\infty$ is a nonnegative solution of
\eqref{eq:EL}. By the strong maximum principle, either
$u_\infty\equiv0$ or $u_\infty>0$ in $\Omega$. In the latter case,
elliptic regularity and the Hopf boundary lemma imply that
\[
 c\,d(x)\leq u_\infty(x)\leq C\,d(x)
 \qquad\mbox{in }\Omega
\]
for some constants $c,C>0$. In particular,
$u_\infty\in\mathcal D$.

In this subsection, we assume that $u_\infty\not\equiv0$. The operator
\[
f \longmapsto [- u_\infty^{-\frac{4}{n-2}}\Delta]^{-1}f
\]
is a bounded linear compact symmetric operator that maps $\mathcal{L}_{u_\infty}(\Omega)$ to itself. By the spectral theorem, there exist eigenfunctions
$\{\phi_l\}_{l\geq1}\subset H_0^1(\Omega)$ and eigenvalues
$\{\mu_l\}_{l\geq1}\subset(0,\infty)$ such that
\[
 0<\mu_1<\mu_2\leq\mu_3\leq\cdots,
 \qquad
 \mu_l\longrightarrow\infty,
\]
and
\be\label{eq:phil}
 \begin{cases}
 -\Delta\phi_l
 =
 \mu_lu_\infty^{\frac4{n-2}}\phi_l
 &\mbox{in }\Omega,\\
 \phi_l=0
 &\mbox{on }\partial\Omega.
 \end{cases}
\ee
The family $\{\phi_l\}_{l\geq1}$ is an orthonormal basis of
$\mathcal L_{u_\infty}(\Omega)$; in particular,
\be\label{eq:phiij}
 \int_\Omega
 \phi_i\phi_ju_\infty^{\frac4{n-2}}\,\ud x
 =
 \delta_{ij}.
\ee
By the regularity theory of linear elliptic equations, $\phi_l\in C^{2+\frac{4}{n-2}}(\overline\Omega)\cap C^{\infty}(\Omega)$ for every $\ell$.

Since
\[
 -\Delta u_\infty
 =
 r_\infty u_\infty^{\frac4{n-2}}\ u_\infty,
\]
and $u_\infty>0$, it follows from the characterization and simplicity
of the first eigenvalue that
\(
 \mu_1=r_\infty
\)
and
\(
 \phi_1
 =
 u_\infty
 \left(
 \int_\Omega u_\infty^{\frac{2n}{n-2}}\,\ud x
 \right)^{-\frac12}.
\)
Moreover,
\(
 \left\{
 \frac{\phi_l}{\sqrt{\mu_l}}:l\geq1
 \right\}
\)
is an orthonormal basis of $H_0^1(\Omega)$ with respect to the inner
product \eqref{eq:innerproduct}.

Let
\[
 L
 :=
 \max\left\{
 l\in\mathbb N:
 \mu_l\leq \frac{n+2}{n-2} r_\infty
 \right\}.
\]
For
$f\in L^p(\Omega)$, $1\leq p<\infty$, define
\be\label{eq:Pif}
 \Pi f
 :=
 f-
 \sum_{i=1}^L
 \left(
 \int_\Omega f\phi_i\,\ud x
 \right)
u_\infty^{\frac4{n-2}}\phi_i.
\ee
Then $\Pi$ is a bounded projection on $L^p(\Omega)$, and
\[
 \Pi(L^p(\Omega))
 =
 \left\{
 f\in L^p(\Omega):
 \int_\Omega f\phi_i\,\ud x=0,\quad 1\leq i\leq L
 \right\}.
\]
In particular, $\Pi(L^p(\Omega))$ is a closed subspace of
$L^p(\Omega)$.

The following estimates are from
Lemmas 4.7 and 4.8 of \cite{JX20-a}; see also Brendle \cite{Br05}.

\begin{lem}[Lemma 4.7 \& 4.8 of \cite{JX20-a}] \label{lem:4.7-8-jx}
(i) For every $1\le p<\infty$, we can find a constant $C$ depending only on $n,\Omega,p$ and $u_\infty$ such that
\[
\|f\|_{L^p(\om)} \le C \Big \|\Delta f+\frac{n+2}{n-2} r_\infty u_\infty^{\frac{4}{n-2}} f \Big\|_{L^p(\om)}+C \sup_{1\le l\le L} \Big|\int_\om u_\infty^{\frac{4}{n-2}} \phi_l f \,\ud x\Big|
\]
for all $f\in W^{2,p}(\Omega)\cap W^{1,p}_0(\Omega)$.

(ii) There exists a constant $C$ depending only on $n,\Omega$ and $u_\infty$ such that
\begin{itemize}
\item
\[
\|f\|_{L^{\frac{n+2}{n-2}}(\om)}\le C \Big \| \Pi(\Delta f+\frac{n+2}{n-2}r_\infty u_\infty^{\frac{4}{n-2}} f )\Big\|_{L^{\frac{n(n+2)}{n^2+4}}(\om)}+C \sup_{1\le l\le L} \Big|\int_\om u_\infty^{\frac{4}{n-2}} \phi_l f \,\ud x \Big|
\]
for all $f\in W^{2,\frac{n(n+2)}{n^2+4}}(\Omega)\cap W^{1,\frac{n(n+2)}{n^2+4}}_0(\Omega)$.
\item
\[
\|f\|_{L^{1}(\om)}\le C \Big \| \Pi(\Delta f+\frac{n+2}{n-2}r_\infty u_\infty^{\frac{4}{n-2}} f )\Big\|_{L^{1}(\om)}+C \sup_{1\le l\le L} \Big|\int_\om u_\infty^{\frac{4}{n-2}} \phi_l f \,\ud x \Big|
\]
for all $f\in W^{2,1}(\Omega)\cap W^{1,1}_0(\Omega)$.
\end{itemize}
\end{lem}
By the Lyapunov-Schmidt reduction, we have the following result.

\begin{lem}[Adaptation of Lemma 4.9 of \cite{JX20-a}]\label{lem:ift}
There exists $\delta_1>0$ such that, for every
$\bm z=(z_1,\dots,z_L)\in\mathbb R^L$ with
$|\bm z|\leq\delta_1$, there exists
$\xi_{\bm z}\in C^{\frac{3n-2}{n-2}}(\overline\Omega)$ satisfying
$1/2\leq\xi_{\bm z}/u_\infty\leq2$ in $\Omega$,
\[
\int_{\om} u_\infty^{\frac{4}{n-2}} (\xi_{\bm{z}}- u_\infty) \phi_l\,\ud x= z_l, \quad l=1,\dots, L,
\]
and
\begin{equation}\label{eq:projection0}
\Pi(\Delta \xi_{\bm{z}}+ r_\infty\xi_{\bm{z}} ^{\frac{n+2}{n-2}})=0.
\end{equation}
Furthermore, the map $\bm z\mapsto\xi_{\bm z}$ is real
analytic as a map into $W^{2,p}(\Omega)\cap W_0^{1,p}(\Omega)$,
and
\[
\left.\frac{\partial\xi_{\bm z}}{\partial z_l}\right|_{\bm z=0}
=\phi_l,
\qquad 1\leq l\leq L.
\]
\end{lem}

The image
\[
 \Sigma_\infty
 :=
 \left\{
 \xi_{\bm z}:
 |\bm z|<\delta_1
 \right\}
\]
is an $L$-dimensional real-analytic surface in $\mathcal D$. By the
construction in the proof of Lemma \ref{lem:ift},
\be\label{eq:xiz}
 \xi_{\bm z}
 =
 u_\infty
 +
 \sum_{l=1}^Lz_l\phi_l
 +
 h_{\bm z},
\ee
where
\[
h_{\bm z}\in
\operatorname{span}\{\phi_1,\ldots,\phi_L\}^{\perp},
\]
the orthogonal complement is taken with respect to
\eqref{eq:innerproduct}, and
\[
\|h_{\bm z}\|=O(|\bm z|^2)
\qquad\mbox{as }|\bm z|\to0.
\]

\begin{lem}\label{lem:L-ineq}
There exist constants $C,\delta_2>0$ and $\gamma\in(0,1)$,
depending on $u_\infty$, such that, if $|\bm z|\leq\delta_2$,
then
\[
\left|Y_\Omega(\xi_{\bm z})-Y_\Omega(u_\infty)\right|
\leq
C\left(
\sup_{1\leq l\leq L}
\left|
\int_\Omega
\left(
\Delta\xi_{\bm z}
+r_\infty\xi_{\bm z}^{\frac{n+2}{n-2}}
\right)\phi_l\,\ud x
\right|
\right)^{1+\gamma}.
\]
\end{lem}

\begin{proof}
Using the analyticity of $Y_{\om}(\cdot)$ shown in Lemma 5.3 of Feireisl-Simondon \cite{FS}, the proof is identical to that of Lemma 6.5 of \cite{Br05}.
\end{proof}

\subsection{Surface generated by bubbles}

Note that
\[
 \begin{split}
 \lambda^{\frac{n-2}{2}}U_{a,\lambda}(x)
 &=
[n(n-2)]^{\frac{n-2}{4}}
 \left(
 \frac{1}
 {\lambda^{-2}+|x-a|^2}
 \right)^{\frac{n-2}{2}}\\
 &\longrightarrow
[n(n-2)]^{\frac{n-2}{4}}|x-a|^{2-n}
 \quad\mbox{in }
 C_{\loc}^2(\mathbb R^n\setminus\{a\})
 \quad\mbox{as }\lambda\to\infty.
 \end{split}
\]
For $a\in\Omega$, let $H(a,\cdot)$ be the solution of
\[
 \begin{cases}
 -\Delta H(a,x)=0
 &\mbox{in }\Omega,\\[1mm]
 H(a,x)=[n(n-2)]^{\frac{n-2}{4}}|a-x|^{2-n}
 &\mbox{on }\partial\Omega.
 \end{cases}
\]

\begin{prop}[Proposition 1 of Rey \cite{Rey}] \label{prop:bubble-derivative}
Suppose $a\in \om$ with $\operatorname{dist}(a,\partial\Omega)\ge \delta>0$ and $\lda>1$.  Let $H(a, \cdot)$ be defined as above and $h_{a,\lda}=U_{a,\lda}-PU_{a,\lda}$. Then
\[
f_{a,\lda} (x):= h_{a,\lda}(x) - \lda^{-\frac{n-2}{2}} H(a,x), \quad  x\in \overline\om
\]
is a smooth function on $\overline \om$, also smooth in parameters $a$ and $\lda$. Moreover,  there hold
\begin{align*}
& |f_{a,\lda} | \le C \lda^{-\frac{n+2}{2}}, \quad |\pa_{a^j} f_{a,\lda}| \le  C \lda^{-\frac{n+2}{2}}, \quad |\pa_\lda f_{a,\lda} | \le  C\lda^{-\frac{n+4}{2}}\\
& |\pa_{a^j}  \pa_\lda f_{a,\lda} | \le  C\lda^{-\frac{n+4}{2}},\quad |\pa_{a^j} \pa_{a^k}   f_{a,\lda} | \le  C\lda^{-\frac{n+2}{2}},\quad |\pa_\lda^2 f_{a,\lda} | \le  C\lda^{-\frac{n+6}{2}},
\end{align*}
for every $j, k=1,\dots, n, $ where  $C$ depends only on $n, \om$ and $\delta$.
\end{prop}

Let $\ell\geq1$, and consider parameters
\[
 \bm\alpha
 =
 (\alpha_1,\ldots,\alpha_\ell)
 \in(0,\infty)^\ell,
 \qquad
 \bm\lambda
 =
 (\lambda_1,\ldots,\lambda_\ell)
 \in(1,\infty)^\ell,
\]
and
\[
 \bm a
 =
 (a_1,\ldots,a_\ell)
 \in\Omega^\ell,
 \qquad \mbox{with }
 a_k=(a_k^1,\ldots,a_k^n)\in\Omega.
\]
Then $ (\bm{\al},\bm{a},  \bm{\lda}) \mapsto \sum_i\alpha_iPU_{a_i,\lda_i} $ is a surface in $\mathcal{D}$. Recall the bubble interactions by
\begin{align*}
&\Lda_{i,j}(\bm{a}, \bm{\lda}):= \Big(\frac{\lda_i}{\lda_j}+\frac{\lda_j}{\lda_i}+ \lda_i\lda_j |a_i-a_j|^2 \Big)^{\frac{2-n}{2}} \quad \mbox{if }i\neq j,\\  &\mbox{and }\Lda_{i,i}(\bm{a}, \bm{\lda})=0.
\end{align*}
We shall restrict these parameters to a neighborhood in which
\[
 |\alpha_i-1|\ll1,
 \qquad
 \lambda_i^{-1}\ll1,
 \qquad
 d(a_i)\geq\delta_0.
\]

\subsection{The energy level of the bubbles}

For simplicity, we write $PU_i=PU_{a_i,\lambda_i}$ and
$h_i=h_{a_i,\lambda_i}$. Set
\begin{align*}
c_0&=[n(n-2)]^{\frac{n-2}{4}},\\
C_1&=
c_0^{\frac{n+2}{n-2}}
\int_{\mathbb R^n}
\left(\frac{1}{1+|y|^2}\right)^{\frac{n+2}{2}}
\,\ud y\\
C_2
&=
c_0^{\frac{2n}{n-2}}
\int_{\mathbb R^n}
\left(\frac{1}{1+|y|^2}\right)^n
\,\ud y.
\end{align*}
Notice that $C_2^{\frac2n}=S_n$. In the case
$u_\infty\equiv0$, $
r_\infty=C_2^{\frac2n}\ell^{\frac2n}.
$

\begin{prop}\label{prop:quantif-1}
Let
\be\label{eq:sigma11}
\sigma
=
r_\infty^{-\frac{n-2}{4}}
\sum_{i=1}^{\ell}\alpha_iPU_i.
\ee
There exist $\va_0\in(0,1)$ and $C>0$, depending only on
$\Omega$, $n$, $\ell$, and $\delta_0$, such that, if
\[
\sum_{i=1}^{\ell}|\alpha_i-1|
+
\sum_{i=1}^{\ell}\lambda_i^{-1}
+
\sum_{i\neq j}\Lambda_{i,j}
\leq\va_0,
\]
then
\be\label{eq:energy-sigma-alpha}
\begin{aligned}
Y_\Omega(\sigma)
={}&
r_\infty
-
C_1c_0r_\infty^{-\frac{n-2}{2}}
\sum_{i\neq j}\Lambda_{i,j}\\
&+
O\left(\sum_{i=1}^{\ell}|\alpha_i-1|^2\right)
+
o\left(\sum_{i\neq j}\Lambda_{i,j}\right)
+
O\left(\sum_{i=1}^{\ell}\lambda_i^{2-n}\right).
\end{aligned}
\ee
Here and below, the $o$-terms are uniform as
\[
\sum_{i=1}^{\ell}|\alpha_i-1|
+
\sum_{i=1}^{\ell}\lambda_i^{-1}
+
\sum_{i\neq j}\Lambda_{i,j}
\longrightarrow0
\]
under the condition $d(a_i)\geq\delta_0$.

\end{prop}

\begin{proof}
Since
\[
-\Delta PU_j
=
U_j^{\frac{n+2}{n-2}}
\qquad\mbox{in }\Omega,
\]
we have
\be\label{eq:sigma-gradient-pre}
\begin{aligned}
\int_\Omega|\nabla\sigma|^2\,\ud x
&=-\int_{\om} \sigma\Delta \sigma \,\ud x\\
&=r_\infty^{-\frac{n-2}{2}}
\sum_{i,j=1}^{\ell}
\alpha_i\alpha_j
\int_\Omega
PU_iU_j^{\frac{n+2}{n-2}}\,\ud x\\
&=
r_\infty^{-\frac{n-2}{2}}
\left[
\sum_{i=1}^{\ell}\alpha_i^2
\int_\Omega
PU_iU_i^{\frac{n+2}{n-2}}\,\ud x+
\sum_{i\neq j}\alpha_i\alpha_j
\int_\Omega
PU_iU_j^{\frac{n+2}{n-2}}\,\ud x
\right].
\end{aligned}
\ee

We first record the expansions of the integrals in
\eqref{eq:sigma-gradient-pre}. For $i\neq j$, we write
\[
\begin{aligned}
\int_\Omega
PU_iU_j^{\frac{n+2}{n-2}}\,\ud x
={}&
\int_{\mathbb R^n}
U_iU_j^{\frac{n+2}{n-2}}\,\ud x-
\int_{\mathbb R^n\setminus\Omega}
U_iU_j^{\frac{n+2}{n-2}}\,\ud x
-
\int_\Omega
h_iU_j^{\frac{n+2}{n-2}}\,\ud x.
\end{aligned}
\]
By \cite[(E1)]{Bahri},
\begin{equation}\label{eq:twodiffbubble}
\int_{\mathbb R^n}
U_iU_j^{\frac{n+2}{n-2}}\,\ud x
=
C_1c_0\Lambda_{i,j}
+
o(\Lambda_{i,j}).
\end{equation}
We use only the $o(\Lambda_{i,j})$ remainder here; no more precise
remainder is needed in the argument.

Since $d(a_i),d(a_j)\geq\delta_0$, Young's inequality gives
\[
\int_{\mathbb R^n\setminus\Omega}
U_iU_j^{\frac{n+2}{n-2}}\,\ud x
\leq \int_{\R^n\backslash\om} \left(U_{i}^{\frac{2n}{n-2}}+U_{j}^{\frac{2n}{n-2}}\right) \,\ud x 
=
O(\lambda_i^{-n})+O(\lambda_j^{-n}).
\]
Moreover, since $\Delta_x H(a_i,x)=0$, by Taylor expanding $H(a_i,x)$ at $x=a_j$ and using the symmetries of $U_{j}^{\frac{n+2}{n-2}}$, we obtain
\[
\begin{aligned}
\int_\Omega
U_j^{\frac{n+2}{n-2}}H(a_i,x)\,\ud x
={}&
C_1\lambda_j^{-\frac{n-2}{2}}H(a_i,a_j)
+
O\left(\lambda_j^{-\frac{n+2}{2}}\right).
\end{aligned}
\]
Hence, by using Proposition \ref{prop:bubble-derivative}, we obtain
\begin{align}
\int_\Omega
h_iU_j^{\frac{n+2}{n-2}}\,\ud x
&=\int_{\om} U_{j}^{\frac{n+2}{n-2}}(\lda_i^{-\frac{n-2}{2}} H(a_i,x)+O(\lda_i^{-\frac{n+2}{2}})) \,\ud x\nonumber\\
&=
C_1(\lambda_i\lambda_j)^{-\frac{n-2}{2}}
H(a_i,a_j)
+
O(\lambda_i^{-n})+O(\lambda_j^{-n}). \label{eq:remainderonebubble}
\end{align}
Consequently,
\be\label{eq:PUiUj^p}
\begin{aligned}
\int_\Omega
PU_iU_j^{\frac{n+2}{n-2}}\,\ud x
=
C_1c_0\Lambda_{i,j}-
C_1(\lambda_i\lambda_j)^{-\frac{n-2}{2}}
H(a_i,a_j)+
o(\Lambda_{i,j})
+
O(\lambda_i^{-n})+O(\lambda_j^{-n}).
\end{aligned}
\ee

The same argument with $i=j$ yields
\be\label{eq:PUiUi^p}
\int_\Omega
PU_iU_i^{\frac{n+2}{n-2}}\,\ud x
=
C_2
-
C_1\lambda_i^{2-n}H(a_i,a_i)
+
O(\lambda_i^{-n}).
\ee
Substituting \eqref{eq:PUiUj^p} and
\eqref{eq:PUiUi^p} into \eqref{eq:sigma-gradient-pre}, we obtain
\be\label{eq:sigma1}
\begin{aligned}
\int_\Omega|\nabla\sigma|^2\,\ud x
=
r_\infty^{-\frac{n-2}{2}}
\Bigg\{&
C_2\sum_{i=1}^{\ell}\alpha_i^2
-
C_1\sum_{i=1}^{\ell}
\alpha_i^2\lambda_i^{2-n}H(a_i,a_i)\\
&+
C_1\sum_{i\neq j}
\alpha_i\alpha_j
\left(
c_0\Lambda_{i,j}
-
(\lambda_i\lambda_j)^{-\frac{n-2}{2}}
H(a_i,a_j)
\right)\\
&+
o\left(\sum_{i\neq j}\Lambda_{i,j}\right)
+
O\left(\sum_{i=1}^{\ell}\lambda_i^{-n}\right)
\Bigg\}.
\end{aligned}
\ee

We now prove the expansion
\be\label{eq:sum-bubbles-power}
\begin{aligned}
\int_\Omega
\left(
\sum_{i=1}^{\ell}\alpha_iPU_i
\right)^{\frac{2n}{n-2}}
\,\ud x
={}&
\sum_{i=1}^{\ell}
\alpha_i^{\frac{2n}{n-2}}
\int_\Omega
(PU_i)^{\frac{2n}{n-2}}\,\ud x\\
&+
\frac{2n}{n-2}
\sum_{i\neq j}
\alpha_i\alpha_j^{\frac{n+2}{n-2}}
\int_\Omega
PU_i(PU_j)^{\frac{n+2}{n-2}}\,\ud x
+
o\left(\sum_{i\neq j}\Lambda_{i,j}\right).
\end{aligned}
\ee

For $\ell=1$, this identity is immediate. We therefore assume that
$\ell\geq2$. By decreasing $\va_0$ if necessary, we may suppose that
\[
\frac12\leq\alpha_i\leq\frac32,
\qquad 1\leq i\leq\ell.
\]
Fix
\[
0<\varepsilon<
\min\left\{1,\frac{2}{n-2}\right\}.
\]
Then
\[
\frac{n+2}{n-2}-\varepsilon>1+\varepsilon
\quad
\mbox{and}
\quad
\left(
\frac{n+2}{n-2}-\varepsilon
\right)
+
(1+\varepsilon)
=
\frac{2n}{n-2}.
\]

We first record the two-bubble estimate that will be used repeatedly.
Since $0<PU_i\leq U_i$ in $\Omega$, by Proposition B.2 of \cite{FG2020}, for every $i\neq j$,
\be\label{eq:two-bubble-epsilon}
\int_\Omega
(PU_i)^{\frac{n+2}{n-2}-\varepsilon}
(PU_j)^{1+\varepsilon}
\,\ud x
\leq \int_{\R^n}
(U_i)^{\frac{n+2}{n-2}-\varepsilon}
(U_j)^{1+\varepsilon}
\,\ud x
\leq
C\Lambda_{i,j}^{1+\varepsilon}.
\ee

We next introduce a measurable partition
\[
\Omega=A_1\cup\cdots\cup A_\ell
\]
such that the sets $A_i$ are pairwise disjoint and
\[
\alpha_iPU_i
=
\max_{1\leq k\leq\ell}\alpha_kPU_k
\qquad\mbox{on }A_i.
\]
Points at which the maximum is attained by more than one index may be
assigned to any one of the corresponding sets. In particular, on
$A_i$,
\be\label{eq:dominant-bubble}
\alpha_jPU_j\leq\alpha_iPU_i
\qquad\mbox{for every }j,
\ee
and hence
\[
PU_j\leq3PU_i
\qquad\mbox{on }A_i.
\]

For $s,t\geq0$ satisfying $t\leq(\ell-1)s$, Taylor's formula gives
\be\label{eq:algebraic-taylor-bubbles}
\left|
(s+t)^{\frac{2n}{n-2}}
-
s^{\frac{2n}{n-2}}
-
\frac{2n}{n-2}s^{\frac{n+2}{n-2}}t
\right|
\leq
Cs^{\frac4{n-2}}t^2.
\ee
On $A_i$, we apply \eqref{eq:algebraic-taylor-bubbles} with
\[
s=\alpha_iPU_i,
\qquad
t=\sum_{j\neq i}\alpha_jPU_j.
\]
The defining property of $A_i$ gives
$
t\leq(\ell-1)s.
$
Consequently,
\be\label{eq:power-expansion-on-Ai}
\begin{aligned}
\left(
\sum_{j=1}^{\ell}\alpha_jPU_j
\right)^{\frac{2n}{n-2}}
={}&
\alpha_i^{\frac{2n}{n-2}}
(PU_i)^{\frac{2n}{n-2}}+
\frac{2n}{n-2}
\alpha_i^{\frac{n+2}{n-2}}
(PU_i)^{\frac{n+2}{n-2}}
\sum_{j\neq i}\alpha_jPU_j\\
&+
O\left(
(PU_i)^{\frac4{n-2}}
\left(
\sum_{j\neq i}PU_j
\right)^2
\right)
\end{aligned}
\ee
on $A_i$, where we have used the uniform upper and lower bounds for
the coefficients $\alpha_i$. We claim that the integral of the last term in
\eqref{eq:power-expansion-on-Ai} is
$o(\sum_{i\neq j}\Lambda_{i,j})$. Indeed,
\[
\left(
\sum_{j\neq i}PU_j
\right)^2
\leq
(\ell-1)\sum_{j\neq i}(PU_j)^2.
\]
Moreover, by \eqref{eq:dominant-bubble}, on $A_i$ we have
$PU_j\leq3PU_i$. Since $\varepsilon<1$, it follows that
\[
\begin{aligned}
(PU_i)^{\frac4{n-2}}(PU_j)^2
&=
(PU_i)^{\frac{n+2}{n-2}-\varepsilon}
(PU_j)^{1+\varepsilon}
\left(\frac{PU_j}{PU_i}\right)^{1-\varepsilon}\leq
C
(PU_i)^{\frac{n+2}{n-2}-\varepsilon}
(PU_j)^{1+\varepsilon}
\end{aligned}
\]
on $A_i$. Therefore, by \eqref{eq:two-bubble-epsilon},
\be\label{eq:Taylor-remainder-bubbles}
\begin{aligned}
\sum_{i=1}^{\ell}
\int_{A_i}
(PU_i)^{\frac4{n-2}}
\left(
\sum_{j\neq i}PU_j
\right)^2
\,\ud x
&\leq
C\sum_{i\neq j}
\int_\Omega
(PU_i)^{\frac{n+2}{n-2}-\varepsilon}
(PU_j)^{1+\varepsilon}
\,\ud x\\
&\leq
C\sum_{i\neq j}\Lambda_{i,j}^{1+\varepsilon}\\
&=
o\left(\sum_{i\neq j}\Lambda_{i,j}\right).
\end{aligned}
\ee
In the last line, we used
$
\sum_{i\neq j}\Lambda_{i,j}^{1+\varepsilon}
\leq
\left(
\max_{i\neq j}\Lambda_{i,j}
\right)^\varepsilon
\sum_{i\neq j}\Lambda_{i,j}
=
o\left(\sum_{i\neq j}\Lambda_{i,j}\right).
$

Integrating \eqref{eq:power-expansion-on-Ai} over $A_i$, summing over
$i$, and using \eqref{eq:Taylor-remainder-bubbles}, we obtain
\be\label{eq:power-expansion-partition}
\begin{aligned}
\int_\Omega
\left(
\sum_{i=1}^{\ell}\alpha_iPU_i
\right)^{\frac{2n}{n-2}}
\,\ud x
={}&
\sum_{i=1}^{\ell}
\alpha_i^{\frac{2n}{n-2}}
\int_{A_i}
(PU_i)^{\frac{2n}{n-2}}
\,\ud x\\
&+
\frac{2n}{n-2}
\sum_{i\neq j}
\alpha_i^{\frac{n+2}{n-2}}\alpha_j
\int_{A_i}
(PU_i)^{\frac{n+2}{n-2}}PU_j
\,\ud x\\
&+
o\left(\sum_{i\neq j}\Lambda_{i,j}\right).
\end{aligned}
\ee

It remains to replace each $A_i$ by $\Omega$ in the first two terms.

We first consider the pure-power terms. Since the sets $A_k$ form a
partition of $\Omega$,
\[
\Omega\setminus A_i
=
\bigcup_{k\neq i}A_k.
\]
On $A_k$, the maximality of $\alpha_kPU_k$ gives
\[
PU_i\leq3PU_k.
\]
It follows that
\[
\begin{aligned}
(PU_i)^{\frac{2n}{n-2}}
&=
(PU_i)^{1+\varepsilon}
(PU_k)^{\frac{n+2}{n-2}-\varepsilon}
\left(
\frac{PU_i}{PU_k}
\right)^{\frac{n+2}{n-2}-\varepsilon}\leq
C
(PU_i)^{1+\varepsilon}
(PU_k)^{\frac{n+2}{n-2}-\varepsilon}
\end{aligned}
\]
on $A_k$. Thus, using \eqref{eq:two-bubble-epsilon} with the indices
interchanged,
\[
\begin{aligned}
\sum_{i=1}^{\ell}
\int_{\Omega\setminus A_i}
(PU_i)^{\frac{2n}{n-2}}
\,\ud x
&=
\sum_{i=1}^{\ell}
\sum_{k\neq i}
\int_{A_k}
(PU_i)^{\frac{2n}{n-2}}
\,\ud x\\
&\leq
C\sum_{i\neq k}
\int_\Omega
(PU_i)^{1+\varepsilon}
(PU_k)^{\frac{n+2}{n-2}-\varepsilon}
\,\ud x\\
&\leq
C\sum_{i\neq k}\Lambda_{i,k}^{1+\varepsilon}\\
&=
o\left(\sum_{i\neq j}\Lambda_{i,j}\right).
\end{aligned}
\]
Consequently,
\be\label{eq:extend-pure-term}
\begin{aligned}
\sum_{i=1}^{\ell}
\alpha_i^{\frac{2n}{n-2}}
\int_{A_i}(PU_i)^{\frac{2n}{n-2}}\,\ud x
={}&
\sum_{i=1}^{\ell}
\alpha_i^{\frac{2n}{n-2}}
\int_\Omega(PU_i)^{\frac{2n}{n-2}}\,\ud x+
o\left(\sum_{i\neq j}\Lambda_{i,j}\right).
\end{aligned}
\ee

We next consider the first-order interaction terms. We claim that
\be\label{eq:extend-cross-term}
\begin{aligned}
&\sum_{i\neq j}
\alpha_i^{\frac{n+2}{n-2}}\alpha_j
\int_{A_i}
(PU_i)^{\frac{n+2}{n-2}}PU_j\,\ud x\\
={}&
\sum_{i\neq j}
\alpha_i^{\frac{n+2}{n-2}}\alpha_j
\int_\Omega
(PU_i)^{\frac{n+2}{n-2}}PU_j\,\ud x
+
o\left(\sum_{i\neq j}\Lambda_{i,j}\right).
\end{aligned}
\ee
To prove this, write
\[
\begin{aligned}
&\sum_{i\neq j}
\int_{\Omega\setminus A_i}
(PU_i)^{\frac{n+2}{n-2}}PU_j\,\ud x=
\sum_{i\neq j}\sum_{k\neq i}
\int_{A_k}
(PU_i)^{\frac{n+2}{n-2}}PU_j\,\ud x.
\end{aligned}
\]
On $A_k$, we have
\[
PU_i\leq3PU_k,
\qquad
PU_j\leq3PU_k.
\]
Therefore,
\[
(PU_i)^{\frac{n+2}{n-2}}PU_j
\leq
C(PU_i)^{\frac{n+2}{n-2}}PU_k.
\]
Furthermore, since
\[
\frac4{n-2}-\varepsilon>0,
\]
we have on $A_k$
\[
\begin{aligned}
(PU_i)^{\frac{n+2}{n-2}}PU_k
&=
(PU_i)^{1+\varepsilon}
(PU_k)^{\frac{n+2}{n-2}-\varepsilon}
\left(
\frac{PU_i}{PU_k}
\right)^{\frac4{n-2}-\varepsilon}\leq
C
(PU_i)^{1+\varepsilon}
(PU_k)^{\frac{n+2}{n-2}-\varepsilon}.
\end{aligned}
\]
It follows from \eqref{eq:two-bubble-epsilon} that
\[
\begin{aligned}
\sum_{i\neq j}
\int_{\Omega\setminus A_i}
(PU_i)^{\frac{n+2}{n-2}}PU_j\,\ud x&\leq
C\sum_{i\neq j}\sum_{k\neq i}
\int_\Omega
(PU_i)^{1+\varepsilon}
(PU_k)^{\frac{n+2}{n-2}-\varepsilon}
\,\ud x\\
&\quad\leq
C\sum_{i\neq k}\Lambda_{i,k}^{1+\varepsilon}\\
&\quad=
o\left(\sum_{i\neq j}\Lambda_{i,j}\right).
\end{aligned}
\]
Here the additional sum over $j$ contributes only a constant
depending on $\ell$. This proves \eqref{eq:extend-cross-term}.

Combining \eqref{eq:power-expansion-partition},
\eqref{eq:extend-pure-term}, and \eqref{eq:extend-cross-term}, we
obtain
\[
\begin{aligned}
\int_\Omega
\left(
\sum_{i=1}^{\ell}\alpha_iPU_i
\right)^{\frac{2n}{n-2}}
\,\ud x
={}&
\sum_{i=1}^{\ell}
\alpha_i^{\frac{2n}{n-2}}
\int_\Omega
(PU_i)^{\frac{2n}{n-2}}\,\ud x\\
&+
\frac{2n}{n-2}
\sum_{i\neq j}
\alpha_i^{\frac{n+2}{n-2}}\alpha_j
\int_\Omega
(PU_i)^{\frac{n+2}{n-2}}PU_j\,\ud x\\
&+
o\left(\sum_{i\neq j}\Lambda_{i,j}\right).
\end{aligned}
\]
This proves \eqref{eq:sum-bubbles-power}.

We next compute the two types of integrals appearing in
\eqref{eq:sum-bubbles-power}. Since
\[
0\leq h_i\leq U_i
\qquad\mbox{and}\qquad
\|h_i\|_{L^\infty(\Omega)}
\leq
C\lambda_i^{-\frac{n-2}{2}},
\]
Taylor's formula gives
\[
\begin{aligned}
\int_\Omega(PU_i)^{\frac{2n}{n-2}}\,\ud x
={}&
\int_\Omega U_i^{\frac{2n}{n-2}}\,\ud x-
\frac{2n}{n-2}
\int_\Omega
U_i^{\frac{n+2}{n-2}}h_i\,\ud x\\
&+
O\left(
\int_\Omega
U_i^{\frac4{n-2}}h_i^2\,\ud x
+
\int_\Omega h_i^{\frac{2n}{n-2}}\,\ud x
\right).
\end{aligned}
\]
Furthermore,
\[
\int_\Omega
U_i^{\frac4{n-2}}h_i^2\,\ud x
\leq
\begin{cases}
C\lambda_i^{-2},&n=3,\\
C\lambda_i^{-4}\log\lambda_i,&n=4,\\
C\lambda_i^{-n},&n\geq5,
\end{cases}
\]
and
\[
\int_\Omega h_i^{\frac{2n}{n-2}}\,\ud x
=
O(\lambda_i^{-n}).
\]
Both terms are $o(\lambda_i^{2-n})$. Therefore,
\be\label{eq:PUi^p+1}
\begin{aligned}
\int_\Omega(PU_i)^{\frac{2n}{n-2}}\,\ud x
=
C_2
-
\frac{2n}{n-2}
C_1\lambda_i^{2-n}H(a_i,a_i)
+
o(\lambda_i^{2-n}).
\end{aligned}
\ee

We next estimate the interaction involving two projected bubbles.
Since
\[
PU_j=U_j-h_j>0
\qquad\mbox{in }\Omega,
\]
we have $0\leq h_j\leq U_j$. Hence, by the mean value theorem,
\be\label{eq:Uj-PUj-power}
0
\leq
U_j^{\frac{n+2}{n-2}}
-
(PU_j)^{\frac{n+2}{n-2}}
\leq
C U_j^{\frac4{n-2}}h_j.
\ee

We first give the details of the interaction estimate used below.
By H\"older's inequality and \eqref{eq:twodiffbubble},
\[
\begin{aligned}
\int_\Omega
U_iU_j^{\frac4{n-2}}\,\ud x
&=
\int_\Omega
U_i^{\frac{n-2}{n+2}}\left(
U_iU_j^{\frac{n+2}{n-2}}
\right)^{\frac4{n+2}}
\,\ud x\\
&\leq
\left(
\int_\Omega U_i\,\ud x
\right)^{\frac{n-2}{n+2}}
\left(
\int_\Omega
U_iU_j^{\frac{n+2}{n-2}}\,\ud x
\right)^{\frac4{n+2}}
\\
&\le C\lambda_i^{-\frac{(n-2)^2}{2(n+2)}} \Lambda_{i,j}^{\frac4{n+2}}.
\end{aligned}
\]
This yields
\be\label{eq:linear-projection-error}
\begin{aligned}
\int_\Omega
|U_iU_j^{\frac4{n-2}}h_j|\,\ud x
\leq{}&
C
\Lambda_{i,j}^{\frac4{n+2}}
\lambda_i^{-\frac{(n-2)^2}{2(n+2)}}
\lambda_j^{-\frac{n-2}{2}}=
o(\Lambda_{i,j})
+
o(\lambda_i^{2-n})
+
o(\lambda_j^{2-n}),
\end{aligned}
\ee
where we used Young's inequality in the last inequality. Then it follows from \eqref{eq:Uj-PUj-power} that
\[
\begin{aligned}
\int_\Omega
U_i(PU_j)^{\frac{n+2}{n-2}}\,\ud x
={}&
\int_\Omega
U_iU_j^{\frac{n+2}{n-2}}\,\ud x+
o(\Lambda_{i,j})
+
o(\lambda_i^{2-n})
+
o(\lambda_j^{2-n})\\
={}&C_1c_0\Lambda_{i,j}
+
o(\Lambda_{i,j})
+
o(\lambda_i^{2-n})
+
o(\lambda_j^{2-n}).
\end{aligned}
\]

By
\eqref{eq:Uj-PUj-power},
\[
\begin{aligned}
&
\left|
\int_\Omega
h_i(PU_j)^{\frac{n+2}{n-2}}\,\ud x
-
\int_\Omega
h_iU_j^{\frac{n+2}{n-2}}\,\ud x
\right|\leq
C\int_\Omega
h_iU_j^{\frac4{n-2}}h_j\,\ud x.
\end{aligned}
\]
Note
\[
\begin{aligned}
\int_\Omega U_j^{\frac4{n-2}}\,\ud x
&\leq
C\lambda_j^{2-n}
\int_0^{\lambda_j\operatorname{diam}(\Omega)}
\frac{s^{n-1}}{(1+s^2)^2}\,\ud s
\leq
C\lambda_j^{-1}.
\end{aligned}
\]
Indeed, the last integral is bounded if $n=3$, is
$O(\log\lambda_j)$ if $n=4$, and is $O(\lambda_j^{n-4})$ if
$n\geq5$. Thus,
\[
\begin{aligned}
\int_\Omega
|h_iU_j^{\frac4{n-2}}h_j|\,\ud x
&\leq
C\lambda_i^{-\frac{n-2}{2}}
\lambda_j^{-\frac{n-2}{2}}
\int_\Omega U_j^{\frac4{n-2}}\,\ud x\\
&\leq
C\lambda_i^{-\frac{n-2}{2}}\lambda_j^{-\frac n2}\\
&=
o(\lambda_i^{2-n})+o(\lambda_j^{2-n}).
\end{aligned}
\]
Combining this estimate with \eqref{eq:remainderonebubble}, we find
\[
\begin{aligned}
\int_\Omega
h_i(PU_j)^{\frac{n+2}{n-2}}\,\ud x
={}&
C_1(\lambda_i\lambda_j)^{-\frac{n-2}{2}}
H(a_i,a_j)+
o(\lambda_i^{2-n})
+
o(\lambda_j^{2-n}).
\end{aligned}
\]
Therefore,
\be\label{eq:PUiPUj^p}
\begin{aligned}
\int_\Omega
PU_i(PU_j)^{\frac{n+2}{n-2}}\,\ud x
={}&
C_1c_0\Lambda_{i,j}-
C_1(\lambda_i\lambda_j)^{-\frac{n-2}{2}}
H(a_i,a_j)\\
&+
o(\Lambda_{i,j})
+
o(\lambda_i^{2-n})
+
o(\lambda_j^{2-n}).
\end{aligned}
\ee

It follows from \eqref{eq:sum-bubbles-power},
\eqref{eq:PUi^p+1}, and \eqref{eq:PUiPUj^p} that
\be\label{eq:sigma22}
\begin{aligned}
\int_\Omega
\sigma^{\frac{2n}{n-2}}\,\ud x
=
r_\infty^{-\frac n2}
\Bigg\{&
C_2\sum_{i=1}^{\ell}
\alpha_i^{\frac{2n}{n-2}}\\
&-
\frac{2n}{n-2}C_1
\sum_{i=1}^{\ell}
\alpha_i^{\frac{2n}{n-2}}
\lambda_i^{2-n}H(a_i,a_i)\\
&+
\frac{2n}{n-2}C_1
\sum_{i\neq j}
\alpha_i\alpha_j^{\frac{n+2}{n-2}}
\left(
c_0\Lambda_{i,j}
-
(\lambda_i\lambda_j)^{-\frac{n-2}{2}}
H(a_i,a_j)
\right)\\
&+
o\left(
\sum_{i\neq j}\Lambda_{i,j}
+
\sum_{i=1}^{\ell}\lambda_i^{2-n}
\right)
\Bigg\}.
\end{aligned}
\ee

Let
\[
K_1=\sum_{i=1}^{\ell}\alpha_i^2,
\qquad
K_2=\sum_{i=1}^{\ell}
\alpha_i^{\frac{2n}{n-2}}.
\]
Expanding the quotient by means of \eqref{eq:sigma1} and
\eqref{eq:sigma22}, we obtain
\[
\begin{aligned}
Y_\Omega(\sigma)
={}&
C_2^{\frac2n}K_1K_2^{-\frac{n-2}{n}}\\
&+
C_1C_2^{-\frac{n-2}{n}}
K_2^{-\frac{n-2}{n}}
\Bigg[
\sum_{i=1}^{\ell}
\left(
2\frac{K_1}{K_2}
\alpha_i^{\frac{2n}{n-2}}
-
\alpha_i^2
\right)
\lambda_i^{2-n}H(a_i,a_i)\\
&\hspace{17mm}
+
\sum_{i\neq j}
\left(
\alpha_i\alpha_j
-
2\frac{K_1}{K_2}
\alpha_i\alpha_j^{\frac{n+2}{n-2}}
\right)
\left(
c_0\Lambda_{i,j}
-
(\lambda_i\lambda_j)^{-\frac{n-2}{2}}
H(a_i,a_j)
\right)
\Bigg]\\
&+
o\left(
\sum_{i\neq j}\Lambda_{i,j}
+
\sum_{i=1}^{\ell}\lambda_i^{2-n}
\right).
\end{aligned}
\]

A Taylor expansion at
$\alpha_1=\cdots=\alpha_\ell=1$ gives
\[
K_1K_2^{-\frac{n-2}{n}}
=
\ell^{\frac2n}
+
O\left(\sum_{i=1}^{\ell}|\alpha_i-1|^2\right),
\]
\[
K_2^{-\frac{n-2}{n}}
=
\ell^{-\frac{n-2}{n}}
+
O\left(\sum_{i=1}^{\ell}|\alpha_i-1|\right),
\]
\[
2\frac{K_1}{K_2}
\alpha_i^{\frac{2n}{n-2}}
-
\alpha_i^2
=
1+
O\left(\sum_{k=1}^{\ell}|\alpha_k-1|\right),
\]
and, for $i\neq j$,
\[
\alpha_i\alpha_j
-
2\frac{K_1}{K_2}
\alpha_i\alpha_j^{\frac{n+2}{n-2}}
=
-1+
O\left(\sum_{k=1}^{\ell}|\alpha_k-1|\right).
\]
Since $H$ is uniformly bounded when
$d(a_i),d(a_j)\geq\delta_0$, 
we conclude that
\[
\begin{aligned}
Y_\Omega(\sigma)
={}&
C_2^{\frac2n}\ell^{\frac2n}
-
C_1c_0C_2^{-\frac{n-2}{n}}
\ell^{-\frac{n-2}{n}}
\sum_{i\neq j}\Lambda_{i,j}\\
&+
O\left(\sum_{i=1}^{\ell}|\alpha_i-1|^2\right)
+
o\left(\sum_{i\neq j}\Lambda_{i,j}\right)
+
O\left(\sum_{i=1}^{\ell}\lambda_i^{2-n}\right).
\end{aligned}
\]
Finally,
\[
C_2^{\frac2n}\ell^{\frac2n}=r_\infty
\]
and
\[
C_2^{-\frac{n-2}{n}}
\ell^{-\frac{n-2}{n}}
=
r_\infty^{-\frac{n-2}{2}},
\]
which proves \eqref{eq:energy-sigma-alpha}.

\end{proof}


\section{Selection of parameters and preliminary energy identities}\label{sec:modulation}

Fix $\ell\geq0$ and $\delta>0$. For any $v\in H_0^1(\Omega)$
and $\va>0$, we define
\begin{align*}
\A_v(u_\infty,\ell,\delta;\va)
:=
\Bigg\{
(\bm z,\bm\alpha,\bm a,\bm\lambda):
\;&
\alpha_i>0,\quad
\lambda_i>0,\quad
a_i\in\Omega,
\qquad 1\leq i\leq\ell,
\\
&
|\bm z|
+
\sum_{i=1}^{\ell}|\alpha_i-1|
+
\sum_{i=1}^{\ell}\frac1{\lambda_i}
+
\sum_{i\neq j}\Lambda_{i,j}
<\va,
\\
&
\left\|
v-\xi_{\bm z}
-r_\infty^{-\frac{n-2}{4}}
\sum_{k=1}^{\ell}
\alpha_kPU_{a_k,\lambda_k}
\right\|
<\va,
\\
&
\operatorname{dist}(a_k,\partial\Omega)\geq\delta,
\qquad 1\leq k\leq\ell
\Bigg\}.
\end{align*}
We then set
\[
\V(u_\infty,\ell,\delta;\va)
:=
\left\{
v\in H_0^1(\Omega):
\A_v(u_\infty,\ell,\delta;\va)\neq\emptyset
\right\}.
\]
When $u_\infty\equiv0$, the $\bm z$-component is omitted and
$\xi_{\bm z}\equiv0$. When $\ell=0$, all the bubble parameters and
the corresponding sums are omitted.

We regard two parameter tuples that differ only by a simultaneous
permutation of
\[
(\alpha_i,a_i,\lambda_i),
\qquad 1\leq i\leq\ell,
\]
as representing the same configuration.

Let $u\in\mathcal D\cap\V(u_\infty,\ell,\delta;\va)$ satisfy $\|\mathcal R_u-r_\infty\|_{C^0(\Omega)}<\va$. We impose the normalization \eqref{eq:normalizedmass}.
By adapting Proposition 0.7 of Bahri \cite{Bahri} or
Proposition 3.10 of Mayer \cite{Mayer} to the present Dirichlet
setting, we obtain the following parameter-selection statement. There exists $\va_1>0$ such that whenever $u\in \mathcal{D}\cap\V(u_\infty,\ell,\delta;\va)$ and satisfies $\|\mathcal R_u-r_\infty\|_{C^0(\Omega)}<\va$ with $0<\va<\va_1$, the variational problem
\be\label{eq:vp}
\inf_{{(\bm z,\bm\alpha,\bm a,\bm\lambda)}
\in\A_u(u_\infty,\ell,\delta;2\va_0)}
\left\|
 u-\xi_{\bm z}
-r_\infty^{-\frac{n-2}{4}}
\sum_{k=1}^{\ell}\alpha_kPU_{a_k,\lambda_k}
\right\|_{\mathcal L_u(\Omega)}^2
\ee
admits a minimizer in
$\A_u(u_\infty,\ell,\delta;\va_0)$. The minimizer is unique
modulo permutations of the bubble indices.  After fixing a local
labeling, let $(\bm z,\bm\alpha,\bm a,\bm\lambda)$ denote this
minimizer and set
\be\label{eq:decomposition-1}
\sigma
=
\xi_{\bm z}
+r_\infty^{-\frac{n-2}{4}}
\sum_{k=1}^{\ell}\alpha_kPU_{a_k,\lambda_k},
\qquad
w=u-\sigma.
\ee

Since the weight $u^{\frac4{n-2}}$ is fixed in the variational
problem \eqref{eq:vp}, differentiation with respect to the
parameters gives
\be\label{eq:w-z-orthogonality}
\int_\Omega
u^{\frac4{n-2}}w
\frac{\partial\xi_{\bm z}}{\partial z_l}
\,\ud x
=0,
\qquad 1\leq l\leq L,
\ee
and
\be\label{eq:w-bubble-orthogonality}
\begin{aligned}
\int_\Omega
u^{\frac4{n-2}}w\,PU_i\,\ud x
&=0,
\\
\int_\Omega
u^{\frac4{n-2}}w\,
\lambda_i\frac{\partial PU_i}{\partial\lambda_i}
\,\ud x
&=0,
\\
\int_\Omega
u^{\frac4{n-2}}w\,
\frac1{\lambda_i}
\frac{\partial PU_i}{\partial a_i^m}
\,\ud x
&=0,
\qquad 1\leq m\leq n,
\end{aligned}
\qquad 1\leq i\leq\ell.
\ee
Here \eqref{eq:w-z-orthogonality} is omitted when
$u_\infty\equiv0$, and \eqref{eq:w-bubble-orthogonality} is omitted
when $\ell=0$.

The remainder $w$ lies in the directions on which the second
variation is coercive.

\begin{prop}[Coercivity]\label{prop:Coercive-1}
After decreasing $\va_0$ if necessary, there exists a constant
$c>0$, depending only on $n$, $\Omega$, $\ell$, $\delta$,
$u_\infty$ and  $r_\infty$,
such that, for every $\va\in(0,\va_0)$ and for every decomposition \eqref{eq:decomposition-1}
obtained above,
\[
\frac{n+2}{n-2}r_\infty
\int_\Omega
\sigma^{\frac4{n-2}}w^2\,\ud x
\leq
(1-c)\|w\|^2,
\]
and
\[
\frac{n+2}{n-2}r_\infty
\|w\|_{\mathcal L_u(\Omega)}^2
\leq
(1-c)\|w\|^2.
\]
\end{prop}

\begin{proof}
For $u_\infty=0$, refer to the proofs of Corollary 5.5 in \cite{Br05}, Proposition 4.5 in \cite{Mayer}, and Lemma 4.3 in \cite{JX20-a}; for $u_\infty>0$, consult Corollary 6.10 in \cite{Br05}, Proposition 5.5 in \cite{Mayer}, and Corollary 4.13 in \cite{JX20-a}.
\end{proof}

\subsection{Auxiliary estimates and preliminary identities}

The following elementary estimates will be used repeatedly. For
$a>0$, $a+b\geq0$, and
$\delta(n)=\min\{1,4/(n-2)\}$,
\[
\begin{aligned}
&\left|
(a+b)^{\frac{n+2}{n-2}}
-a^{\frac{n+2}{n-2}}
-\frac{n+2}{n-2}a^{\frac4{n-2}}b
\right|\leq
C a^{\frac4{n-2}-\delta(n)}|b|^{1+\delta(n)}
+C|b|^{\frac{n+2}{n-2}}.
\end{aligned}
\]
Also, by H\"older's inequality,
\[
\begin{aligned}
&\left|
\int_\Omega
(\mathcal R_u-r_\infty)
 u^{\frac{n+2}{n-2}}\psi\,\ud x
\right|\leq
\left(
\int_\Omega
|\mathcal R_u-r_\infty|^{\frac{2n}{n+2}}
 u^{\frac{2n}{n-2}}\,\ud x
\right)^{\frac{n+2}{2n}}
\|\psi\|_{L^{\frac{2n}{n-2}}(\Omega)}.
\end{aligned}
\]
Finally, the normalization of $u=\sigma+w$ and the triangle inequality imply $$1-\|w\|_{L^{\frac{2n}{n-2}}(\Omega)}\le \|\sigma\|_{L^{\frac{2n}{n-2}}(\Omega)}\le 1+ \|w\|_{L^{\frac{2n}{n-2}}(\Omega)}.$$ Hence,
\be\label{eq:mass-comparison}
\left|
\int_\Omega\sigma^{\frac{2n}{n-2}}\,\ud x-1
\right|
\leq
C\|w\|_{L^{\frac{2n}{n-2}}(\Omega)},
\ee
since $\|w\|$ is small. 
Consequently,
\be\label{eq:mass-factor-expansion}
\begin{aligned}
&\left(
\int_\Omega\sigma^{\frac{2n}{n-2}}\,\ud x
\right)^{\frac{n-2}{n}}-1=
\frac{n-2}{n}
\left(
\int_\Omega\sigma^{\frac{2n}{n-2}}\,\ud x-1
\right)
+O\left(\|w\|_{L^{\frac{2n}{n-2}}(\Omega)}^2\right).
\end{aligned}
\ee

We shall now expand the energy $Y_\Omega(u)$. Since
\[
\mathcal R_u
=
-u^{-\frac{n+2}{n-2}}\Delta u\quad
\mbox{and}\quad
\int_\Omega u^{\frac{2n}{n-2}}\,\ud x=1,
\]
we have
\[
Y_\Omega(u)
=
\int_\Omega|\nabla u|^2\,\ud x
=
\int_\Omega
\mathcal R_u u^{\frac{2n}{n-2}}\,\ud x.
\]
Moreover, since $u=\sigma+w$,
\[
\begin{aligned}
\int_\Omega|\nabla u|^2\,\ud x
&=
\int_\Omega|\nabla\sigma|^2\,\ud x
+
2\int_\Omega\nabla\sigma\cdot\nabla w\,\ud x
+
\int_\Omega|\nabla w|^2\,\ud x\\
&=
\int_\Omega|\nabla\sigma|^2\,\ud x
+
2\int_\Omega
\mathcal R_u u^{\frac{n+2}{n-2}}w\,\ud x
-
\int_\Omega|\nabla w|^2\,\ud x.
\end{aligned}
\]
Consequently,
\[
\begin{aligned}
\int_\Omega
\mathcal R_u u^{\frac{2n}{n-2}}\,\ud x
={}&
Y_\Omega(\sigma)
\left(
\int_\Omega
\sigma^{\frac{2n}{n-2}}\,\ud x
\right)^{\frac{n-2}{n}}+
2\int_\Omega
(\mathcal R_u-r_\infty)
u^{\frac{n+2}{n-2}}w\,\ud x\\
&-
\int_\Omega
\left(
|\nabla w|^2
-
\frac{n+2}{n-2}r_\infty
\sigma^{\frac4{n-2}}w^2
\right)\,\ud x\\
&+
r_\infty
\int_\Omega
\left[
-\frac{n+2}{n-2}
\sigma^{\frac4{n-2}}w^2
+
2(\sigma+w)^{\frac{n+2}{n-2}}w
\right]\,\ud x.
\end{aligned}
\]

We next expand the normalization factor. By \eqref{eq:mass-factor-expansion},
\[
\begin{aligned}
&
\left(
\int_\Omega
\sigma^{\frac{2n}{n-2}}\,\ud x
\right)^{\frac{n-2}{n}}
-1=
\frac{n-2}{n}
\left(
\int_\Omega
\sigma^{\frac{2n}{n-2}}\,\ud x
-1
\right)
+
O\left(
\|w\|_{L^{\frac{2n}{n-2}}(\Omega)}^2
\right).
\end{aligned}
\]
Using
\[
1
=
\int_\Omega
(\sigma+w)^{\frac{2n}{n-2}}\,\ud x,
\]
we equivalently have
\[
\begin{aligned}
-1
={}&
\int_\Omega
\left[
\frac{n-2}{n}
\sigma^{\frac{2n}{n-2}}
-
\frac{n-2}{n}
(\sigma+w)^{\frac{2n}{n-2}}
\right]\,\ud x-
\left(
\int_\Omega
\sigma^{\frac{2n}{n-2}}\,\ud x
\right)^{\frac{n-2}{n}}
+
O\left(
\|w\|_{L^{\frac{2n}{n-2}}(\Omega)}^2
\right).
\end{aligned}
\]

Substituting this identity into the preceding expansion, we obtain
\be\label{eq:Ruup0}
\begin{aligned}
&
\int_\Omega
\mathcal R_u u^{\frac{2n}{n-2}}\,\ud x\\
={}&
r_\infty
+
\big(Y_\Omega(\sigma)-r_\infty\big)
\left(
\int_\Omega
\sigma^{\frac{2n}{n-2}}\,\ud x
\right)^{\frac{n-2}{n}}+
2\int_\Omega
(\mathcal R_u-r_\infty)
u^{\frac{n+2}{n-2}}w\,\ud x\\
&-
\int_\Omega
\left(
|\nabla w|^2
-
\frac{n+2}{n-2}r_\infty
\sigma^{\frac4{n-2}}w^2
\right)\,\ud x\\
&+
r_\infty
\int_\Omega
\Bigg[
\frac{n-2}{n}
\sigma^{\frac{2n}{n-2}}
-
\frac{n+2}{n-2}
\sigma^{\frac4{n-2}}w^2
+
2(\sigma+w)^{\frac{n+2}{n-2}}w
-
\frac{n-2}{n}
(\sigma+w)^{\frac{2n}{n-2}}
\Bigg]\,\ud x\\
&+
O\left(
\|w\|_{L^{\frac{2n}{n-2}}(\Omega)}^2
\right).
\end{aligned}
\ee

By H\"older's inequality,
\be\label{eq:Ruup1}
\begin{aligned}
&
\left|
\int_\Omega
(\mathcal R_u-r_\infty)
u^{\frac{n+2}{n-2}}w\,\ud x
\right|\leq
\left(
\int_\Omega
|\mathcal R_u-r_\infty|^{\frac{2n}{n+2}}
u^{\frac{2n}{n-2}}\,\ud x
\right)^{\frac{n+2}{2n}}
\|w\|_{L^{\frac{2n}{n-2}}(\Omega)}.
\end{aligned}
\ee
Moreover,
\[
\begin{aligned}
&
\left|
\int_\Omega
\left(
|\nabla w|^2
-
\frac{n+2}{n-2}r_\infty
\sigma^{\frac4{n-2}}w^2
\right)\,\ud x
\right|\leq
C\|w\|^2.
\end{aligned}
\]

We next estimate the nonlinear remainder. For every $\sigma\geq0$
and $w\in\mathbb R$ satisfying $\sigma+w\geq0$, one has from Taylor expansion that 
\[
\begin{aligned}
&
\Bigg|
\frac{n-2}{n}
\sigma^{\frac{2n}{n-2}}
-
\frac{n+2}{n-2}
\sigma^{\frac4{n-2}}w^2
+
2(\sigma+w)^{\frac{n+2}{n-2}}w
-
\frac{n-2}{n}
(\sigma+w)^{\frac{2n}{n-2}}
\Bigg|\\
&\quad\leq
C
\sigma^{\max\{0,\frac{2n}{n-2}-3\}}
|w|^{\min\{\frac{2n}{n-2},3\}}
+
C|w|^{\frac{2n}{n-2}}.
\end{aligned}
\]
Indeed, on the set where $|w|\leq\frac12\sigma$, this follows by
Taylor expansion: the terms of orders zero, one, and two cancel
exactly. On the complementary set, $\sigma\leq2|w|$, and the
right-hand side follows from the elementary bound
\[
(\sigma+|w|)^{\frac{2n}{n-2}}
\leq C|w|^{\frac{2n}{n-2}}.
\]
Consequently, by H\"older's inequality and the uniform boundedness
of $
\int_\Omega
\sigma^{\frac{2n}{n-2}}\,\ud x$,
we obtain
\be\label{eq:Ruup2}
\begin{aligned}
&
\int_\Omega
\Bigg|
\frac{n-2}{n}
\sigma^{\frac{2n}{n-2}}
-
\frac{n+2}{n-2}
\sigma^{\frac4{n-2}}w^2
+
2(\sigma+w)^{\frac{n+2}{n-2}}w
-
\frac{n-2}{n}
(\sigma+w)^{\frac{2n}{n-2}}
\Bigg|\,\ud x\\
&\quad\leq
C
\left(
\int_\Omega
|w|^{\frac{2n}{n-2}}\,\ud x
\right)^{
\frac{n-2}{n}
\min\{\frac{n}{n-2},\frac32\}
}\\
&\quad=C\|w\|^{\min\left\{
\frac{2n}{n-2},3
\right\}}\\
&\quad\le C \|w\|^2,
\end{aligned}
\ee
since
$
\min\left\{
\frac{2n}{n-2},3
\right\}>2.
$

Combining \eqref{eq:Ruup0}--\eqref{eq:Ruup2} and using Young's
inequality, we obtain
\be\label{eq:Ruup}
\begin{aligned}
Y_\Omega(u)
={}&
r_\infty
+
\big(Y_\Omega(\sigma)-r_\infty\big)
\left(
\int_\Omega
\sigma^{\frac{2n}{n-2}}\,\ud x
\right)^{\frac{n-2}{n}}\\
&+
O\left(
\left(
\int_\Omega
|\mathcal R_u-r_\infty|^{\frac{2n}{n+2}}
u^{\frac{2n}{n-2}}\,\ud x
\right)^{\frac{n+2}{n}}
\right)
+
O(\|w\|^2).
\end{aligned}
\ee

We also need a one-sided version of \eqref{eq:Ruup}. Since the
function $t\mapsto t^{\frac{n-2}{n}}$ is concave on $(0,\infty)$,
\[
\left(
\int_\Omega
\sigma^{\frac{2n}{n-2}}\,\ud x
\right)^{\frac{n-2}{n}}
-1
\leq
\frac{n-2}{n}
\left(
\int_\Omega
\sigma^{\frac{2n}{n-2}}\,\ud x
-1
\right).
\]
Thus,
\[
\begin{aligned}
-1
\leq{}&
\int_\Omega
\left[
\frac{n-2}{n}
\sigma^{\frac{2n}{n-2}}
-
\frac{n-2}{n}
(\sigma+w)^{\frac{2n}{n-2}}
\right]\,\ud x-
\left(
\int_\Omega
\sigma^{\frac{2n}{n-2}}\,\ud x
\right)^{\frac{n-2}{n}}.
\end{aligned}
\]
Arguing as in the derivation of \eqref{eq:Ruup0}, we obtain
\[
\begin{aligned}
Y_\Omega(u)
\leq{}&
r_\infty
+
\big(Y_\Omega(\sigma)-r_\infty\big)
\left(
\int_\Omega
\sigma^{\frac{2n}{n-2}}\,\ud x
\right)^{\frac{n-2}{n}}\\
&+
2\int_\Omega
(\mathcal R_u-r_\infty)
u^{\frac{n+2}{n-2}}w\,\ud x\\
&-
\int_\Omega
\left(
|\nabla w|^2
-
\frac{n+2}{n-2}r_\infty
\sigma^{\frac4{n-2}}w^2
\right)\,\ud x\\
&+
r_\infty
\int_\Omega
\Bigg[
\frac{n-2}{n}
\sigma^{\frac{2n}{n-2}}
-
\frac{n+2}{n-2}
\sigma^{\frac4{n-2}}w^2
+
2(\sigma+w)^{\frac{n+2}{n-2}}w
-
\frac{n-2}{n}
(\sigma+w)^{\frac{2n}{n-2}}
\Bigg]\,\ud x.
\end{aligned}
\]
By Proposition \ref{prop:Coercive-1},
\[
\int_\Omega
\left(
|\nabla w|^2
-
\frac{n+2}{n-2}r_\infty
\sigma^{\frac4{n-2}}w^2
\right)\,\ud x
\geq
c\|w\|^2.
\]
Using \eqref{eq:Ruup1}, \eqref{eq:Ruup2}, the Sobolev inequality,
and Young's inequality, and then decreasing $\va_0$ if necessary,
we conclude that
\be\label{eq:Ruup<}
\begin{aligned}
Y_\Omega(u)
\leq{}&
r_\infty
+
\big(Y_\Omega(\sigma)-r_\infty\big)
\left(
\int_\Omega
\sigma^{\frac{2n}{n-2}}\,\ud x
\right)^{\frac{n-2}{n}}\\
&+
C\left(
\int_\Omega
|\mathcal R_u-r_\infty|^{\frac{2n}{n+2}}
u^{\frac{2n}{n-2}}\,\ud x
\right)^{\frac{n+2}{n}}
-
c\|w\|^2.
\end{aligned}
\ee

The identities \eqref{eq:Ruup} and \eqref{eq:Ruup<} are
preliminary: at this stage no final estimate for the modulation
parameters has been used. Their consequences for the energy gap
will be derived after Theorem \ref{thm:LS} has been proved.

\subsection{Estimates for the regular component}

We first need estimates for $u-\xi_{\bm z}$ in
$L^{\frac{n+2}{n-2}}(\Omega)$ and $L^1(\Omega)$. The following
lemma is obtained by adapting the proofs of Lemmas 4.14 and 4.15 of
\cite{JX20-a} to the present setting.

\begin{lem}\label{lem:n+2norm}
There exist constants $C>0$ and $\va_1\in(0,1)$ such that, for every $\va\in(0,\va_1)$, and $u\in \mathcal{D}\cap\V(u_\infty, \ell,\delta; \va)$ with $\|\mathcal{R}_{u}-r_\infty\|_{C^0(\om)}<\va$,  there holds
\[
\begin{aligned}
\|u-\xi_{\bm z}\|_{L^{\frac{n+2}{n-2}}(\Omega)}
 ^{\frac{n+2}{n-2}}
\leq{}&
C
\left\|
(\mathcal R_u-r_\infty)u^{\frac{n+2}{n-2}}
\right\|_{L^{\frac{2n}{n+2}}(\Omega)}
 ^{\frac{n+2}{n-2}}+
C\sum_{k=1}^{\ell}
\lambda_k^{-\frac{n-2}{2}},
\end{aligned}
\]
and
\[
\begin{aligned}
\|u-\xi_{\bm z}\|_{L^1(\Omega)}
\leq{}&
C
\left\|
(\mathcal R_u-r_\infty)u^{\frac{n+2}{n-2}}
\right\|_{L^{\frac{2n}{n+2}}(\Omega)}+
C\sum_{k=1}^{\ell}
\lambda_k^{-\frac{n-2}{2}}.
\end{aligned}
\]
\end{lem}

\begin{prop}\label{prop:regular-part}
There exist constants $C>0$ and $\va_1\in(0,1)$  such that, for every
$\va\in(0,\va_1)$, and $u\in \mathcal{D}\cap\V(u_\infty, \ell,\delta; \va)$ with $\|\mathcal{R}_{u}-r_\infty\|_{C^0(\om)}<\va$,  there holds
\[
\begin{aligned}
\left|
Y_\Omega(\xi_{\bm z})-Y_\Omega(u_\infty)
\right|
\leq{}&
C
\left(
\int_\Omega
|\mathcal R_u-r_\infty|^{\frac{2n}{n+2}}
u^{\frac{2n}{n-2}}\,\ud x
\right)^{\frac{n+2}{2n}(1+\gamma)}+
C\sum_{k=1}^{\ell}
\lambda_k^{-\frac{n-2}{2}(1+\gamma)},
\end{aligned}
\]
where $\gamma\in(0,1)$ is the exponent in
Lemma \ref{lem:L-ineq}.
\end{prop}

\begin{proof}
For every $1\leq l\leq L$, integration by parts and
\eqref{eq:phil} give
\[
\begin{aligned}
&
\int_\Omega
\left(
\Delta\xi_{\bm z}
+
r_\infty
\xi_{\bm z}^{\frac{n+2}{n-2}}
\right)\phi_l\,\ud x\\
={}&
\int_\Omega
\left(
\Delta u+r_\infty u^{\frac{n+2}{n-2}}
\right)\phi_l\,\ud x+
\mu_l
\int_\Omega
u_\infty^{\frac4{n-2}}
\phi_l(u-\xi_{\bm z})\,\ud x-
r_\infty
\int_\Omega
\phi_l
\left(
u^{\frac{n+2}{n-2}}
-
\xi_{\bm z}^{\frac{n+2}{n-2}}
\right)\,\ud x\\
={}&
\int_\Omega
(r_\infty-\mathcal R_u)
u^{\frac{n+2}{n-2}}\phi_l\,\ud x+
\mu_l
\int_\Omega
u_\infty^{\frac4{n-2}}
\phi_l(u-\xi_{\bm z})\,\ud x-
r_\infty
\int_\Omega
\phi_l
\left(
u^{\frac{n+2}{n-2}}
-
\xi_{\bm z}^{\frac{n+2}{n-2}}
\right)\,\ud x.
\end{aligned}
\]
Since $u,\xi_{\bm z}\geq0$, we have
\[
\begin{aligned}
\left|
u^{\frac{n+2}{n-2}}
-
\xi_{\bm z}^{\frac{n+2}{n-2}}
\right|
\leq{}&
C\xi_{\bm z}^{\frac4{n-2}}
|u-\xi_{\bm z}|+
C|u-\xi_{\bm z}|^{\frac{n+2}{n-2}}.
\end{aligned}
\]
The functions $\phi_l$, $u_\infty^{\frac4{n-2}}\phi_l$, and
$\xi_{\bm z}^{\frac4{n-2}}\phi_l$ are uniformly bounded for
$|\bm z|$ sufficiently small. Therefore,
\be\label{eq:xi_zphi_l}
\begin{aligned}
&
\sup_{1\leq l\leq L}
\left|
\int_\Omega
\left(
\Delta\xi_{\bm z}
+
r_\infty
\xi_{\bm z}^{\frac{n+2}{n-2}}
\right)\phi_l\,\ud x
\right|\\
&\quad\leq
C
\left\|
(r_\infty-\mathcal R_u)
u^{\frac{n+2}{n-2}}
\right\|_{L^{\frac{2n}{n+2}}(\Omega)}+
C\|u-\xi_{\bm z}\|_{L^1(\Omega)}
+
C
\|u-\xi_{\bm z}\|_{L^{\frac{n+2}{n-2}}(\Omega)}
 ^{\frac{n+2}{n-2}}\\
&\quad\leq
C
\left\|
(r_\infty-\mathcal R_u)
u^{\frac{n+2}{n-2}}
\right\|_{L^{\frac{2n}{n+2}}(\Omega)}+
C
\left\|
(\mathcal R_u-r_\infty)u^{\frac{n+2}{n-2}}
\right\|_{L^{\frac{2n}{n+2}}(\Omega)}
 ^{\frac{n+2}{n-2}}+
C\sum_{k=1}^{\ell}
\lambda_k^{-\frac{n-2}{2}}\\
&\quad\leq
C
\left\|
(r_\infty-\mathcal R_u)
u^{\frac{n+2}{n-2}}
\right\|_{L^{\frac{2n}{n+2}}(\Omega)}+
C\sum_{k=1}^{\ell}
\lambda_k^{-\frac{n-2}{2}},
\end{aligned}
\ee
where we used Lemma \ref{lem:n+2norm} in the second inequality, and 
$\|\mathcal R_u-r_\infty\|_{C^0(\Omega)}<\va$
and
$
\int_\Omega u^{\frac{2n}{n-2}}\,\ud x=1
$
in the last inequality. 

The conclusion now follows from Lemma \ref{lem:L-ineq}.
\end{proof}

\begin{lem}\label{lem:regular-energy-balance}
Suppose that $u_\infty>0$. Then
\[
\begin{aligned}
&\left|
\int_\Omega|\nabla\xi_{\bm z}|^2\,\ud x
-r_\infty
\int_\Omega\xi_{\bm z}^{\frac{2n}{n-2}}\,\ud x
\right|\leq
C\left(
\int_\Omega
|\mathcal R_u-r_\infty|^{\frac{2n}{n+2}}
 u^{\frac{2n}{n-2}}\,\ud x
\right)^{\frac{n+2}{2n}}
+C\sum_{k=1}^{\ell}\lambda_k^{-\frac{n-2}{2}}.
\end{aligned}
\]
\end{lem}

\begin{proof}
By \eqref{eq:Pif} and \eqref{eq:projection0},
\[
\Delta\xi_{\bm z}
+r_\infty\xi_{\bm z}^{\frac{n+2}{n-2}}
=
\sum_{l=1}^{L}
\left(
\int_\Omega
\left(
\Delta\xi_{\bm z}
+r_\infty\xi_{\bm z}^{\frac{n+2}{n-2}}
\right)\phi_l\,\ud x
\right)
 u_\infty^{\frac4{n-2}}\phi_l.
\]
Multiplying by $\xi_{\bm z}$, integrating, and using
\eqref{eq:xi_zphi_l}, together with the uniform boundedness of
\[
\int_\Omega u_\infty^{\frac4{n-2}}\phi_l\xi_{\bm z}\,\ud x,
\qquad 1\leq l\leq L,
\]
gives the conclusion.
\end{proof}

\begin{lem}\label{lem:regular-surface-expansions}
As $|\bm z|\to0$, one has
\be\label{eq:nablaxiz}
\begin{aligned}
\int_\Omega|\nabla\xi_{\bm z}|^2\,\ud x
={}&
\int_\Omega|\nabla u_\infty|^2\,\ud x
+2r_\infty z_1
\left(
\int_\Omega u_\infty^{\frac{2n}{n-2}}\,\ud x
\right)^{1/2}+
\sum_{i=1}^{L}\mu_i z_i^2
+o(|\bm z|^2),
\end{aligned}
\ee
and
\be\label{eq:xiz^p}
\begin{aligned}
\int_\Omega\xi_{\bm z}^{\frac{2n}{n-2}}\,\ud x
={}&
\int_\Omega u_\infty^{\frac{2n}{n-2}}\,\ud x
+\frac{2n}{n-2}z_1
\left(
\int_\Omega u_\infty^{\frac{2n}{n-2}}\,\ud x
\right)^{1/2}+
\frac{n(n+2)}{(n-2)^2}
\sum_{i=1}^{L}z_i^2
+o(|\bm z|^2).
\end{aligned}
\ee
\end{lem}

\begin{proof}
By \eqref{eq:xiz},
\[
\xi_{\bm z}
=u_\infty+
\sum_{i=1}^{L}z_i\phi_i+h_{\bm z},
\qquad
\|h_{\bm z}\|=O(|\bm z|^2).
\]
The orthogonality of $h_{\bm z}$, together with
\eqref{eq:phiij}, gives
\[
\int_\Omega
\left|
\nabla\left(\sum_{i=1}^{L}z_i\phi_i\right)
\right|^2\,\ud x
=
\sum_{i=1}^{L}\mu_i z_i^2
\]
and
\[
\int_\Omega
\nabla u_\infty\cdot
\nabla\left(\sum_{i=1}^{L}z_i\phi_i\right)\,\ud x
=
r_\infty z_1
\left(
\int_\Omega u_\infty^{\frac{2n}{n-2}}\,\ud x
\right)^{1/2}.
\]
This proves \eqref{eq:nablaxiz}. The expansion
\eqref{eq:xiz^p} follows by Taylor expansion, using
$\xi_{\bm z}/u_\infty$ uniformly bounded above and below,
\eqref{eq:phiij}, and for $i=2, \cdots,L$ that
\[
\int_\Omega u_\infty^{\frac{n+2}{n-2}}\phi_i\,\ud x=\int_\Omega u_\infty^{\frac{n+2}{n-2}}h_{\bm z}\,\ud x=0.
\]
\end{proof}

\begin{prop}\label{pro:Yom(u)-1}
Suppose that $u_\infty>0$ in $\Omega$ and that $\sigma$ is given by
\eqref{eq:decomposition-1}. There exist $\va_0\in(0,1)$ and $c>0$,
depending only on $\Omega$, $n$, $\ell$, $\delta_0$, and
$u_\infty$, such that, if
\[
d(a_i)\geq\delta_0,
\qquad 1\leq i\leq\ell,
\]
and
\[
|\bm z|
+
\sum_{i=1}^{\ell}|\alpha_i-1|
+
\sum_{i=1}^{\ell}\lambda_i^{-1}
+
\sum_{i\neq j}\Lambda_{i,j}
\leq\va_0,
\]
then
\[
\begin{aligned}
Y_\Omega(\sigma)
&+
c\sum_{k=1}^{\ell}
\lambda_k^{-\frac{n-2}{2}}
+
c\sum_{i<j}\Lambda_{i,j}\leq
\left(
Y_\Omega(\xi_{\bm z})^{\frac n2}
+
\sum_{k=1}^{\ell}
Y_\Omega(PU_k)^{\frac n2}
\right)^{\frac2n}.
\end{aligned}
\]
\end{prop}

\begin{proof}
For convenience, define
\[
F(\xi)
:=
\frac{\displaystyle\int_\Omega|\nabla\xi|^2\,\ud x}
{\displaystyle\int_\Omega|\xi|^{\frac{2n}{n-2}}\,\ud x},
\qquad
\xi\in H_0^1(\Omega)\setminus\{0\}.
\]
Then
\[
Y_\Omega(\xi)^{\frac n2}
=
\int_\Omega
F(\xi)^{\frac n2}
|\xi|^{\frac{2n}{n-2}}\,\ud x.
\]

By \eqref{eq:decomposition-1},
\[
\begin{aligned}
\int_\Omega|\nabla\sigma|^2\,\ud x
={}&
\int_\Omega|\nabla\xi_{\bm z}|^2\,\ud x+
2r_\infty^{-\frac{n-2}{4}}
\sum_{k=1}^{\ell}\alpha_k
\int_\Omega
\nabla\xi_{\bm z}\cdot\nabla PU_k\,\ud x\\
&+
r_\infty^{-\frac{n-2}{2}}
\sum_{k=1}^{\ell}\alpha_k^2
\int_\Omega|\nabla PU_k|^2\,\ud x+
2r_\infty^{-\frac{n-2}{2}}
\sum_{i<j}\alpha_i\alpha_j
\int_\Omega\nabla PU_i\cdot\nabla PU_j\,\ud x.
\end{aligned}
\]
Therefore,
\be\label{eq:2.11-1}
\begin{aligned}
&
Y_\Omega(\sigma)
\left(
\int_\Omega
\sigma^{\frac{2n}{n-2}}\,\ud x
\right)^{\frac{n-2}{n}}\\
&=
\int_\Omega
F(\xi_{\bm z})
\xi_{\bm z}^{\frac{2n}{n-2}}\,\ud x+
r_\infty^{-\frac{n-2}{2}}
\sum_{k=1}^{\ell}\alpha_k^2
\int_\Omega
F(PU_k)(PU_k)^{\frac{2n}{n-2}}\,\ud x\\
&\quad-
2r_\infty^{-\frac{n-2}{4}}
\sum_{k=1}^{\ell}\alpha_k
\int_\Omega
PU_k\Delta\xi_{\bm z}\,\ud x-
2r_\infty^{-\frac{n-2}{2}}
\sum_{i<j}\alpha_i\alpha_j
\int_\Omega
PU_i\Delta PU_j\,\ud x.
\end{aligned}
\ee

On the other hand, H\"older's inequality gives
\be\label{eq:2.11-3}
\begin{aligned}
&
\left(
Y_\Omega(\xi_{\bm z})^{\frac n2}
+
\sum_{k=1}^{\ell}
Y_\Omega(PU_k)^{\frac n2}
\right)^{\frac2n}
\left(
\int_\Omega
\sigma^{\frac{2n}{n-2}}\,\ud x
\right)^{\frac{n-2}{n}}\\
&=
\left(
\int_\Omega
\left[
F(\xi_{\bm z})^{\frac n2}
\xi_{\bm z}^{\frac{2n}{n-2}}
+
\sum_{k=1}^{\ell}
F(PU_k)^{\frac n2}
(PU_k)^{\frac{2n}{n-2}}
\right]\,\ud x
\right)^{\frac2n}
\left(
\int_\Omega
\sigma^{\frac{2n}{n-2}}\,\ud x
\right)^{\frac{n-2}{n}}\\
&\quad\geq
\int_\Omega
\left[
F(\xi_{\bm z})^{\frac n2}
\xi_{\bm z}^{\frac{2n}{n-2}}
+
\sum_{k=1}^{\ell}
F(PU_k)^{\frac n2}
(PU_k)^{\frac{2n}{n-2}}
\right]^{\frac2n}
\sigma^2\,\ud x.
\end{aligned}
\ee

Expanding $\sigma^2$ and retaining the corresponding nonnegative
terms, we obtain
\[
\begin{aligned}
&\left(
Y_\Omega(\xi_{\bm z})^{\frac n2}
+
\sum_{k=1}^{\ell}
Y_\Omega(PU_k)^{\frac n2}
\right)^{\frac2n}
\left(
\int_\Omega
\sigma^{\frac{2n}{n-2}}\,\ud x
\right)^{\frac{n-2}{n}}\\
&\geq
\int_\Omega
F(\xi_{\bm z})
\xi_{\bm z}^{\frac{2n}{n-2}}\,\ud x\\
&\quad+
r_\infty^{-\frac{n-2}{2}}
\sum_{k=1}^{\ell}\alpha_k^2
\int_\Omega
F(PU_k)(PU_k)^{\frac{2n}{n-2}}\,\ud x\\
&\quad+
2r_\infty^{-\frac{n-2}{4}}
\sum_{k=1}^{\ell}\alpha_k
\int_\Omega
\left[
F(\xi_{\bm z})^{\frac n2}
\xi_{\bm z}^{\frac{2n}{n-2}}
+
F(PU_k)^{\frac n2}
(PU_k)^{\frac{2n}{n-2}}
\right]^{\frac2n}
\xi_{\bm z}PU_k\,\ud x\\
&\quad+
2r_\infty^{-\frac{n-2}{2}}
\sum_{i<j}\alpha_i\alpha_j
\int_\Omega
\left[
F(PU_i)^{\frac n2}
(PU_i)^{\frac{2n}{n-2}}
+
F(PU_j)^{\frac n2}
(PU_j)^{\frac{2n}{n-2}}
\right]^{\frac2n}
PU_iPU_j\,\ud x.
\end{aligned}
\]

We next obtain quantitative lower bounds for the last two terms.

Set
\[
\Omega_{\delta_0}
:=
\left\{
x\in\Omega:
\operatorname{dist}(x,\partial\Omega)\geq\delta_0
\right\}.
\]
We first claim that, for $\lambda$ sufficiently large,
\be\label{eq:PU-lower-interior}
PU_{a,\lambda}(x)
\geq
cU_{a,\lambda}(x)
\end{equation}
whenever
$
a\in\Omega_{\delta_0},
$ and $
x\in\Omega_{\delta_0/2}.
$
Indeed, by the Green representation formula,
\[
PU_{a,\lambda}(x)
=
\int_\Omega
G(x,y)
U_{a,\lambda}(y)^{\frac{n+2}{n-2}}
\,\ud y,
\]
where $G(x,y)$ denotes Green's function of \(-\Delta\) with Dirichlet boundary conditions in \(\Omega\). For $x,y\in\Omega_{\delta_0/2}$, one has
\[
G(x,y)\geq c|x-y|^{2-n}.
\]
Moreover,
$
B_{\lambda^{-1}}(a)
\subset\Omega_{\delta_0/2}
$
for $\lambda$ sufficiently large. Thus, setting
$
z=\lambda(x-a)
$ and 
$\eta=\lambda(y-a)$,
we obtain
\[
\begin{aligned}
PU_{a,\lambda}(x)
&\geq
c\int_{B_{\lambda^{-1}}(a)}
|x-y|^{2-n}
U_{a,\lambda}(y)^{\frac{n+2}{n-2}}
\,\ud y\\
&\geq
c\lambda^{\frac{n-2}{2}}
\int_{B_1(0)}
\frac{|z-\eta|^{2-n}}
{(1+|\eta|^2)^{\frac{n+2}{2}}}
\,\ud\eta\\
&\geq
c\lambda^{\frac{n-2}{2}}
(1+|z|^2)^{-\frac{n-2}{2}}\\
&\geq
cU_{a,\lambda}(x).
\end{aligned}
\]
The third inequality follows by considering separately
$|z|\leq2$ and $|z|>2$.

We now fix an unordered pair of distinct indices. After interchanging
the two indices within this pair, we may assume that
\[
\lambda_i\geq\lambda_j.
\]
This does not affect the cross Dirichlet integral, since
\[
\int_\Omega\nabla PU_i\cdot\nabla PU_j\,\ud x
=
-\int_\Omega PU_i\Delta PU_j\,\ud x
=
-\int_\Omega PU_j\Delta PU_i\,\ud x.
\]

If
\[
|x-a_i|\leq\lambda_i^{-1},
\]
then $x\in\Omega_{\delta_0/2}$ for $\va_0$ sufficiently small.
Using \eqref{eq:PU-lower-interior}, we have $PU_{i}^{\frac{n+2}{n-2}}\geq cU_{i}^{\frac{n+2}{n-2}}\geq c \lambda_i^{\frac{n+2}{2}}$. Also, by $|x-a_j| \leq|x-a_i|+|a_i-a_j| \leq \lambda_i^{-1}+|a_i-a_j|\leq \lambda_j^{-1}+|a_i-a_j|$, we have 
$$
1+\lambda_j^2|x-a_j|^2 \leq C(1+\lambda_j^2 |a_i-a_j|^2),
$$
which implies that
$$
U_j(x) \geq c\lambda_j^{\frac{n-2}{2}}(1+\lambda_j^2 |a_i-a_j|^2)^{-\frac{n-2}{2}}.
$$
Hence, if $|x-a_{i}| \leq \lda_{i}^{-1}$, then
\[
\begin{aligned}
(PU_i)^{\frac{n+2}{n-2}}PU_j
&\geq
c\lambda_i^{\frac{n+2}{2}}
\lambda_j^{\frac{n-2}{2}}
\left(
1+\lambda_j^2|a_i-a_j|^2
\right)^{-\frac{n-2}{2}}\\
&=c\lambda_i^n\Big(\frac{\lambda_i}{\lambda_j}+\lambda_i \lambda_j |a_i-a_j|^2\Big)^{-\frac{n-2}{2}}\\
&\geq c\lambda_i^{n}\Big(\frac{\lambda_i}{\lambda_j}+\frac{\lambda_j}{\lambda_i}+\lambda_i\lambda_j|a_i-a_j|^2\Big)^{-\frac{n-2}{2}}\\
&=
c\lambda_i^n\Lambda_{i,j}.
\end{aligned}
\]

Furthermore, if $|x-a_i| \leq \lambda_i^{-1}$, then
$$
|x-a_j| \geq |a_i-a_j|-|x-a_i|\geq |a_i-a_j|-\lambda_i^{-1}.
$$
We want to show that there exists a constant $c>0$ such that
\be\label{eq:1lambdaj}
1+\lambda_j^2|x-a_j|^2 \geq c(1+\lambda_j^2 |a_i-a_j|^2).
\ee
Indeed, if $|a_i-a_j| \leq 2 \lambda_i^{-1}$, by $\lambda_i \geq \lambda_j$, we obtain $\lambda_j |a_i-a_j| \leq 2 {\lambda_j}\lambda_i^{-1} \leq 2$ and $1+\lambda_j^2 |a_i-a_j|^2 \leq 5$. Hence,
$$
1+\lambda_j^2|x-a_j|^2 \geq 1\geq \frac{1}{5}(1+\lambda_j^2 |a_i-a_j|^2).
$$
If $|a_i-a_j| > 2 \lambda_i^{-1}$, then $|a_i-a_j|-\lambda_i^{-1} \geq \frac{1}{2} |a_i-a_j|$ and $|x-a_j| \geq \frac{1}{2} |a_i-a_j|$. Hence,
$$
1+\lambda_j^2|x-a_j|^2 \geq 1+\frac{1}{4} \lambda_j^2 |a_i-a_j|^2 \geq \frac{1}{4}(1+\lambda_j^2 |a_i-a_j|^2).
$$
Since $P U_i \leq U_i$, $P U_j \leq U_j$, it follows from \eqref{eq:1lambdaj} that
\[
\begin{aligned}
P U_i(x)(P U_j)^{\frac{n+2}{n-2}} & \leq C \lambda_i^{\frac{n-2}{2}} \lambda_j^{\frac{n+2}{2}}(1+\lambda_j^2|x-a_j|^2)^{-\frac{n+2}{2}}\\
&\leq C \lambda_i^{\frac{n-2}{2}} \lambda_j^{\frac{n+2}{2}}(1+\lambda_j^2 |a_i-a_j|^2)^{-\frac{n+2}{2}} \\
& =C \lambda_i^n\Big(\frac{\lambda_i}{\lambda_j}+\lambda_i \lambda_j |a_i-a_j|^2\Big)^{-\frac{n+2}{2}} \\
& \leq C \lambda_i^n \Big(\frac{\lambda_i}{\lambda_j}+\frac{\lambda_j}{\lambda_i}+\lambda_i\lambda_j|a_i-a_j|^2\Big)^{-\frac{n+2}{2}}.
\end{aligned}
\]
Consequently,
\be\label{eq:<<}
\frac{
PU_i(PU_j)^{\frac{n+2}{n-2}}
}{
(PU_i)^{\frac{n+2}{n-2}}PU_j
}
\leq
C\Lambda_{i,j}^{\frac4{n-2}}
\end{equation}
on $B_{\lambda_i^{-1}}(a_i)$ for $i\neq j$. The quantities $F(PU_i)$ are uniformly bounded above and below
away from zero. Therefore, \eqref{eq:<<} implies
\[
\frac{
F(PU_j)^{\frac n2}
(PU_j)^{\frac{2n}{n-2}}
}{
F(PU_i)^{\frac n2}
(PU_i)^{\frac{2n}{n-2}}
}
\leq
C\Lambda_{i,j}^{\frac{2n}{n-2}}
\]
on $B_{\lambda_i^{-1}}(a_i)$. For $0<\theta<1$, there exists $c_\theta>0$ such that
\be\label{eq:(A+B)p}
(A+B)^{\frac2n}
\geq
B^{\frac2n}
+
c_\theta A^{\frac2n}
\end{equation}
whenever
$
0\leq B\leq\theta A.
$
This gives on
$B_{\lambda_i^{-1}}(a_i)$ that
\[
\begin{aligned}
&
\left[
F(PU_i)^{\frac n2}
(PU_i)^{\frac{2n}{n-2}}
+
F(PU_j)^{\frac n2}
(PU_j)^{\frac{2n}{n-2}}
\right]^{\frac2n}
PU_iPU_j\\
&\quad\geq
F(PU_j)PU_i(PU_j)^{\frac{n+2}{n-2}}
+
cF(PU_i)(PU_i)^{\frac{n+2}{n-2}}PU_j\\
&\quad\geq
F(PU_j)PU_i(PU_j)^{\frac{n+2}{n-2}}
+
c\lambda_i^n\Lambda_{i,j}.
\end{aligned}
\]
Integrating over $B_{\lambda_i^{-1}}(a_i)$, whose measure is
comparable to $\lambda_i^{-n}$, yields
\be\label{eq:(A+B)p-1}
\begin{aligned}
&
\int_\Omega
\left[
F(PU_i)^{\frac n2}
(PU_i)^{\frac{2n}{n-2}}
+
F(PU_j)^{\frac n2}
(PU_j)^{\frac{2n}{n-2}}
\right]^{\frac2n}
PU_iPU_j\,\ud x\\
&\quad\geq
\int_\Omega
F(PU_j)PU_i(PU_j)^{\frac{n+2}{n-2}}
\,\ud x
+
c\Lambda_{i,j}.
\end{aligned}
\ee

We next estimate the interaction between $\xi_{\bm z}$ and a bubble.
Since
\[
\frac12
\leq
\frac{\xi_{\bm z}}{u_\infty}
\leq2
\]
and $u_\infty$ is positive on $\Omega_{\delta_0/2}$, there exist
constants $c,C>0$ such that
\[
c\leq\xi_{\bm z}(x)\leq C
\qquad
\mbox{for }x\in\Omega_{\delta_0/2}.
\]
On $B_{\lambda_k^{-1}}(a_k)$,
\[
PU_k\geq c\lambda_k^{\frac{n-2}{2}},
\]
and hence
\[
\frac{
F(\xi_{\bm z})^{\frac n2}
\xi_{\bm z}^{\frac{2n}{n-2}}
}{
F(PU_k)^{\frac n2}
(PU_k)^{\frac{2n}{n-2}}
}
\leq
C\lambda_k^{-n}.
\]
Applying \eqref{eq:(A+B)p}, we obtain
\[
\begin{aligned}
&
\left[
F(\xi_{\bm z})^{\frac n2}
\xi_{\bm z}^{\frac{2n}{n-2}}
+
F(PU_k)^{\frac n2}
(PU_k)^{\frac{2n}{n-2}}
\right]^{\frac2n}
\xi_{\bm z}PU_k\\
&\quad\geq
F(\xi_{\bm z})
\xi_{\bm z}^{\frac{n+2}{n-2}}PU_k
+
cF(PU_k)\xi_{\bm z}(PU_k)^{\frac{n+2}{n-2}}\\
&\quad\geq
F(\xi_{\bm z})
\xi_{\bm z}^{\frac{n+2}{n-2}}PU_k
+
c\lambda_k^{\frac{n+2}{2}}
\end{aligned}
\]
on $B_{\lambda_k^{-1}}(a_k)$. Consequently,
\be\label{164}
\begin{aligned}
&
\int_\Omega
\left[
F(\xi_{\bm z})^{\frac n2}
\xi_{\bm z}^{\frac{2n}{n-2}}
+
F(PU_k)^{\frac n2}
(PU_k)^{\frac{2n}{n-2}}
\right]^{\frac2n}
\xi_{\bm z}PU_k\,\ud x\\
&\quad\geq
\int_\Omega
F(\xi_{\bm z})
\xi_{\bm z}^{\frac{n+2}{n-2}}PU_k\,\ud x
+
c\lambda_k^{-\frac{n-2}{2}}.
\end{aligned}
\ee

Combining \eqref{eq:2.11-3}, \eqref{eq:(A+B)p-1}, and
\eqref{164}, we find
\be\label{eq:2.11-2}
\begin{aligned}
&
\left(
Y_\Omega(\xi_{\bm z})^{\frac n2}
+
\sum_{k=1}^{\ell}
Y_\Omega(PU_k)^{\frac n2}
\right)^{\frac2n}
\left(
\int_\Omega
\sigma^{\frac{2n}{n-2}}\,\ud x
\right)^{\frac{n-2}{n}}\\
&\geq
\int_\Omega
F(\xi_{\bm z})
\xi_{\bm z}^{\frac{2n}{n-2}}\,\ud x+
r_\infty^{-\frac{n-2}{2}}
\sum_{k=1}^{\ell}\alpha_k^2
\int_\Omega
F(PU_k)(PU_k)^{\frac{2n}{n-2}}\,\ud x\\
&\quad+
2r_\infty^{-\frac{n-2}{4}}
\sum_{k=1}^{\ell}\alpha_k
\int_\Omega
F(\xi_{\bm z})
\xi_{\bm z}^{\frac{n+2}{n-2}}PU_k\,\ud x+
c\sum_{k=1}^{\ell}
\lambda_k^{-\frac{n-2}{2}}\\
&\quad+
2r_\infty^{-\frac{n-2}{2}}
\sum_{i<j}\alpha_i\alpha_j
\int_\Omega
F(PU_j)PU_i(PU_j)^{\frac{n+2}{n-2}}\,\ud x+
c\sum_{i<j}\Lambda_{i,j}.
\end{aligned}
\ee
It follows from \eqref{eq:2.11-2} and \eqref{eq:2.11-1} that
\be\label{eq:2.11-4}
\begin{aligned}
&
Y_\Omega(\sigma)
\left(
\int_\Omega
\sigma^{\frac{2n}{n-2}}\,\ud x
\right)^{\frac{n-2}{n}}\\
&\leq
\left(
Y_\Omega(\xi_{\bm z})^{\frac n2}
+
\sum_{k=1}^{\ell}
Y_\Omega(PU_k)^{\frac n2}
\right)^{\frac2n}
\left(
\int_\Omega
\sigma^{\frac{2n}{n-2}}\,\ud x
\right)^{\frac{n-2}{n}}\\
&\quad-
2r_\infty^{-\frac{n-2}{4}}
\sum_{k=1}^{\ell}\alpha_k
\int_\Omega
PU_k
\left[
\Delta\xi_{\bm z}
+
F(\xi_{\bm z})
\xi_{\bm z}^{\frac{n+2}{n-2}}
\right]\,\ud x\\
&\quad-
2r_\infty^{-\frac{n-2}{2}}
\sum_{i<j}\alpha_i\alpha_j
\int_\Omega
PU_i
\left[
\Delta PU_j
+
F(PU_j)(PU_j)^{\frac{n+2}{n-2}}
\right]\,\ud x\\
&\quad-
c\sum_{k=1}^{\ell}
\lambda_k^{-\frac{n-2}{2}}
-
c\sum_{i<j}\Lambda_{i,j}.
\end{aligned}
\ee

We now estimate the two error terms. Since the map
$\bm z\mapsto\xi_{\bm z}$ is analytic in
$W^{2,p}(\Omega)\cap W_0^{1,p}(\Omega)$ and
\[
F(u_\infty)
=
\frac{\int_\Omega|\nabla u_\infty|^2\,\ud x}
{\int_\Omega u_\infty^{\frac{2n}{n-2}}\,\ud x}
=
r_\infty,
\]
we have
\[
\left\|
\Delta\xi_{\bm z}
+
F(\xi_{\bm z})
\xi_{\bm z}^{\frac{n+2}{n-2}}
\right\|_{L^p(\Omega)}
\leq
C|\bm z|.
\]
Moreover, since $p>n$,
\[
\|PU_k\|_{L^{\frac p{p-1}}(\Omega)}
\leq
C\lambda_k^{-\frac{n-2}{2}}.
\]
Therefore,
\[
\begin{aligned}
&
\left|
\int_\Omega
PU_k
\left[
\Delta\xi_{\bm z}
+
F(\xi_{\bm z})
\xi_{\bm z}^{\frac{n+2}{n-2}}
\right]\,\ud x
\right|\leq
C|\bm z|\lambda_k^{-\frac{n-2}{2}}
=
o\left(
\lambda_k^{-\frac{n-2}{2}}
\right).
\end{aligned}
\]

It follows from \eqref{eq:PUiUi^p} and
\eqref{eq:PUi^p+1} that
\[
F(PU_j)
=
1+
\frac{n+2}{n-2}
\frac{C_1}{C_2}
\lambda_j^{2-n}H(a_j,a_j)
+
o(\lambda_j^{2-n}).
\]
Since
\[
\Delta PU_j=-U_j^{\frac{n+2}{n-2}},
\]
we obtain
\[
\begin{aligned}
&
\left|
\int_\Omega
PU_i
\left[
\Delta PU_j
+
F(PU_j)(PU_j)^{\frac{n+2}{n-2}}
\right]\,\ud x
\right|\\
&\quad\leq
\int_\Omega
PU_i
\left[
U_j^{\frac{n+2}{n-2}}
-
(PU_j)^{\frac{n+2}{n-2}}
\right]\,\ud x+
|F(PU_j)-1|
\int_\Omega
PU_i(PU_j)^{\frac{n+2}{n-2}}\,\ud x\\
&\quad=
o(\Lambda_{i,j})
+
o(\lambda_i^{2-n})
+
o(\lambda_j^{2-n}),
\end{aligned}
\]
where we used \eqref{eq:linear-projection-error} and \eqref{eq:PUiPUj^p} in the last equality.

Since
$
\lambda_k^{2-n}
=
o\left(
\lambda_k^{-\frac{n-2}{2}}
\right),
$
all these error terms can be absorbed into the two positive
deficits in \eqref{eq:2.11-4}, after decreasing $\va_0$ if
necessary. Hence,
\[
\begin{aligned}
Y_\Omega(\sigma)
\left(
\int_\Omega
\sigma^{\frac{2n}{n-2}}\,\ud x
\right)^{\frac{n-2}{n}}
&\leq
\left(
Y_\Omega(\xi_{\bm z})^{\frac n2}
+
\sum_{k=1}^{\ell}
Y_\Omega(PU_k)^{\frac n2}
\right)^{\frac2n}
\left(
\int_\Omega
\sigma^{\frac{2n}{n-2}}\,\ud x
\right)^{\frac{n-2}{n}}\\
&\qquad-
c\sum_{k=1}^{\ell}
\lambda_k^{-\frac{n-2}{2}}
-
c\sum_{i<j}\Lambda_{i,j}.
\end{aligned}
\]

We finally estimate the mass of $\sigma$. Using
\[
\left|
(a+b)^{\frac{2n}{n-2}}
-a^{\frac{2n}{n-2}}
-b^{\frac{2n}{n-2}}
\right|
\leq
C\left(
a^{\frac{n+2}{n-2}}b
+
ab^{\frac{n+2}{n-2}}
\right),
\qquad a,b\geq0,
\]
together with the uniform boundedness of $\xi_{\bm z}$ and
\[
\int_\Omega PU_k\,\ud x
+
\int_\Omega
(PU_k)^{\frac{n+2}{n-2}}\,\ud x
\leq
C\lambda_k^{-\frac{n-2}{2}},
\]
we obtain
\[
\begin{aligned}
\int_\Omega
\sigma^{\frac{2n}{n-2}}\,\ud x
={}&
\int_\Omega
\xi_{\bm z}^{\frac{2n}{n-2}}\,\ud x+
r_\infty^{-\frac n2}
\int_\Omega
\left(
\sum_{k=1}^{\ell}\alpha_kPU_k
\right)^{\frac{2n}{n-2}}\,\ud x
+
o(1).
\end{aligned}
\]
By Lemma \ref{lem:regular-surface-expansions},
\[
\int_\Omega
\xi_{\bm z}^{\frac{2n}{n-2}}\,\ud x
=
\int_\Omega
u_\infty^{\frac{2n}{n-2}}\,\ud x
+
o(1).
\]
Moreover, \eqref{eq:sum-bubbles-power}, 
\eqref{eq:PUi^p+1} and \eqref{eq:PUiPUj^p} give
\[
\begin{aligned}
\int_\Omega
\left(
\sum_{k=1}^{\ell}\alpha_kPU_k
\right)^{\frac{2n}{n-2}}\,\ud x
&=
\sum_{k=1}^{\ell}
\alpha_k^{\frac{2n}{n-2}}
\int_\Omega
(PU_k)^{\frac{2n}{n-2}}\,\ud x
+
o(1)=
C_2\ell+o(1).
\end{aligned}
\]
Consequently,
\[
\int_\Omega
\sigma^{\frac{2n}{n-2}}\,\ud x
=
\int_\Omega
u_\infty^{\frac{2n}{n-2}}\,\ud x
+
r_\infty^{-\frac n2}C_2\ell
+
o(1).
\]
Since $u_\infty$ satisfies
\[
-\Delta u_\infty
=
r_\infty
u_\infty^{\frac{n+2}{n-2}},
\]
we have
\[
Y_\Omega(u_\infty)
=
r_\infty
\left(
\int_\Omega
u_\infty^{\frac{2n}{n-2}}\,\ud x
\right)^{\frac2n},
\]
and therefore
\[
\int_\Omega
u_\infty^{\frac{2n}{n-2}}\,\ud x
=
r_\infty^{-\frac n2}
Y_\Omega(u_\infty)^{\frac n2}.
\]
It follows that
\be\label{eq:estsigma}
\begin{aligned}
\int_\Omega
\sigma^{\frac{2n}{n-2}}\,\ud x
&=
r_\infty^{-\frac n2}
\left(
Y_\Omega(u_\infty)^{\frac n2}
+
C_2\ell
\right)
+
o(1)=
1+o(1),
\end{aligned}
\ee
where in the last equality we used \eqref{eq:r-infty-level}.  Dividing the preceding
inequality by this factor and decreasing $c>0$ proves the result.
\end{proof}

It follows from \eqref{eq:PUiUi^p} and
\eqref{eq:PUi^p+1} that
\[
\begin{aligned}
Y_\Omega(PU_k)
&=
\frac{
C_2
-
C_1\lambda_k^{2-n}H(a_k,a_k)
+
o(\lambda_k^{2-n})
}{
\left(
C_2
-
\frac{2n}{n-2}
C_1\lambda_k^{2-n}H(a_k,a_k)
+
o(\lambda_k^{2-n})
\right)^{\frac{n-2}{n}}}\\
&=
C_2^{\frac2n}
+
C_1C_2^{\frac{2-n}{n}}
\lambda_k^{2-n}H(a_k,a_k)
+
o(\lambda_k^{2-n}).
\end{aligned}
\]
In particular,
\[
Y_\Omega(PU_k)^{\frac n2}
=
C_2+O(\lambda_k^{2-n}).
\]

By Proposition \ref{prop:regular-part},
\[
\begin{aligned}
Y_\Omega(\xi_{\bm z})
\leq{}&
Y_\Omega(u_\infty)+
C
\left(
\int_\Omega
|\mathcal R_u-r_\infty|^{\frac{2n}{n+2}}
u^{\frac{2n}{n-2}}\,\ud x
\right)^{\frac{n+2}{2n}(1+\gamma)}+
C\sum_{k=1}^{\ell}
\lambda_k^{-\frac{n-2}{2}(1+\gamma)}.
\end{aligned}
\]
Therefore, Proposition \ref{pro:Yom(u)-1} yields
\[
\begin{aligned}
&Y_\Omega(\sigma)
+
c\sum_{k=1}^{\ell}
\lambda_k^{-\frac{n-2}{2}}
+
c\sum_{i<j}\Lambda_{i,j}\\
&\leq
\left(
Y_\Omega(u_\infty)^{\frac n2}
+
C_2\ell
\right)^{\frac2n}+
C
\left(
\int_\Omega
|\mathcal R_u-r_\infty|^{\frac{2n}{n+2}}
u^{\frac{2n}{n-2}}\,\ud x
\right)^{\frac{n+2}{2n}(1+\gamma)}\\
&\quad+
C\sum_{k=1}^{\ell}
\lambda_k^{-\frac{n-2}{2}(1+\gamma)}
+
C\sum_{k=1}^{\ell}\lambda_k^{2-n}.
\end{aligned}
\]
Using  \eqref{eq:r-infty-level} 
and observing that
$
\lambda_k^{-\frac{n-2}{2}(1+\gamma)}
+
\lambda_k^{2-n}
=
o\left(
\lambda_k^{-\frac{n-2}{2}}
\right),
$
we may absorb the last two sums into the left-hand side and finally obtain
\be\label{eq:Bruinfty>0}
\begin{aligned}
&Y_\Omega(\sigma)
+
c\sum_{k=1}^{\ell}
\lambda_k^{-\frac{n-2}{2}}
+
c\sum_{i<j}\Lambda_{i,j}\leq
r_\infty
+
C
\left(
\int_\Omega
|\mathcal R_u-r_\infty|^{\frac{2n}{n+2}}
u^{\frac{2n}{n-2}}\,\ud x
\right)^{\frac{n+2}{2n}(1+\gamma)}.
\end{aligned}
\ee

\section{Projection estimates and gradient inequalities}\label{sec:gradient}

We will derive energy expansions near the
almost-critical surfaces in terms of the interaction parameters
$\Lambda_{i,j}$, the concentration parameters $\lambda_k$, the
amplitude defects $|\alpha_k-1|$, and the orthogonal remainder $w$.
In this section, we estimate these quantities in terms of the
Euler--Lagrange residual first.

When $u_\infty>0$, let $L_1$ be the largest integer such that
\[
\mu_l<\frac{n+2}{n-2}r_\infty,
\qquad 1\leq l\leq L_1.
\]
We set
\[
\bm z'=(z_1,\ldots,z_{L_1}),
\qquad
\bm z''=(z_{L_1+1},\ldots,z_L).
\]
If $L_1=L$, we use the convention $\bm z''=0$.

Our main estimates are as follows.

\begin{thm}\label{thm:LS}
There exist $\va_1\in(0,1)$ and $C>0$, depending only on the fixed
data, such that the following assertions hold for every
$\va\in(0,\va_1)$ and every normalized function
$
u\in\mathcal D\cap\V(u_\infty,\ell,\delta;\va)
$
satisfying
$
\|\mathcal R_u-r_\infty\|_{C^0(\Omega)}<\va
$
and
$
\int_\Omega u^{\frac{2n}{n-2}}\,\ud x=1.
$
Let
\[
u=\sigma+w
\]
be the decomposition determined by \eqref{eq:vp} and \eqref{eq:decomposition-1}. Recall from \eqref{eq:Gu} that
\[
 \mathfrak G(u)
 =
 \left(
 \int_\Omega
 |\mathcal R_u-r_\infty|^{\frac{2n}{n+2}}
 u^{\frac{2n}{n-2}}\,\ud x
 \right)^{\frac{n+2}{2n}}.
\]
Then we have
\begin{itemize}

\item[(i)] \textbf{Estimate of the amplitude parameters.}

If $u_\infty\equiv0$, then
\be\label{thm:alpha-1-0}
\begin{aligned}
\sum_{k=1}^{\ell}|\alpha_k-1|
={}&
O\left(\mathfrak G(u)\right)+
O\left(\sum_{k=1}^{\ell}\lambda_k^{2-n}\right)
+
O\left(\sum_{i\neq j}\Lambda_{i,j}\right).
\end{aligned}
\ee

If $u_\infty>0$, then
\be\label{thm:alpha-1}
\begin{aligned}
\sum_{k=1}^{\ell}|\alpha_k-1|
={}&
O\left(\mathfrak G(u)\right)+
O\left(
\sum_{k=1}^{\ell}
\lambda_k^{-\frac{n-2}{2}}
\right)
+
O\left(
\sum_{i\neq j}\Lambda_{i,j}
\right).
\end{aligned}
\ee

Under the additional assumption that
$
Y_\Omega(u)\geq r_\infty,
$
we have
\be\label{eq:alpha-1inc}
\sum_{k=1}^{\ell}|\alpha_k-1|
=
\begin{cases}
\displaystyle
O\left( \mathfrak G(u)\right)
+
O\left(
\sum_{k=1}^{\ell}
\lambda_k^{
2-n}
\right),
& u_\infty\equiv0,\\[4mm]
\displaystyle
O\left( \mathfrak G(u)\right),
& u_\infty>0.
\end{cases}
\ee

\item[(ii)] \textbf{Estimate of the orthogonal remainder.}

If $u_\infty\equiv0$, then
\be\label{thm:w-0}
\begin{aligned}
\|w\|
={}&
O\left( \mathfrak G(u)\right)\\
&+
\begin{cases}
\displaystyle
O\left(
\sum_{k=1}^{\ell}\lambda_k^{-\frac{n+2}{2}}
\right)+O\left(
\sum_{i\neq j}
\Lambda_{i,j}^{\frac{n+2}{2(n-2)}}
\left(
\ln\Lambda_{i,j}^{-1}
\right)^{\frac{n+2}{2n}}
\right),
& n>6,\\[2mm]
\displaystyle
O\left(
\sum_{k=1}^{\ell}
\lambda_k^{-4}(\ln\lambda_k)^{\frac23}
\right)+O\left(
\sum_{i\neq j}
\Lambda_{i,j}
\left(
\ln\Lambda_{i,j}^{-1}
\right)^{\frac{2}{3}}
\right),
& n=6,\\[2mm]
\displaystyle
O\left(
\sum_{k=1}^{\ell}\lambda_k^{2-n}
\right)+O\left(
\sum_{i\neq j}\Lambda_{i,j}
\right),
& n=3,4,5,
\end{cases}
\end{aligned}
\ee

If $u_\infty>0$, then
\be\label{thm:w}
\begin{aligned}
\|w\|
={}&
O
\left( \mathfrak G(u)
\right)\\
&+
\begin{cases}
\displaystyle
O\left(\sum_{k=1}^{\ell}
\lambda_k^{-\frac{n+2}{4}}
(\ln\lambda_k)^{\frac{n+2}{2n}}\right)+O\left(\sum_{i\neq j}
\Lambda_{i,j}^{\frac{n+2}{2(n-2)}}
\left(
\ln\Lambda_{i,j}^{-1}
\right)^{\frac{n+2}{2n}}\right),
&n\geq6,\\[3mm]
\displaystyle
O\left(\sum_{k=1}^{\ell}\lambda_k^{\frac{2-n}{2}}\right)+O\left(\sum_{i\neq j}\Lambda_{i,j}\right),
&n=3,4,5,
\end{cases}
\end{aligned}
\ee

If $u_\infty\equiv0$ and $Y_\Omega(u)\geq r_\infty$, then, 
\be\label{eq:w3}
\begin{aligned}
\|w\|
={}&
O\left(\mathfrak G(u)\right)+
\begin{cases}
\displaystyle
O\left(
\sum_{k=1}^{\ell}\lambda_k^{-\frac{n+2}{2}}(\ln\lambda_k)^\frac{n+2}{2n}
\right),
& n\geq6,\\[2mm]
\displaystyle
O\left(
\sum_{k=1}^{\ell}\lambda_k^{2-n}
\right),
& n=3,4,5.
\end{cases}
\end{aligned}
\ee

If $u_\infty>0$ and $Y_\Omega(u)\geq r_\infty$, then
\be\label{eq:w4}
\|w\|
=
\begin{cases}
\displaystyle
O\left(\mathfrak G(u)\right)+
O\left(
\mathfrak G(u)^{\frac{(n+2)}{2(n-2)}(1+\gamma)}
\left(-\ln\mathfrak G(u)\right)^{\frac{n+2}{2n}}\right),\quad
&n\geq6,\\[5mm]
\displaystyle
O\left(\mathfrak G(u)\right), \quad
&n=3,4,5.
\end{cases}
\ee

\item[(iii)] \textbf{Estimate of the regular-part parameters.}

Suppose that $u_\infty>0$. Then
\[
\begin{aligned}
|\bm z'|^2
={}&
o(|\bm z''|^2)+
O\left(\mathfrak G(u)^{1+\gamma}\right)+
O\left(
\sum_{k=1}^{\ell}
\lambda_k^{-\frac{n-2}{2}(1+\gamma)}
\right).
\end{aligned}
\]
Under the additional assumption that
$
Y_\Omega(u)\geq r_\infty,
$
we have
\[
|\bm z'|^2
=
o(|\bm z''|^2)
+
O\left(\mathfrak G(u)^{1+\gamma}\right).
\]

\end{itemize}
\end{thm}

\begin{proof}[Proof of Theorem \ref{thm:LS}\,(i) and (ii)]

If $\ell=0$, the assertions concerning the amplitude parameters are
empty. We therefore assume $\ell\geq1$ whenever the parameters
$\alpha_i$, $a_i$, and $\lambda_i$ occur.

Set
\[
\delta(n):=
\min\left\{
1,\frac4{n-2}
\right\}.
\]

We first derive several estimates that will be used in both cases
$u_\infty\equiv0$ and $u_\infty>0$.

\medskip

\noindent
\textbf{Step 1. Common amplitude estimates.}

\medskip

Fix $1\leq j\leq\ell$. Recall from \eqref{eq:w-bubble-orthogonality} that
\be\label{eq:amplitude-orthogonality}
\int_\Omega
u^{\frac4{n-2}}PU_jw\,\ud x=0.
\ee
Since $\alpha_j$ remains in a fixed neighborhood of $1$, in either
case we have
\be\label{eq:sigma-dominates-PUj}
\sigma\geq cPU_j
\qquad\mbox{in }\Omega
\ee
for some constant $c>0$ independent of the parameters.

We first prove
\be\label{eq:quadratic-amplitude-orthogonality}
\left|
\int_\Omega
\sigma^{\frac4{n-2}}PU_jw\,\ud x
\right|
\leq C\|w\|^2.
\ee
Indeed, by \eqref{eq:amplitude-orthogonality},
\[
\int_\Omega
\sigma^{\frac4{n-2}}PU_jw\,\ud x
=
\int_\Omega
\left(
\sigma^{\frac4{n-2}}
-u^{\frac4{n-2}}
\right)PU_jw\,\ud x.
\]

Suppose first that $3\leq n\leq6$. Then
\[
\left|
u^{\frac4{n-2}}
-\sigma^{\frac4{n-2}}
\right|
\leq
C\sigma^{\frac{6-n}{n-2}}|w|
+
C|w|^{\frac4{n-2}}.
\]
Consequently,
\[
\begin{aligned}
&
\left|
\int_\Omega
\left(
\sigma^{\frac4{n-2}}
-u^{\frac4{n-2}}
\right)PU_jw\,\ud x
\right|\leq
C\int_\Omega
\sigma^{\frac{6-n}{n-2}}PU_jw^2\,\ud x
+
C\int_\Omega
PU_j|w|^{\frac{n+2}{n-2}}\,\ud x.
\end{aligned}
\]
By H\"older's and Sobolev's inequalities,
\[
\begin{aligned}
\int_\Omega
\sigma^{\frac{6-n}{n-2}}PU_jw^2\,\ud x
&\leq
C
\left\|
\sigma^{\frac{6-n}{n-2}}PU_j
\right\|_{L^{\frac n2}(\Omega)}
\|w\|_{L^{\frac{2n}{n-2}}(\Omega)}^2\leq C\|w\|^2,
\end{aligned}
\]
and
\[
\begin{aligned}
\int_\Omega
PU_j|w|^{\frac{n+2}{n-2}}\,\ud x
&\leq
C
\|w\|_{L^{\frac{2n}{n-2}}(\Omega)}
 ^{\frac{n+2}{n-2}}\leq C\|w\|^2.
\end{aligned}
\]
Here we used the smallness of $\|w\|$ and
\[
\frac{n+2}{n-2}\geq2
\qquad\mbox{for }3\leq n\leq6.
\]

Suppose now that $n>6$. Since $\frac4{n-2}\in(0,1)$, for
$a>0$ and $b\geq-a$ one has
\[
\left|
(a+b)^{\frac4{n-2}}
-a^{\frac4{n-2}}
\right|
\leq
Ca^{\frac{6-n}{n-2}}|b|.
\]
Using \eqref{eq:sigma-dominates-PUj}, we obtain
\[
\begin{aligned}
&
\left|
\int_\Omega
\left(
\sigma^{\frac4{n-2}}
-u^{\frac4{n-2}}
\right)PU_jw\,\ud x
\right|\leq
C\int_\Omega
\sigma^{\frac{6-n}{n-2}}PU_jw^2\,\ud x\leq
C\int_\Omega
(PU_j)^{\frac4{n-2}}w^2\,\ud x\leq C\|w\|^2.
\end{aligned}
\]
This proves \eqref{eq:quadratic-amplitude-orthogonality}.

We next claim that
\be\label{eq:u-power-reduction}
\int_\Omega
u^{\frac{n+2}{n-2}}PU_j\,\ud x
=
\int_\Omega
\sigma^{\frac{n+2}{n-2}}PU_j\,\ud x
+
O(\|w\|^2).
\ee
For $3\leq n\leq6$, we use
\[
\begin{aligned}
&
\left|
(\sigma+w)^{\frac{n+2}{n-2}}
-\sigma^{\frac{n+2}{n-2}}
-\frac{n+2}{n-2}
\sigma^{\frac4{n-2}}w
\right|\leq
C\sigma^{\frac{6-n}{n-2}}w^2
+
C|w|^{\frac{n+2}{n-2}}.
\end{aligned}
\]
After multiplication by $PU_j$ and integration, the two terms on
the right are $O(\|w\|^2)$ by the preceding estimates. The linear
term is also $O(\|w\|^2)$ by
\eqref{eq:quadratic-amplitude-orthogonality}.

For $n>6$, we use
\[
\begin{aligned}
&
\left|
(\sigma+w)^{\frac{n+2}{n-2}}
-\sigma^{\frac{n+2}{n-2}}
-\frac{n+2}{n-2}
\sigma^{\frac4{n-2}}w
\right|\leq
C\sigma^{\frac{6-n}{n-2}}w^2.
\end{aligned}
\]
By \eqref{eq:sigma-dominates-PUj},
\[
\int_\Omega
\sigma^{\frac{6-n}{n-2}}PU_jw^2\,\ud x
\leq
C\int_\Omega
(PU_j)^{\frac4{n-2}}w^2\,\ud x
\leq C\|w\|^2.
\]
Together with \eqref{eq:quadratic-amplitude-orthogonality}, this
proves \eqref{eq:u-power-reduction}.

We shall also use the following standard estimate following from H\"older's inequality
\be\label{eq:amplitude-residual}
\begin{aligned}
&
\left|
\int_\Omega
(\mathcal R_u-r_\infty)
u^{\frac{n+2}{n-2}}PU_j\,\ud x
\right|\leq
C
\left(
\int_\Omega
|\mathcal R_u-r_\infty|^{\frac{2n}{n+2}}
u^{\frac{2n}{n-2}}\,\ud x
\right)^{\frac{n+2}{2n}}.
\end{aligned}
\ee

\medskip

\noindent
\textbf{Step 2. Preliminary amplitude estimate when
$u_\infty\equiv0$.}

\medskip

In this case,
\[
\sigma
=
r_\infty^{-\frac{n-2}{4}}
\sum_{i=1}^{\ell}\alpha_iPU_i.
\]
By the definition of $\mathcal R_u$,
\be\label{eq:amplitude-equation-zero}
\begin{aligned}
0
={}&
\int_\Omega
\left(
\Delta u+\mathcal R_u
u^{\frac{n+2}{n-2}}
\right)PU_j\,\ud x\\
={}&
r_\infty
\int_\Omega
u^{\frac{n+2}{n-2}}PU_j\,\ud x+
r_\infty^{-\frac{n-2}{4}}
\sum_{i=1}^{\ell}\alpha_i
\int_\Omega
\Delta PU_i\,PU_j\,\ud x\\
&+
\int_\Omega
\Delta w\,PU_j\,\ud x+
\int_\Omega
(\mathcal R_u-r_\infty)
u^{\frac{n+2}{n-2}}PU_j\,\ud x.
\end{aligned}
\ee
We are going to establish 
\be\label{eq:sigma-power-PUj-zero}
\begin{aligned}
\int_\Omega
\sigma^{\frac{n+2}{n-2}}PU_j\,\ud x
&=
r_\infty^{-\frac{n+2}{4}}
\alpha_j^{\frac{n+2}{n-2}}C_2-
\frac{2n}{n-2}
r_\infty^{-\frac{n+2}{4}}
C_1\lambda_j^{2-n}H(a_j,a_j)\\
&\quad+
\frac{2n}{n-2}
r_\infty^{-\frac{n+2}{4}}
C_1
\sum_{i\neq j}
\left[
c_0\Lambda_{i,j}
-
(\lambda_i\lambda_j)^{-\frac{n-2}{2}}
H(a_i,a_j)
\right]\\
&\quad+
o(\bm\Lambda)
+
o\left(
\sum_{k=1}^{\ell}\lambda_k^{2-n}
\right)
+
O\left(
\sum_{k=1}^{\ell}|\alpha_k-1|^2
\right).
\end{aligned}
\ee
By decreasing $\va_0$ if necessary, we may assume that
\[
\frac12\leq\alpha_i\leq\frac32,
\qquad 1\leq i\leq\ell.
\]
Fix $1\leq j\leq\ell$ and set
\be\label{eq:Aj}
A_j
=
\left\{
x\in\Omega:
\alpha_jPU_j
\geq
\sum_{i\neq j}\alpha_iPU_i
\right\}.
\ee
Choose $\varepsilon>0$ sufficiently small so that
\[
\varepsilon<
\min\left\{
1,\frac{2}{n-2}
\right\},
\]
and, when $3\leq n\leq5$, also
\[
\varepsilon<
\frac{6-n}{n-2}.
\]

On $A_j$, Taylor's formula gives
\[
\begin{aligned}
&
\left(
\sum_{i=1}^{\ell}\alpha_iPU_i
\right)^{\frac{n+2}{n-2}}\\
={}&
(\alpha_jPU_j)^{\frac{n+2}{n-2}}+
\frac{n+2}{n-2}
(\alpha_jPU_j)^{\frac4{n-2}}
\sum_{i\neq j}\alpha_iPU_i+
O\left(
(PU_j)^{\frac4{n-2}-\varepsilon}
\left(
\sum_{i\neq j}PU_i
\right)^{1+\varepsilon}
\right).
\end{aligned}
\]
After multiplication by $PU_j$, the remainder is bounded by
\[
C\sum_{i\neq j}
(PU_j)^{\frac{n+2}{n-2}-\varepsilon}
(PU_i)^{1+\varepsilon}.
\]
It follows from \eqref{eq:two-bubble-epsilon} that
\[
\begin{aligned}
&
\int_{A_j}
(PU_j)^{\frac4{n-2}-\varepsilon}
\left(
\sum_{i\neq j}PU_i
\right)^{1+\varepsilon}
PU_j\,\ud x=
o\left(
\sum_{i\neq j}\Lambda_{i,j}
\right).
\end{aligned}
\]
Therefore,
\be\label{eq:pure-bubble-tested-Bj}
\begin{aligned}
&
\int_{A_j}
\left(
\sum_{i=1}^{\ell}\alpha_iPU_i
\right)^{\frac{n+2}{n-2}}
PU_j\,\ud x\\
={}&
\alpha_j^{\frac{n+2}{n-2}}
\int_{A_j}
(PU_j)^{\frac{2n}{n-2}}\,\ud x+
\frac{n+2}{n-2}
\alpha_j^{\frac4{n-2}}
\sum_{i\neq j}\alpha_i
\int_{A_j}
(PU_j)^{\frac{n+2}{n-2}}PU_i\,\ud x+
o\left(
\sum_{i\neq j}\Lambda_{i,j}
\right).
\end{aligned}
\ee

On $A_j^c$, we have
\[
\alpha_jPU_j
<
\sum_{i\neq j}\alpha_iPU_i.
\]
Taylor expansion around the sum of the remaining bubbles gives
\[
\begin{aligned}
&
\left(
\sum_{i=1}^{\ell}\alpha_iPU_i
\right)^{\frac{n+2}{n-2}}\\
={}&
\left(
\sum_{i\neq j}\alpha_iPU_i
\right)^{\frac{n+2}{n-2}}+
\frac{n+2}{n-2}
\left(
\sum_{i\neq j}\alpha_iPU_i
\right)^{\frac4{n-2}}
\alpha_jPU_j+
O\left(
\Big(
\sum_{i\neq j}PU_i
\Big)^{\frac4{n-2}-\varepsilon}
(PU_j)^{1+\varepsilon}
\right).
\end{aligned}
\]
After multiplication by $PU_j$, the last two terms are bounded by
\[
C
\left(
\sum_{i\neq j}PU_i
\right)^{\frac{n+2}{n-2}-\varepsilon}
(PU_j)^{1+\varepsilon}.
\]
Since $\ell$ is fixed,
\[
\left(
\sum_{i\neq j}PU_i
\right)^{\frac{n+2}{n-2}-\varepsilon}
\leq
C\sum_{i\neq j}
(PU_i)^{\frac{n+2}{n-2}-\varepsilon}.
\]
Thus, by \eqref{eq:two-bubble-epsilon},
\be\label{eq:pure-bubble-tested-Bjc}
\begin{aligned}
&
\int_{A_j^c}
\left(
\sum_{i=1}^{\ell}\alpha_iPU_i
\right)^{\frac{n+2}{n-2}}
PU_j\,\ud x
=
\int_{A_j^c}
\left(
\sum_{i\neq j}\alpha_iPU_i
\right)^{\frac{n+2}{n-2}}
PU_j\,\ud x+
o\left(
\sum_{i\neq j}\Lambda_{i,j}
\right).
\end{aligned}
\ee

As in Section \ref{sec:2}, we also have
\[
\begin{aligned}
&
\int_{A_j^c}
(PU_j)^{\frac{2n}{n-2}}\,\ud x+
\sum_{i\neq j}
\int_{A_j^c}
(PU_j)^{\frac{n+2}{n-2}}PU_i\,\ud x+
\int_{A_j}
\left(
\sum_{i\neq j}\alpha_iPU_i
\right)^{\frac{n+2}{n-2}}
PU_j\,\ud x\\
&\qquad=
o\left(
\sum_{i\neq j}\Lambda_{i,j}
\right).
\end{aligned}
\]
Indeed, on $A_j$ the remaining bubbles are bounded by
$CPU_j$, while on $A_j^c$ the $j$-th bubble is bounded by a
constant multiple of the sum of the remaining bubbles; one then
uses \eqref{eq:two-bubble-epsilon} exactly as above.

Combining \eqref{eq:pure-bubble-tested-Bj} and
\eqref{eq:pure-bubble-tested-Bjc}, we obtain
\be\label{eq:pure-bubble-tested-intermediate}
\begin{aligned}
&
\int_\Omega
\left(
\sum_{i=1}^{\ell}\alpha_iPU_i
\right)^{\frac{n+2}{n-2}}
PU_j\,\ud x\\
&=
\alpha_j^{\frac{n+2}{n-2}}
\int_\Omega
(PU_j)^{\frac{2n}{n-2}}\,\ud x+
\frac{n+2}{n-2}
\alpha_j^{\frac4{n-2}}
\sum_{i\neq j}\alpha_i
\int_\Omega
(PU_j)^{\frac{n+2}{n-2}}PU_i\,\ud x\\
&\quad+
\int_\Omega
\left(
\sum_{i\neq j}\alpha_iPU_i
\right)^{\frac{n+2}{n-2}}
PU_j\,\ud x+
o\left(
\sum_{i\neq j}\Lambda_{i,j}
\right).
\end{aligned}
\ee

It remains to expand the last integral. We claim that
\be\label{eq:remaining-bubbles-tested}
\begin{aligned}
&
\int_\Omega
\left|
\left(
\sum_{i\neq j}\alpha_iPU_i
\right)^{\frac{n+2}{n-2}}
-
\sum_{i\neq j}
\alpha_i^{\frac{n+2}{n-2}}
(PU_i)^{\frac{n+2}{n-2}}
\right|
PU_j\,\ud x=
o(\bm\Lambda).
\end{aligned}
\ee

For $3\leq n\leq5$, the algebraic inequality
\[
\left|
\left(\sum_i b_i\right)^{\frac{n+2}{n-2}}
-
\sum_i b_i^{\frac{n+2}{n-2}}
\right|
\leq
C\sum_{r\neq s}
b_r^{\frac4{n-2}}b_s
\]
gives triple integrals of the form
\[
\int_\Omega
(PU_r)^{\frac4{n-2}}PU_sPU_j\,\ud x.
\]
By Young's inequality,
\[
\begin{aligned}
&
\int_\Omega
(PU_r)^{\frac4{n-2}}PU_sPU_j\,\ud x\leq
\frac12
\int_\Omega
(PU_r)^{\frac4{n-2}}(PU_s)^2\,\ud x
+
\frac12
\int_\Omega
(PU_r)^{\frac4{n-2}}(PU_j)^2\,\ud x.
\end{aligned}
\]
Splitting each integral into the regions where the two bubbles
dominate one another and using \eqref{eq:two-bubble-epsilon}, we
obtain
\[
\int_\Omega
(PU_r)^{\frac4{n-2}}(PU_s)^2\,\ud x
=
o(\Lambda_{r,s}),
\]
and
\[
\int_\Omega
(PU_r)^{\frac4{n-2}}(PU_j)^2\,\ud x
=
o(\Lambda_{r,j}).
\]
After summing over the finitely many indices, this proves
\eqref{eq:remaining-bubbles-tested} for $3\leq n\leq5$.

For $n\geq6$, we use
\[
\left|
\left(\sum_i b_i\right)^{\frac{n+2}{n-2}}
-
\sum_i b_i^{\frac{n+2}{n-2}}
\right|
\leq
C\sum_{r\neq s}
(b_rb_s)^{\frac{n+2}{2(n-2)}}.
\]
For each triple $(r,s,j)$, partition $\Omega$ into the regions
where $PU_r$, $PU_s$, or $PU_j$ is the largest. On the region where
$PU_r$ is the largest, one has, for sufficiently small
$\varepsilon>0$,
\[
\begin{aligned}
&
(PU_rPU_s)^{\frac{n+2}{2(n-2)}}PU_j\leq
C
(PU_r)^{\frac{n+2}{n-2}-\varepsilon}
(PU_s)^{1+\varepsilon}
+
C
(PU_r)^{\frac{n+2}{n-2}-\varepsilon}
(PU_j)^{1+\varepsilon}.
\end{aligned}
\]
The analogous estimates hold on the other two regions. Therefore,
by \eqref{eq:two-bubble-epsilon},
\[
\begin{aligned}
&
\int_\Omega
(PU_rPU_s)^{\frac{n+2}{2(n-2)}}PU_j\,\ud x\leq
C\Lambda_{r,s}^{1+\varepsilon}
+
C\Lambda_{r,j}^{1+\varepsilon}
+
C\Lambda_{s,j}^{1+\varepsilon}.
\end{aligned}
\]
Summing over the finitely many triples yields
$
o(\bm\Lambda),
$
which proves \eqref{eq:remaining-bubbles-tested} for $n\geq6$.

It follows from
\eqref{eq:pure-bubble-tested-intermediate} and
\eqref{eq:remaining-bubbles-tested} that
for every $1\leq j\leq\ell$, we have
\be\label{eq:pure-bubble-tested}
\begin{aligned}
&
\int_\Omega
\left(
\sum_{i=1}^{\ell}\alpha_iPU_i
\right)^{\frac{n+2}{n-2}}
PU_j\,\ud x\\
&=
\alpha_j^{\frac{n+2}{n-2}}
\int_\Omega
(PU_j)^{\frac{2n}{n-2}}\,\ud x+
\sum_{i\neq j}
\alpha_i^{\frac{n+2}{n-2}}
\int_\Omega
(PU_i)^{\frac{n+2}{n-2}}PU_j\,\ud x\\
&\quad+
\frac{n+2}{n-2}
\alpha_j^{\frac4{n-2}}
\sum_{i\neq j}\alpha_i
\int_\Omega
(PU_j)^{\frac{n+2}{n-2}}PU_i\,\ud x+
o(\bm\Lambda).
\end{aligned}
\ee
Then the estimate \eqref{eq:sigma-power-PUj-zero} follows from \eqref{eq:PUi^p+1} and
\eqref{eq:PUiPUj^p}. 

It follows from \eqref{eq:u-power-reduction} and  \eqref{eq:sigma-power-PUj-zero} that
\be\label{eq:u-power-PUj-zero}
\begin{aligned}
&
\int_\Omega
u^{\frac{n+2}{n-2}}PU_j\,\ud x\\
&=
r_\infty^{-\frac{n+2}{4}}
\alpha_j^{\frac{n+2}{n-2}}C_2-
\frac{2n}{n-2}
r_\infty^{-\frac{n+2}{4}}
C_1\lambda_j^{2-n}H(a_j,a_j)\\
&\quad+
\frac{2n}{n-2}
r_\infty^{-\frac{n+2}{4}}
C_1
\sum_{i\neq j}
\left[
c_0\Lambda_{i,j}
-
(\lambda_i\lambda_j)^{-\frac{n-2}{2}}
H(a_i,a_j)
\right]\\
&\quad+
o(\bm\Lambda)
+
o\left(
\sum_{k=1}^{\ell}\lambda_k^{2-n}
\right)+
O\left(
\sum_{k=1}^{\ell}|\alpha_k-1|^2
\right)+O(\|w\|^2).
\end{aligned}
\ee

We next need an estimate of
\(\int_\Omega\Delta w\,PU_j\,\ud x\). 

Since
\[
-\Delta PU_j
=
U_j^{\frac{n+2}{n-2}},
\]
we have
\be\label{eq:Delta-w-decomposition-zero}
\begin{aligned}
\int_\Omega
\Delta w\,PU_j\,\ud x
={}&
-\int_\Omega
(PU_j)^{\frac{n+2}{n-2}}w\,\ud x+
\int_\Omega
\left[
(PU_j)^{\frac{n+2}{n-2}}
-
U_j^{\frac{n+2}{n-2}}
\right]w\,\ud x.
\end{aligned}
\ee

The projection error satisfies
\be\label{eq:projection-error-zero}
\begin{aligned}
&\Big|\int_{\Omega} (PU_j^{\frac{n+2}{n-2}}-U_j^{\frac{n+2}{n-2}})w \,\ud x\Big|\\
&\leq C\lambda_j^{-\frac{n-2}{2}}\int_{\Omega} U_j^{\frac{4}{n-2}}w \,\ud x \\
&\leq C\lambda_j^{-\frac{n-2}{2}}\Big(\int_{\Omega} U_j^{\frac{8n}{(n+2)(n-2)}} \,\ud x\Big)^{\frac{n+2}{2n}} \Big(\int_{\Omega} |w|^{\frac{2n}{n-2}} \,\ud x\Big)^{\frac{n-2}{2n}}\\
&\leq C\begin{cases}
\lambda_j^{-\frac{n+2}{2}}\|w\|  & n> 6,\\
\lambda_j^{-4}(\ln \lambda_j)^{\frac{2}{3}}\|w\|  & n= 6,\\
\lambda_j^{2-n}\|w\|  & n=3,4,5,\\
\end{cases}\\
&\leq C\|w\|^2+C\begin{cases}
\lambda_j^{-n-2}  & n> 6,\\
\lambda_j^{-8}(\ln \lambda_j)^{\frac{4}{3}}  & n= 6,\\
\lambda_j^{4-2n}  & n=3,4,5,\\
\end{cases}\\
&\leq C\|w\|^2+C \lambda_j^{1-n}.
\end{aligned}
\ee

It follows from \eqref{eq:quadratic-amplitude-orthogonality} 
\be\label{eq:mayer-v-1}
\int_\Omega
\left(
\sum_{i=1}^{\ell}\alpha_iPU_i
\right)^{\frac4{n-2}}
PU_jw\,\ud x
=
O(\|w\|^2).
\ee
On $A_j$ (defined in \eqref{eq:Aj}), since
$
\sum_{i\neq j}\alpha_iPU_i
\leq
\alpha_jPU_j,
$
we have
\[
\begin{aligned}
&
\left|
\left(
\sum_{i=1}^{\ell}\alpha_iPU_i
\right)^{\frac4{n-2}}
-
(\alpha_jPU_j)^{\frac4{n-2}}
\right|PU_j\leq
C
\sum_{i\neq j}
PU_i (PU_j)^{\frac4{n-2}}.
\end{aligned}
\]
On $A_j^c$, we have
\[
\begin{aligned}
&
\left|
\left(
\sum_{i=1}^{\ell}\alpha_iPU_i
\right)^{\frac4{n-2}}
-
(\alpha_jPU_j)^{\frac4{n-2}}
\right|PU_j\leq
C
\left(
\sum_{i\neq j}PU_i
\right)^{\frac4{n-2}}PU_j.
\end{aligned}
\]

Suppose first that $3\leq n\leq5$. Since
$\frac4{n-2}>1$, it follows from the fact that $\ell$ is
fixed that
\[
\left(
\sum_{i\neq j}PU_i
\right)^{\frac4{n-2}}PU_j
\leq
C\sum_{i\neq j}
(PU_i)^{\frac4{n-2}}PU_j.
\]
By Proposition B.2 of \cite{FG2020}, for every $i\neq j$,
\[
\begin{aligned}
&
\left(
\int_\Omega
\left(
(PU_j)^{\frac4{n-2}}PU_i
\right)^{\frac{2n}{n+2}}
\,\ud x
\right)^{\frac{n+2}{n}}+
\left(
\int_\Omega
\left(
(PU_i)^{\frac4{n-2}}PU_j
\right)^{\frac{2n}{n+2}}
\,\ud x
\right)^{\frac{n+2}{n}}\leq
C\Lambda_{i,j}^2
=
o(\Lambda_{i,j}).
\end{aligned}
\]

If $n=6$, then $\frac4{n-2}=1$. Proposition B.2 of
\cite{FG2020} gives
\[
\left(
\int_\Omega
(PU_iPU_j)^{\frac{2n}{n+2}}
\,\ud x
\right)^{\frac{n+2}{n}}
\leq
C\Lambda_{i,j}^2
\left(
\ln\Lambda_{i,j}^{-1}
\right)^{\frac43}
=
o(\Lambda_{i,j}).
\]

Suppose now that $n\geq7$. On $A_j$, we have
$PU_i\leq CPU_j$ for every $i\neq j$, and therefore
\[
\begin{aligned}
(PU_j)^{\frac4{n-2}}PU_i
&=
(PU_iPU_j)^{\frac{n+2}{2(n-2)}}
\left(
\frac{PU_i}{PU_j}
\right)^{\frac{n-6}{2(n-2)}}\leq
C
(PU_iPU_j)^{\frac{n+2}{2(n-2)}}.
\end{aligned}
\]
On $A_j^c$, divide the region into finitely many subregions
according to which $\alpha_iPU_i$, $i\neq j$, is maximal. On such
a subregion,
\[
\sum_{k\neq j}PU_k\leq CPU_i,
\qquad
PU_j\leq CPU_i,
\]
and hence
\[
\begin{aligned}
\left(
\sum_{k\neq j}PU_k
\right)^{\frac4{n-2}}PU_j
&\leq
C(PU_i)^{\frac4{n-2}}PU_j\\
&=
C
(PU_iPU_j)^{\frac{n+2}{2(n-2)}}
\left(
\frac{PU_j}{PU_i}
\right)^{\frac{n-6}{2(n-2)}}\\
&\leq
C
(PU_iPU_j)^{\frac{n+2}{2(n-2)}}.
\end{aligned}
\]
Again by Proposition B.2 of \cite{FG2020},
\[
\begin{aligned}
&
\left(
\int_\Omega
(PU_iPU_j)^{\frac{n}{n-2}}
\,\ud x
\right)^{\frac{n+2}{n}}\leq
C
\Lambda_{i,j}^{\frac{n+2}{n-2}}
\left(
\ln\Lambda_{i,j}^{-1}
\right)^{\frac{n+2}{n}}
=
o(\Lambda_{i,j}).
\end{aligned}
\]
Hence,
\[
\begin{aligned}
&
\left(
\int_\Omega
\left|
\left[
\left(
\sum_{i=1}^{\ell}\alpha_iPU_i
\right)^{\frac4{n-2}}
-
(\alpha_jPU_j)^{\frac4{n-2}}
\right]PU_j
\right|^{\frac{2n}{n+2}}
\,\ud x
\right)^{\frac{n+2}{n}}=
o\left(
\sum_{i\neq j}\Lambda_{i,j}
\right).
\end{aligned}
\]
Therefore, by H\"older's, Sobolev's, and Young's inequalities,
for every fixed $\varepsilon>0$,
\[
\begin{aligned}
&
\left|
\int_\Omega
\left[
\left(
\sum_{i=1}^{\ell}\alpha_iPU_i
\right)^{\frac4{n-2}}
-
(\alpha_jPU_j)^{\frac4{n-2}}
\right]PU_jw\,\ud x
\right|\\
&\quad\leq
C
\left(
\int_\Omega
\left|
\left[
\left(
\sum_{i=1}^{\ell}\alpha_iPU_i
\right)^{\frac4{n-2}}
-
(\alpha_jPU_j)^{\frac4{n-2}}
\right]PU_j
\right|^{\frac{2n}{n+2}}
\,\ud x
\right)^{\frac{n+2}{2n}}
\|w\|\\
&\quad\leq
C\|w\|^2
+
o\left(
\sum_{i\neq j}\Lambda_{i,j}
\right).
\end{aligned}
\]
Consequently,
\[
\begin{aligned}
&
\int_\Omega
\left(
\sum_{i=1}^{\ell}\alpha_iPU_i
\right)^{\frac4{n-2}}
PU_jw\,\ud x=
\int_\Omega
(\alpha_jPU_j)^{\frac4{n-2}}PU_jw\,\ud x
+
o\left(
\sum_{i\neq j}\Lambda_{i,j}
\right)
+
O(\|w\|^2).
\end{aligned}
\]
Combining this with \eqref{eq:mayer-v-1}, we obtain
\be\label{eq:localized-PUj-w}
\int_\Omega
(PU_j)^{\frac{n+2}{n-2}}w\,\ud x
=
o\left(
\sum_{i\neq j}\Lambda_{i,j}
\right)
+
O\left(
\|w\|^{2}
\right).
\ee
It follows from
\eqref{eq:Delta-w-decomposition-zero},
\eqref{eq:projection-error-zero}, and
\eqref{eq:localized-PUj-w} that
\be\label{eq:Delta-w-zero-sharp}
\begin{aligned}
\int_\Omega
\Delta w\,PU_j\,\ud x
={}&
o\left(
\sum_{i\neq j}\Lambda_{i,j}
\right)+
O\left(
\|w\|^{2}
\right)
+
o\left(
\sum_{k=1}^{\ell}\lambda_k^{2-n}
\right).
\end{aligned}
\ee

Substituting
\eqref{eq:PUiUj^p}, \eqref{eq:PUiUi^p},
\eqref{eq:amplitude-residual},
\eqref{eq:u-power-PUj-zero}, and
\eqref{eq:Delta-w-zero-sharp}
into \eqref{eq:amplitude-equation-zero}, we obtain
\[
\begin{aligned}
C_2
\left(
\alpha_j^{\frac{n+2}{n-2}}-\alpha_j
\right)
={}&
\frac{n+2}{n-2}
C_1\lambda_j^{2-n}H(a_j,a_j)\\
&-
\frac{n+2}{n-2}
C_1
\sum_{i\neq j}
\left[
c_0\Lambda_{i,j}
-
(\lambda_i\lambda_j)^{-\frac{n-2}{2}}
H(a_i,a_j)
\right]\\
&+
o(\bm\Lambda)
+
o\left(
\sum_{k=1}^{\ell}\lambda_k^{2-n}
\right)+
O\left(
\|w\|^{2}
\right)+
O\left(
\sum_{k=1}^{\ell}|\alpha_k-1|^2
\right)\\
&+
O\left(
\left(
\int_\Omega
|\mathcal R_u-r_\infty|^{\frac{2n}{n+2}}
u^{\frac{2n}{n-2}}\,\ud x
\right)^{\frac{n+2}{2n}}
\right).
\end{aligned}
\]
Hence,
\be\label{eq:amplitude-preliminary-zero}
\begin{aligned}
\sum_{k=1}^{\ell}|\alpha_k-1|
\leq{}&
C
\left(
\int_\Omega
|\mathcal R_u-r_\infty|^{\frac{2n}{n+2}}
u^{\frac{2n}{n-2}}\,\ud x
\right)^{\frac{n+2}{2n}}+
C\sum_{k=1}^{\ell}\lambda_k^{2-n}
+
C\bm\Lambda
+
C\|w\|^{2}.
\end{aligned}
\ee

\medskip

\noindent
\textbf{Step 3. Preliminary estimate of \(w\) when
\(u_\infty\equiv0\).}

\medskip

The algebraic inequality
\be\label{eq:nonlinear-Taylor}
\begin{aligned}
&
\left|
a^{\frac{n+2}{n-2}}
-b^{\frac{n+2}{n-2}}
-\frac{n+2}{n-2}
b^{\frac4{n-2}}(a-b)
\right|\leq
C
b^{\frac4{n-2}-\delta(n)}
|a-b|^{1+\delta(n)}
+
C|a-b|^{\frac{n+2}{n-2}}
\end{aligned}
\ee
gives
\[
\begin{aligned}
-\Delta w
={}&
(\mathcal R_u-r_\infty)
u^{\frac{n+2}{n-2}}+
\frac{n+2}{n-2}r_\infty
\sigma^{\frac4{n-2}}w+
O\left(
\sigma^{\frac4{n-2}-\delta(n)}
|w|^{1+\delta(n)}
+
|w|^{\frac{n+2}{n-2}}
\right)\\
&+
r_\infty^{-\frac{n-2}{4}}
\left[
\left(
\sum_{k=1}^{\ell}\alpha_kPU_k
\right)^{\frac{n+2}{n-2}}
-
\sum_{k=1}^{\ell}
\alpha_k(PU_k)^{\frac{n+2}{n-2}}
\right]\\
&+
r_\infty^{-\frac{n-2}{4}}
\sum_{k=1}^{\ell}\alpha_k
\left[
(PU_k)^{\frac{n+2}{n-2}}
-
U_k^{\frac{n+2}{n-2}}
\right].
\end{aligned}
\]
Multiplying by \(w\), integrating by parts, and applying
Proposition \ref{prop:Coercive-1}, we find
\[
\begin{aligned}
c\|w\|^2
\leq{}&
\left|
\int_\Omega
(\mathcal R_u-r_\infty)
u^{\frac{n+2}{n-2}}w\,\ud x
\right|+
C\|w\|^{2+\delta(n)}\\
&+
C\left|
\int_\Omega
\left[
\left(
\sum_{k=1}^{\ell}\alpha_kPU_k
\right)^{\frac{n+2}{n-2}}
-
\sum_{k=1}^{\ell}
\alpha_k(PU_k)^{\frac{n+2}{n-2}}
\right]w\,\ud x
\right|\\
&+
C\sum_{k=1}^{\ell}
\left|
\int_\Omega
\left[
(PU_k)^{\frac{n+2}{n-2}}
-
U_k^{\frac{n+2}{n-2}}
\right]w\,\ud x
\right|.
\end{aligned}
\]

Using H\"older's inequality, we have
\be\label{eq:w-residual-estimate}
\begin{aligned}
&
\left|
\int_\Omega
(\mathcal R_u-r_\infty)
u^{\frac{n+2}{n-2}}w\,\ud x
\right|\leq
\frac{c}{8}\|w\|^2
+
C
\left(
\int_\Omega
|\mathcal R_u-r_\infty|^{\frac{2n}{n+2}}
u^{\frac{2n}{n-2}}\,\ud x
\right)^{\frac{n+2}{n}}.
\end{aligned}
\ee
We now estimate the multi-bubble term. We first decompose
\[
\begin{aligned}
&
\left(
\sum_{k=1}^{\ell}\alpha_kPU_k
\right)^{\frac{n+2}{n-2}}
-
\sum_{k=1}^{\ell}
\alpha_k(PU_k)^{\frac{n+2}{n-2}}\\
={}&
\sum_{k=1}^{\ell}
\left(
\alpha_k^{\frac{n+2}{n-2}}-\alpha_k
\right)
(PU_k)^{\frac{n+2}{n-2}}+
\left[
\left(
\sum_{k=1}^{\ell}\alpha_kPU_k
\right)^{\frac{n+2}{n-2}}
-
\sum_{k=1}^{\ell}
(\alpha_kPU_k)^{\frac{n+2}{n-2}}
\right].
\end{aligned}
\]

Since $\alpha_k$ is close to $1$,
\[
\left|
\alpha_k^{\frac{n+2}{n-2}}-\alpha_k
\right|
\leq
C|\alpha_k-1|.
\]
Moreover, by H\"older's and Sobolev's inequalities,
\[
\begin{aligned}
\left|
\int_\Omega
(PU_k)^{\frac{n+2}{n-2}}w\,\ud x
\right|
&\leq
\left(
\int_\Omega
(PU_k)^{\frac{2n}{n-2}}\,\ud x
\right)^{\frac{n+2}{2n}}
\|w\|_{L^{\frac{2n}{n-2}}(\Omega)}\leq C\|w\|.
\end{aligned}
\]
Consequently,
\be\label{eq:multi-bubble-amplitude-part}
\begin{aligned}
&
\sum_{k=1}^{\ell}
\left|
\alpha_k^{\frac{n+2}{n-2}}-\alpha_k
\right|
\left|
\int_\Omega
(PU_k)^{\frac{n+2}{n-2}}w\,\ud x
\right|\leq
\frac{c}{16}\|w\|^2
+
C\sum_{k=1}^{\ell}|\alpha_k-1|^2.
\end{aligned}
\ee

For nonnegative $b_1,\ldots,b_\ell$, one has
\[
\left(
\sum_{k=1}^{\ell}b_k
\right)^p
-
\sum_{k=1}^{\ell}b_k^p
\leq
\begin{cases}
\displaystyle
C\sum_{i\neq j}(b_ib_j)^{p/2},
&1<p\leq2,\\[3mm]
\displaystyle
C\sum_{i\neq j}b_i^{p-1}b_j,
&p>2.
\end{cases}
\]
Applying this inequality with
$
p=\frac{n+2}{n-2},
b_k=\alpha_kPU_k,
$
and using that the $\alpha_k$ are uniformly bounded above and below,
we obtain
\[
\begin{aligned}
&
\left|
\left(
\sum_{k=1}^{\ell}\alpha_kPU_k
\right)^{\frac{n+2}{n-2}}
-
\sum_{k=1}^{\ell}
(\alpha_kPU_k)^{\frac{n+2}{n-2}}
\right|\leq
\begin{cases}
\displaystyle
C\sum_{i\neq j}
(PU_iPU_j)^{\frac{n+2}{2(n-2)}},
&n\geq6,\\[3mm]
\displaystyle
C\sum_{i\neq j}
(PU_i)^{\frac4{n-2}}PU_j,
&n=3,4,5.
\end{cases}
\end{aligned}
\]

Suppose first that $n\geq6$. By H\"older's and Sobolev's
inequalities,
\[
\begin{aligned}
&
\int_\Omega
(PU_iPU_j)^{\frac{n+2}{2(n-2)}}|w|\,\ud x\leq
C
\left(
\int_\Omega
(PU_iPU_j)^{\frac{n}{n-2}}\,\ud x
\right)^{\frac{n+2}{2n}}
\|w\|.
\end{aligned}
\]
By \cite[Proposition B.2]{FG2020},
\[
\left(
\int_\Omega
(PU_iPU_j)^{\frac{n}{n-2}}\,\ud x
\right)^{\frac{n+2}{n}}
\leq
C
\Lambda_{i,j}^{\frac{n+2}{n-2}}
\left(
\ln\Lambda_{i,j}^{-1}
\right)^{\frac{n+2}{n}}.
\]
Therefore, by Young's inequality,
\be\label{eq:multi-bubble-interaction-high}
\begin{aligned}
&
C\sum_{i\neq j}
\int_\Omega
(PU_iPU_j)^{\frac{n+2}{2(n-2)}}|w|\,\ud x\leq
\frac{c}{16}\|w\|^2+
C\sum_{i\neq j}
\Lambda_{i,j}^{\frac{n+2}{n-2}}
\left(
\ln\Lambda_{i,j}^{-1}
\right)^{\frac{n+2}{n}}.
\end{aligned}
\ee

Suppose now that $n=3,4,5$. Again by H\"older's and Sobolev's
inequalities,
\[
\begin{aligned}
&
\int_\Omega
(PU_i)^{\frac4{n-2}}PU_j|w|\,\ud x\leq
C
\left(
\int_\Omega
(PU_i)^{
\frac{8n}{(n+2)(n-2)}
}
(PU_j)^{\frac{2n}{n+2}}
\,\ud x
\right)^{\frac{n+2}{2n}}
\|w\|.
\end{aligned}
\]
By \cite[Proposition B.2]{FG2020},
\[
\left(
\int_\Omega
(PU_i)^{
\frac{8n}{(n+2)(n-2)}
}
(PU_j)^{\frac{2n}{n+2}}
\,\ud x
\right)^{\frac{n+2}{n}}
\leq
C\Lambda_{i,j}^2.
\]
It follows from Young's inequality that
\be\label{eq:multi-bubble-interaction-low}
\begin{aligned}
&
C\sum_{i\neq j}
\int_\Omega
(PU_i)^{\frac4{n-2}}PU_j|w|\,\ud x\leq
\frac{c}{16}\|w\|^2
+
C\sum_{i\neq j}\Lambda_{i,j}^2.
\end{aligned}
\ee

Combining \eqref{eq:multi-bubble-amplitude-part},
\eqref{eq:multi-bubble-interaction-high}, and
\eqref{eq:multi-bubble-interaction-low}, we obtain

\be\label{eq:multi-bubble-w-estimate}
\begin{aligned}
&
\left|
\int_\Omega
\left[
\left(
\sum_{k=1}^{\ell}\alpha_kPU_k
\right)^{\frac{n+2}{n-2}}
-
\sum_{k=1}^{\ell}
\alpha_k(PU_k)^{\frac{n+2}{n-2}}
\right]w\,\ud x
\right|\\
&\quad\leq
\frac{c}{8}\|w\|^2
+
C\sum_{k=1}^{\ell}|\alpha_k-1|^2+
\begin{cases}
\displaystyle
C\sum_{i\neq j}
\Lambda_{i,j}^{\frac{n+2}{n-2}}
\left(
\ln\Lambda_{i,j}^{-1}
\right)^{\frac{n+2}{n}},
& n\geq6,\\[3mm]
\displaystyle
C\sum_{i\neq j}\Lambda_{i,j}^2,
& n=3,4,5.
\end{cases}
\end{aligned}
\ee
Finally, \eqref{eq:projection-error-zero} gives
\be\label{eq:projection-error-w-zero}
\begin{aligned}
&
\sum_{k=1}^{\ell}
\left|
\int_\Omega
\left[
(PU_k)^{\frac{n+2}{n-2}}
-
U_k^{\frac{n+2}{n-2}}
\right]w\,\ud x
\right|\leq
\frac{c}{8}\|w\|^2
+
\begin{cases}
\displaystyle
C\sum_{k=1}^{\ell}\lambda_k^{-n-2},
& n>6,\\[2mm]
\displaystyle
C\sum_{k=1}^{\ell}
\lambda_k^{-8}(\ln\lambda_k)^{\frac43},
& n=6,\\[2mm]
\displaystyle
C\sum_{k=1}^{\ell}\lambda_k^{4-2n},
& n=3,4,5.
\end{cases}
\end{aligned}
\ee

After absorbing \(C\|w\|^{2+\delta(n)}\), we conclude that
\be\label{eq:w-preliminary-zero}
\begin{aligned}
\|w\|^2
\leq{}&
C
\left(
\int_\Omega
|\mathcal R_u-r_\infty|^{\frac{2n}{n+2}}
u^{\frac{2n}{n-2}}\,\ud x
\right)^{\frac{n+2}{n}}+
C\sum_{k=1}^{\ell}|\alpha_k-1|^2\\
&+
\begin{cases}
\displaystyle
C\sum_{k=1}^{\ell}\lambda_k^{-n-2},
& n>6,\\[2mm]
\displaystyle
C\sum_{k=1}^{\ell}
\lambda_k^{-8}(\ln\lambda_k)^{\frac43},
& n=6,\\[2mm]
\displaystyle
C\sum_{k=1}^{\ell}\lambda_k^{4-2n},
& n=3,4,5,
\end{cases}+
\begin{cases}
\displaystyle
C\sum_{i\neq j}
\Lambda_{i,j}^{\frac{n+2}{n-2}}
\left(
\ln\Lambda_{i,j}^{-1}
\right)^{\frac{n+2}{n}},
& n\geq6,\\[3mm]
\displaystyle
C\sum_{i\neq j}\Lambda_{i,j}^2,
& n=3,4,5.
\end{cases}
\end{aligned}
\ee
This proves \eqref{thm:w-0}, and together with \eqref{eq:amplitude-preliminary-zero}, proves \eqref{thm:alpha-1-0}.

\medskip

\noindent
\textbf{Step 4. Preliminary amplitude estimate when
$u_\infty>0$.}

\medskip

Now
\[
\sigma
=
\xi_{\bm z}
+
r_\infty^{-\frac{n-2}{4}}
\sum_{i=1}^{\ell}\alpha_iPU_i.
\]
As before,
\be\label{eq:amplitude-equation-positive}
\begin{aligned}
0
={}&
\int_\Omega
\left(
\Delta u+\mathcal R_u
u^{\frac{n+2}{n-2}}
\right)PU_j\,\ud x\\
={}&
\int_\Omega
\left(
\Delta\xi_{\bm z}
+
r_\infty u^{\frac{n+2}{n-2}}
\right)PU_j\,\ud x+
r_\infty^{-\frac{n-2}{4}}
\sum_{i=1}^{\ell}\alpha_i
\int_\Omega
\Delta PU_i\,PU_j\,\ud x\\
&+
\int_\Omega
\Delta w\,PU_j\,\ud x+
\int_\Omega
(\mathcal R_u-r_\infty)
u^{\frac{n+2}{n-2}}PU_j\,\ud x.
\end{aligned}
\ee

We are going to show first that
\be\label{eq:sigma-power-PUj-positive}
\begin{aligned}
\int_\Omega
\sigma^{\frac{n+2}{n-2}}PU_j\,\ud x
&=
r_\infty^{-\frac{n+2}{4}}
\alpha_j^{\frac{n+2}{n-2}}C_2+
\frac{2n}{n-2}
r_\infty^{-\frac{n+2}{4}}
C_1c_0
\sum_{i\neq j}\Lambda_{i,j}\\
&+
o(\bm\Lambda)
+
O\left(
\sum_{k=1}^{\ell}
\lambda_k^{-\frac{n-2}{2}}
\right)
+
O\left(
\sum_{k=1}^{\ell}|\alpha_k-1|^2
\right).
\end{aligned}
\ee

For $a,b\geq0$, one has
\[
0\leq
(a+b)^{\frac{n+2}{n-2}}
-b^{\frac{n+2}{n-2}}
\leq
C\left(
a^{\frac{n+2}{n-2}}
+
ab^{\frac4{n-2}}
\right).
\]
Applying this with
\[
a=\xi_{\bm z},
\qquad
b=
r_\infty^{-\frac{n-2}{4}}
\sum_{i=1}^{\ell}\alpha_iPU_i,
\]
we obtain
\[
\begin{aligned}
&
\left|
\int_\Omega
\sigma^{\frac{n+2}{n-2}}PU_j\,\ud x
-
r_\infty^{-\frac{n+2}{4}}
\int_\Omega
\left(
\sum_{i=1}^{\ell}\alpha_iPU_i
\right)^{\frac{n+2}{n-2}}
PU_j\,\ud x
\right|\\
&\quad\leq
C\int_\Omega
\xi_{\bm z}^{\frac{n+2}{n-2}}PU_j\,\ud x+
C\int_\Omega
\xi_{\bm z}
\left(
\sum_{i=1}^{\ell}\alpha_iPU_i
\right)^{\frac4{n-2}}
PU_j\,\ud x.
\end{aligned}
\]
Since $\xi_{\bm z}$ is uniformly bounded,
\[
\int_\Omega
\xi_{\bm z}^{\frac{n+2}{n-2}}PU_j\,\ud x
\leq
C\lambda_j^{-\frac{n-2}{2}}.
\]
Moreover,
\[
\left(
\sum_{i=1}^{\ell}\alpha_iPU_i
\right)^{\frac4{n-2}}
\leq
C\sum_{i=1}^{\ell}
(PU_i)^{\frac4{n-2}},
\]
and Young's inequality gives
\[
(PU_i)^{\frac4{n-2}}PU_j
\leq
C(PU_i)^{\frac{n+2}{n-2}}
+
C(PU_j)^{\frac{n+2}{n-2}}.
\]
Since
\[
\int_\Omega
(PU_i)^{\frac{n+2}{n-2}}\,\ud x
\leq
C\lambda_i^{-\frac{n-2}{2}},
\]
we conclude that
\[
\begin{aligned}
&
\int_\Omega
\sigma^{\frac{n+2}{n-2}}PU_j\,\ud x=
r_\infty^{-\frac{n+2}{4}}
\int_\Omega
\left(
\sum_{i=1}^{\ell}\alpha_iPU_i
\right)^{\frac{n+2}{n-2}}
PU_j\,\ud x+
O\left(
\sum_{k=1}^{\ell}
\lambda_k^{-\frac{n-2}{2}}
\right).
\end{aligned}
\]
Then, \eqref{eq:sigma-power-PUj-positive} follows from \eqref{eq:sigma-power-PUj-zero}.

Since
\[
\Pi\left(
\Delta\xi_{\bm z}
+
r_\infty\xi_{\bm z}^{\frac{n+2}{n-2}}
\right)=0,
\]
\eqref{eq:Pif}  and \eqref{eq:xi_zphi_l} imply
\[
\begin{aligned}
&
\int_\Omega
\left(
\Delta\xi_{\bm z}
+
r_\infty\xi_{\bm z}^{\frac{n+2}{n-2}}
\right)PU_j\,\ud x\\
&=
O\left(
\left(
\int_\Omega
|\mathcal R_u-r_\infty|^{\frac{2n}{n+2}}
u^{\frac{2n}{n-2}}\,\ud x
\right)^{\frac{n+2}{2n}}
\right)+
O\left(
\sum_{k=1}^{\ell}
\lambda_k^{-\frac{n-2}{2}}
\right).
\end{aligned}
\]
Consequently,
\be\label{eq:regular-term-positive}
\begin{aligned}
&
\int_\Omega
\left(
\Delta\xi_{\bm z}
+
r_\infty u^{\frac{n+2}{n-2}}
\right)PU_j\,\ud x\\
&=
r_\infty
\int_\Omega
\left(
u^{\frac{n+2}{n-2}}
-
\xi_{\bm z}^{\frac{n+2}{n-2}}
\right)PU_j\,\ud x+
O\left(
\left(
\int_\Omega
|\mathcal R_u-r_\infty|^{\frac{2n}{n+2}}
u^{\frac{2n}{n-2}}\,\ud x
\right)^{\frac{n+2}{2n}}
\right)\\
&\quad+
O\left(
\sum_{k=1}^{\ell}
\lambda_k^{-\frac{n-2}{2}}
\right).
\end{aligned}
\ee
Using \eqref{eq:u-power-reduction} and \eqref{eq:sigma-power-PUj-positive}, we obtain
\[
\begin{aligned}
\int_\Omega
u^{\frac{n+2}{n-2}}PU_j\,\ud x
&=
r_\infty^{-\frac{n+2}{4}}
\alpha_j^{\frac{n+2}{n-2}}C_2+
\frac{2n}{n-2}
r_\infty^{-\frac{n+2}{4}}
C_1c_0
\sum_{i\neq j}\Lambda_{i,j}\\
&\quad+
o(\bm\Lambda)
+
O\left(
\sum_{k=1}^{\ell}
\lambda_k^{-\frac{n-2}{2}}
\right)
+
O\left(
\sum_{k=1}^{\ell}|\alpha_k-1|^2
\right)+
O(\|w\|^2).
\end{aligned}
\]
Since \(\xi_{\bm z}\) is uniformly bounded,
\[
\int_\Omega
\xi_{\bm z}^{\frac{n+2}{n-2}}PU_j\,\ud x
\leq
C\lambda_j^{-\frac{n-2}{2}}.
\]
It follows that
\be\label{eq:regular-difference-positive}
\begin{aligned}
&
\int_\Omega
\left(
\Delta\xi_{\bm z}
+
r_\infty u^{\frac{n+2}{n-2}}
\right)PU_j\,\ud x\\
&=
r_\infty^{-\frac{n-2}{4}}
\alpha_j^{\frac{n+2}{n-2}}C_2+
\frac{2n}{n-2}
r_\infty^{-\frac{n-2}{4}}
C_1c_0
\sum_{i\neq j}\Lambda_{i,j}+
O\left(
\left(
\int_\Omega
|\mathcal R_u-r_\infty|^{\frac{2n}{n+2}}
u^{\frac{2n}{n-2}}\,\ud x
\right)^{\frac{n+2}{2n}}
\right)\\
&\quad+
o(\bm\Lambda)
+
O\left(
\sum_{k=1}^{\ell}
\lambda_k^{-\frac{n-2}{2}}
\right)
+
O\left(
\sum_{k=1}^{\ell}|\alpha_k-1|^2
\right)+
O(\|w\|^2).
\end{aligned}
\ee

We next improve the estimate of
\[
\int_\Omega \Delta w\,PU_j\,\ud x
\]
in the case $u_\infty>0$. We first prove the analogue of
\eqref{eq:Delta-w-zero-sharp}.

Let
\[
D
=
\left\{
x\in\Omega:
\xi_{\bm z}
\geq
r_\infty^{-\frac{n-2}{4}}
\sum_{i=1}^{\ell}\alpha_iPU_i
\right\},
\]
and
\[
A_j
=
\left\{
x\in\Omega:
\alpha_jPU_j
\geq
\sum_{i\neq j}\alpha_iPU_i
\right\}.
\]
After decreasing $\va_0$ if necessary, we may assume that
\[
\frac12\leq\alpha_i\leq2,
\qquad
1\leq i\leq\ell.
\]

We claim that
\be\label{eq:positive-mayer-v-2}
\begin{aligned}
\int_\Omega
\sigma^{\frac4{n-2}}PU_jw\,\ud x
={}&
r_\infty^{-1}
\alpha_j^{\frac4{n-2}}
\int_\Omega
(PU_j)^{\frac{n+2}{n-2}}w\,\ud x\\
&+
o\left(
\lambda_j^{-\frac{n-2}{2}}
\right)
+
o\left(
\sum_{i\neq j}\Lambda_{i,j}
\right)
+
O(\|w\|^2).
\end{aligned}
\ee

We begin with pointwise estimates. On $D$, we have
\[
r_\infty^{-\frac{n-2}{4}}
\sum_{i=1}^{\ell}\alpha_iPU_i
\leq
\xi_{\bm z}.
\]
Consequently,
\[
\sigma\leq2\xi_{\bm z},
\qquad
PU_j\leq C\xi_{\bm z},
\]
and therefore
\be\label{eq:positive-comparison-D}
\left|
\sigma^{\frac4{n-2}}
-
r_\infty^{-1}
\alpha_j^{\frac4{n-2}}
(PU_j)^{\frac4{n-2}}
\right|PU_j
\leq
C\xi_{\bm z}^{\frac4{n-2}}PU_j
\quad\mbox{on }D.
\ee

On $D^c\cap A_j$, we have
\[
\xi_{\bm z}
+
r_\infty^{-\frac{n-2}{4}}
\sum_{i\neq j}\alpha_iPU_i
\leq
CPU_j.
\]
For every $s>0$, if $a>0$ and $0\leq b\leq Ca$, then
\[
|(a+b)^s-a^s|
\leq
C a^{s-1}b.
\]
Applying this with
\[
s=\frac4{n-2},
\qquad
a=r_\infty^{-\frac{n-2}{4}}\alpha_jPU_j,
\]
and
\[
b=
\xi_{\bm z}
+
r_\infty^{-\frac{n-2}{4}}
\sum_{i\neq j}\alpha_iPU_i,
\]
we obtain
\[
\begin{aligned}
&
\left|
\sigma^{\frac4{n-2}}
-
r_\infty^{-1}
\alpha_j^{\frac4{n-2}}
(PU_j)^{\frac4{n-2}}
\right|PU_j\leq
C(PU_j)^{\frac4{n-2}}\xi_{\bm z}
+
C\sum_{i\neq j}
(PU_j)^{\frac4{n-2}}PU_i
\end{aligned}
\]
on $D^c\cap A_j$.

Finally, on $D^c\cap A_j^c$, we have
\[
\xi_{\bm z}
+
r_\infty^{-\frac{n-2}{4}}
\sum_{i=1}^{\ell}\alpha_iPU_i
\leq
C\sum_{i\neq j}PU_i.
\]
Hence
\be\label{eq:positive-comparison-DcAjc}
\begin{aligned}
&
\left|
\sigma^{\frac4{n-2}}
-
r_\infty^{-1}
\alpha_j^{\frac4{n-2}}
(PU_j)^{\frac4{n-2}}
\right|PU_j\leq
C
\left(
\sum_{i\neq j}PU_i
\right)^{\frac4{n-2}}PU_j
\end{aligned}
\ee
on $D^c\cap A_j^c$.

We first estimate the terms involving $\xi_{\bm z}$.

Suppose that $3\leq n\leq5$. Since $\xi_{\bm z}$ is uniformly
bounded, \eqref{eq:positive-comparison-D} gives
\[
\xi_{\bm z}^{\frac4{n-2}}PU_j
\leq
CPU_j
\qquad\mbox{on }D.
\]
Moreover, on $D$ we have $PU_j\leq C\xi_{\bm z}\leq M$ for some
constant $M>0$ independent of the parameters. By
\eqref{eq:PU-lower-interior}, we may choose $c>0$ sufficiently
small and then $\lambda_j$ sufficiently large so that
\[
D\cap
B_{c\lambda_j^{-1/2}}(a_j)
=
\varnothing.
\]
Indeed, for
$x\in B_{c\lambda_j^{-1/2}}(a_j)$,
\[
\begin{aligned}
PU_j(x)
\geq
C^{-1}U_j(x)\geq
C^{-1}c_0
\left(
\frac{\lambda_j}
{1+c^2\lambda_j}
\right)^{\frac{n-2}{2}}\geq
C^{-1}c_0
2^{-\frac{n-2}{2}}
c^{-(n-2)},
\end{aligned}
\]
which is larger than $M$ if $c$ is chosen sufficiently small.

It follows from the change of variables
$y=\lambda_j(x-a_j)$ that
\[
\begin{aligned}
&
\left(
\int_D
\left(
\xi_{\bm z}^{\frac4{n-2}}PU_j
\right)^{\frac{2n}{n+2}}
\,\ud x
\right)^{\frac{n+2}{n}}\leq
C
\left(
\int_{\{|x-a_j|\geq
c\lambda_j^{-1/2}\}\cap\Omega}
U_j^{\frac{2n}{n+2}}
\,\ud x
\right)^{\frac{n+2}{n}}\leq
C\lambda_j^{-(n-2)}.
\end{aligned}
\]
Similarly,
\[
\begin{aligned}
&
\left(
\int_{D^c\cap A_j}
\left(
(PU_j)^{\frac4{n-2}}\xi_{\bm z}
\right)^{\frac{2n}{n+2}}
\,\ud x
\right)^{\frac{n+2}{n}}\leq
C
\left(
\int_\Omega
U_j^{\frac{8n}{(n+2)(n-2)}}
\,\ud x
\right)^{\frac{n+2}{n}}\leq
C\lambda_j^{-(n-2)}.
\end{aligned}
\]

Suppose now that $n\geq6$. Since
$\frac4{n-2}\leq1$, on $D$ we have
\[
\begin{aligned}
\xi_{\bm z}^{\frac4{n-2}}PU_j
&=
(\xi_{\bm z}PU_j)^{\frac{n+2}{2(n-2)}}
\left(
\frac{PU_j}{\xi_{\bm z}}
\right)^{\frac{n-6}{2(n-2)}}\leq
C
(\xi_{\bm z}PU_j)^{\frac{n+2}{2(n-2)}}.
\end{aligned}
\]
Similarly, on $D^c\cap A_j$,
\[
\begin{aligned}
(PU_j)^{\frac4{n-2}}\xi_{\bm z}
&=
(\xi_{\bm z}PU_j)^{\frac{n+2}{2(n-2)}}
\left(
\frac{\xi_{\bm z}}{PU_j}
\right)^{\frac{n-6}{2(n-2)}}\leq
C
(\xi_{\bm z}PU_j)^{\frac{n+2}{2(n-2)}}.
\end{aligned}
\]
Since $\xi_{\bm z}$ is uniformly bounded and $PU_j\leq U_j$,
\[
\begin{aligned}
&
\left(
\int_D
\left(
\xi_{\bm z}^{\frac4{n-2}}PU_j
\right)^{\frac{2n}{n+2}}
\,\ud x
\right)^{\frac{n+2}{n}}+
\left(
\int_{D^c\cap A_j}
\left(
(PU_j)^{\frac4{n-2}}\xi_{\bm z}
\right)^{\frac{2n}{n+2}}
\,\ud x
\right)^{\frac{n+2}{n}}\\
&\quad\leq
C
\left(
\int_\Omega
U_j^{\frac n{n-2}}\,\ud x
\right)^{\frac{n+2}{n}}\\
&\quad\leq
C
\lambda_j^{-\frac{n+2}{2}}
(\ln\lambda_j)^{\frac{n+2}{n}}.
\end{aligned}
\]
In particular, the quantities in the above three estimates are
$
o\left(
\lambda_j^{-\frac{n-2}{2}}
\right).
$

We next estimate the bubble-interaction terms. When
$3\leq n\leq5$, Proposition B.2 of \cite{FG2020} gives, for
$i\neq j$,
\[
\begin{aligned}
&
\left(
\int_\Omega
\left(
(PU_j)^{\frac4{n-2}}PU_i
\right)^{\frac{2n}{n+2}}
\,\ud x
\right)^{\frac{n+2}{n}}+
\left(
\int_\Omega
\left(
(PU_i)^{\frac4{n-2}}PU_j
\right)^{\frac{2n}{n+2}}
\,\ud x
\right)^{\frac{n+2}{n}}\leq
C\Lambda_{i,j}^2
=
o(\Lambda_{i,j}).
\end{aligned}
\]

Suppose that $n\geq6$. On $A_j$, one has
$PU_i\leq CPU_j$ for $i\neq j$, and therefore
\[
(PU_j)^{\frac4{n-2}}PU_i
\leq
C
(PU_iPU_j)^{\frac{n+2}{2(n-2)}}.
\]
On $A_j^c$, divide the region into finitely many subregions
according to which $\alpha_iPU_i$, $i\neq j$, is maximal. On each
such subregion,
\[
\sum_{k\neq j}PU_k\leq CPU_i,
\qquad
PU_j\leq CPU_i,
\]
and hence
\[
\left(
\sum_{k\neq j}PU_k
\right)^{\frac4{n-2}}PU_j
\leq
C
(PU_iPU_j)^{\frac{n+2}{2(n-2)}}.
\]
By Proposition B.2 of \cite{FG2020},
\[
\begin{aligned}
&
\left(
\int_\Omega
(PU_iPU_j)^{\frac n{n-2}}
\,\ud x
\right)^{\frac{n+2}{n}}\leq
C
\Lambda_{i,j}^{\frac{n+2}{n-2}}
\left(
\ln\Lambda_{i,j}^{-1}
\right)^{\frac{n+2}{n}}
=
o(\Lambda_{i,j}).
\end{aligned}
\]

Since $\ell$ is fixed, it follows from
\eqref{eq:positive-comparison-D}--
\eqref{eq:positive-comparison-DcAjc} and (1)--(8) that
\[
\begin{aligned}
&
\left(
\int_\Omega
\left|
\sigma^{\frac4{n-2}}PU_j
-
r_\infty^{-1}
\alpha_j^{\frac4{n-2}}
(PU_j)^{\frac{n+2}{n-2}}
\right|^{\frac{2n}{n+2}}
\,\ud x
\right)^{\frac{n+2}{n}}\leq
o\left(
\lambda_j^{-\frac{n-2}{2}}
\right)
+
o\left(
\sum_{i\neq j}\Lambda_{i,j}
\right).
\end{aligned}
\]

By H\"older's, Sobolev's, and Young's inequalities, for every
fixed $\varepsilon>0$,
\[
\begin{aligned}
&
\left|
\int_\Omega
\left[
\sigma^{\frac4{n-2}}PU_j
-
r_\infty^{-1}
\alpha_j^{\frac4{n-2}}
(PU_j)^{\frac{n+2}{n-2}}
\right]w\,\ud x
\right|\leq
\varepsilon\|w\|^2
+
o\left(
\lambda_j^{-\frac{n-2}{2}}
\right)
+
o\left(
\sum_{i\neq j}\Lambda_{i,j}
\right).
\end{aligned}
\]
This proves \eqref{eq:positive-mayer-v-2}.

By \eqref{eq:quadratic-amplitude-orthogonality},
\[
\int_\Omega
\sigma^{\frac4{n-2}}PU_jw\,\ud x
=
O(\|w\|^2).
\]
Combining this with \eqref{eq:positive-mayer-v-2}, and using that
$\alpha_j$ is bounded above and below away from zero, we obtain
\be\label{eq:positive-mayer-v-3}
\int_\Omega
(PU_j)^{\frac{n+2}{n-2}}w\,\ud x
=
o\left(
\lambda_j^{-\frac{n-2}{2}}
\right)
+
o\left(
\sum_{i\neq j}\Lambda_{i,j}
\right)
+
O(\|w\|^2).
\ee
Hence,
\be\label{eq:Delta-w-positive}
\begin{aligned}
\int_\Omega
\Delta w\,PU_j\,\ud x
={}&
-\int_\Omega
U_j^{\frac{n+2}{n-2}}w\,\ud x\\
={}&
-\int_\Omega
(PU_j)^{\frac{n+2}{n-2}}w\,\ud x+
\int_\Omega
\left[
(PU_j)^{\frac{n+2}{n-2}}
-
U_j^{\frac{n+2}{n-2}}
\right]w\,\ud x\\
={}&
o\left(
\lambda_j^{-\frac{n-2}{2}}
\right)+
o\left(
\sum_{i\neq j}\Lambda_{i,j}
\right)
+
O(\|w\|^2).
\end{aligned}
\ee
where we used  \eqref{eq:positive-mayer-v-3} and
\eqref{eq:projection-error-zero} in the last equality.

Substituting
\eqref{eq:PUiUj^p}, \eqref{eq:PUiUi^p},
\eqref{eq:amplitude-residual},
\eqref{eq:regular-term-positive},
\eqref{eq:regular-difference-positive}, and
\eqref{eq:Delta-w-positive}
into \eqref{eq:amplitude-equation-positive}, we obtain
\[
\begin{aligned}
\left|
\alpha_j^{\frac{n+2}{n-2}}-\alpha_j
\right|
\leq{}&
C
\left(
\int_\Omega
|\mathcal R_u-r_\infty|^{\frac{2n}{n+2}}
u^{\frac{2n}{n-2}}\,\ud x
\right)^{\frac{n+2}{2n}}\\
&+
C\sum_{k=1}^{\ell}
\lambda_k^{-\frac{n-2}{2}}
+
C\bm\Lambda+
C\|w\|^2
+
C\sum_{k=1}^{\ell}|\alpha_k-1|^2.
\end{aligned}
\]
Summing over \(j\) and absorbing the quadratic amplitude term gives
\be\label{eq:amplitude-preliminary-positive}
\begin{aligned}
\sum_{k=1}^{\ell}|\alpha_k-1|
\leq{}&
C
\left(
\int_\Omega
|\mathcal R_u-r_\infty|^{\frac{2n}{n+2}}
u^{\frac{2n}{n-2}}\,\ud x
\right)^{\frac{n+2}{2n}}+
C\sum_{k=1}^{\ell}
\lambda_k^{-\frac{n-2}{2}}
+
C\bm\Lambda
+
C\|w\|^2.
\end{aligned}
\ee

\medskip

\noindent
\textbf{Step 5. Preliminary estimate of \(w\) when
\(u_\infty>0\).}

\medskip

Write
\[
\sigma
=
\xi_{\bm z}
+
r_\infty^{-\frac{n-2}{4}}
\sum_{k=1}^{\ell}\alpha_kPU_k.
\]
Subtracting the equation satisfied by the components of \(\sigma\)
from the equation for \(u\), and using
\eqref{eq:nonlinear-Taylor}, we obtain
\[
\begin{aligned}
-\Delta w
={}&
(\mathcal R_u-r_\infty)
u^{\frac{n+2}{n-2}}+
\frac{n+2}{n-2}r_\infty
\sigma^{\frac4{n-2}}w\\
&+
O\left(
\sigma^{\frac4{n-2}-\delta(n)}
|w|^{1+\delta(n)}
+
|w|^{\frac{n+2}{n-2}}
\right)\\
&+
\left(
\Delta\xi_{\bm z}
+
r_\infty\xi_{\bm z}^{\frac{n+2}{n-2}}
\right)\\
&+
r_\infty
\left[
\sigma^{\frac{n+2}{n-2}}
-
\xi_{\bm z}^{\frac{n+2}{n-2}}
-
\left(
r_\infty^{-\frac{n-2}{4}}
\sum_{k=1}^{\ell}\alpha_kPU_k
\right)^{\frac{n+2}{n-2}}
\right]\\
&+
r_\infty^{-\frac{n-2}{4}}
\left[
\left(
\sum_{k=1}^{\ell}\alpha_kPU_k
\right)^{\frac{n+2}{n-2}}
-
\sum_{k=1}^{\ell}
\alpha_k(PU_k)^{\frac{n+2}{n-2}}
\right]\\
&+
r_\infty^{-\frac{n-2}{4}}
\sum_{k=1}^{\ell}\alpha_k
\left[
(PU_k)^{\frac{n+2}{n-2}}
-
U_k^{\frac{n+2}{n-2}}
\right].
\end{aligned}
\]
Multiplying by \(w\), integrating, and using coercivity, the
residual, nonlinear, multi-bubble, and projection terms are
estimated as in
\eqref{eq:w-residual-estimate},
\eqref{eq:multi-bubble-w-estimate}, and
\eqref{eq:projection-error-w-zero}.

Since
$
\Pi\left(
\Delta\xi_{\bm z}
+
r_\infty\xi_{\bm z}^{\frac{n+2}{n-2}}
\right)=0,
$
\eqref{eq:Pif} and \eqref{eq:xi_zphi_l} imply
\[
\begin{aligned}
&
\left|
\int_\Omega
\left(
\Delta\xi_{\bm z}
+
r_\infty\xi_{\bm z}^{\frac{n+2}{n-2}}
\right)w\,\ud x
\right|\leq
\frac{c}{8}\|w\|^2
+
C
\left(
\int_\Omega
|\mathcal R_u-r_\infty|^{\frac{2n}{n+2}}
u^{\frac{2n}{n-2}}\,\ud x
\right)^{\frac{n+2}{n}}
+
C\sum_{k=1}^{\ell}\lambda_k^{2-n}.
\end{aligned}
\]

It remains to estimate the regular--bubble interaction. For
$A,B\geq0$ and $p>1$, one has
\[
\left|
(A+B)^p-A^p-B^p
\right|
\leq
C\left(
A^{p-1}B+AB^{p-1}
\right).
\]
Moreover, if $1<p\leq2$, then
\[
\left|
(A+B)^p-A^p-B^p
\right|
\leq
C(AB)^{p/2}.
\]

Suppose first that $3\leq n\leq5$. Applying the first inequality
with
\[
A=\xi_{\bm z},
\qquad
B=
r_\infty^{-\frac{n-2}{4}}
\sum_{k=1}^{\ell}\alpha_kPU_k,
\qquad
p=\frac{n+2}{n-2},
\]
and using the uniform boundedness of $\xi_{\bm z}$, we obtain
\[
\begin{aligned}
&
\left|
\sigma^{\frac{n+2}{n-2}}
-
\xi_{\bm z}^{\frac{n+2}{n-2}}
-
\left(
r_\infty^{-\frac{n-2}{4}}
\sum_{k=1}^{\ell}\alpha_kPU_k
\right)^{\frac{n+2}{n-2}}
\right|\leq
C\sum_{k=1}^{\ell}PU_k
+
C\sum_{k=1}^{\ell}
(PU_k)^{\frac4{n-2}}.
\end{aligned}
\]
Here we used that $\ell$ is fixed and
\[
\left(
\sum_{k=1}^{\ell}PU_k
\right)^{\frac4{n-2}}
\leq
C\sum_{k=1}^{\ell}
(PU_k)^{\frac4{n-2}}.
\]
Direct scaling gives
\[
\left(
\int_\Omega
(PU_k)^{\frac{2n}{n+2}}\,\ud x
\right)^{\frac{n+2}{2n}}
\leq
C\lambda_k^{-\frac{n-2}{2}},
\]
and
\[
\left(
\int_\Omega
(PU_k)^{
\frac{8n}{(n+2)(n-2)}
}\,\ud x
\right)^{\frac{n+2}{2n}}
\leq
C\lambda_k^{-\frac{n-2}{2}}.
\]
It follows that
\[
\begin{aligned}
&
\left\|
\sigma^{\frac{n+2}{n-2}}
-
\xi_{\bm z}^{\frac{n+2}{n-2}}
-
\left(
r_\infty^{-\frac{n-2}{4}}
\sum_{k=1}^{\ell}\alpha_kPU_k
\right)^{\frac{n+2}{n-2}}
\right\|_{L^{\frac{2n}{n+2}}(\Omega)}\leq
C\sum_{k=1}^{\ell}
\lambda_k^{-\frac{n-2}{2}},
\qquad
3\leq n\leq5.
\end{aligned}
\]

Suppose now that $n\geq6$. Since
\[
1<
\frac{n+2}{n-2}
\leq2,
\]
the second algebraic inequality gives
\[
\begin{aligned}
&
\left|
\sigma^{\frac{n+2}{n-2}}
-
\xi_{\bm z}^{\frac{n+2}{n-2}}
-
\left(
r_\infty^{-\frac{n-2}{4}}
\sum_{k=1}^{\ell}\alpha_kPU_k
\right)^{\frac{n+2}{n-2}}
\right|\leq
C
\left[
\xi_{\bm z}
\left(
\sum_{k=1}^{\ell}PU_k
\right)
\right]^{\frac{n+2}{2(n-2)}}.
\end{aligned}
\]
Consequently,
\[
\begin{aligned}
\left\|
\sigma^{\frac{n+2}{n-2}}
-
\xi_{\bm z}^{\frac{n+2}{n-2}}
-
\left(
r_\infty^{-\frac{n-2}{4}}
\sum_{k=1}^{\ell}\alpha_kPU_k
\right)^{\frac{n+2}{n-2}}
\right\|_{L^{\frac{2n}{n+2}}(\Omega)}&\leq
C
\left(
\int_\Omega
\xi_{\bm z}^{\frac n{n-2}}
\left(
\sum_{k=1}^{\ell}PU_k
\right)^{\frac n{n-2}}
\,\ud x
\right)^{\frac{n+2}{2n}}\\
&\quad\leq
C
\left(
\sum_{k=1}^{\ell}
\int_\Omega
(PU_k)^{\frac n{n-2}}\,\ud x
\right)^{\frac{n+2}{2n}}.
\end{aligned}
\]
Since
\[
\begin{aligned}
\int_\Omega
(PU_k)^{\frac n{n-2}}\,\ud x
&\leq
C\lambda_k^{-\frac n2}
\int_0^{C\lambda_k}
\frac{r^{n-1}}
{(1+r^2)^{\frac n2}}\,\ud r\leq
C\lambda_k^{-\frac n2}\ln\lambda_k,
\end{aligned}
\]
we conclude that
\[
\begin{aligned}
&
\left\|
\sigma^{\frac{n+2}{n-2}}
-
\xi_{\bm z}^{\frac{n+2}{n-2}}
-
\left(
r_\infty^{-\frac{n-2}{4}}
\sum_{k=1}^{\ell}\alpha_kPU_k
\right)^{\frac{n+2}{n-2}}
\right\|_{L^{\frac{2n}{n+2}}(\Omega)}\leq
C\sum_{k=1}^{\ell}
\lambda_k^{-\frac{n+2}{4}}
(\ln\lambda_k)^{\frac{n+2}{2n}},
\  n\geq6.
\end{aligned}
\]
This proves 
\[
\begin{aligned}
&
\left\|
\sigma^{\frac{n+2}{n-2}}
-
\xi_{\bm z}^{\frac{n+2}{n-2}}
-
\left(
r_\infty^{-\frac{n-2}{4}}
\sum_{k=1}^{\ell}\alpha_kPU_k
\right)^{\frac{n+2}{n-2}}
\right\|_{L^{\frac{2n}{n+2}}(\Omega)}\leq
\begin{cases}
\displaystyle
C\sum_{k=1}^{\ell}
\lambda_k^{-\frac{n-2}{2}},
\ 3\leq n\leq5,\\[3mm]
\displaystyle
C\sum_{k=1}^{\ell}
\lambda_k^{-\frac{n+2}{4}}
(\ln\lambda_k)^{\frac{n+2}{2n}},
\  n\geq6.
\end{cases}
\end{aligned}
\]

Applying H\"older's and Young's inequalities, we obtain
\[
\begin{aligned}
&\left|
\int_\Omega
\left[
\sigma^{\frac{n+2}{n-2}}
-
\xi_{\bm z}^{\frac{n+2}{n-2}}
-
\left(
r_\infty^{-\frac{n-2}{4}}
\sum_{k=1}^{\ell}\alpha_kPU_k
\right)^{\frac{n+2}{n-2}}
\right]w\,\ud x
\right|\\
&\leq
\frac{c}{8}\|w\|^2+
\begin{cases}
\displaystyle
C\sum_{k=1}^{\ell}\lambda_k^{2-n},
&3\leq n\leq5,\\[3mm]
\displaystyle
C\sum_{k=1}^{\ell}
\lambda_k^{-\frac{n+2}{2}}
(\ln\lambda_k)^{\frac{n+2}{n}},
&n\geq6.
\end{cases}
\end{aligned}
\]

We therefore obtain
\be\label{eq:w-preliminary-positive}
\begin{aligned}
\|w\|^2
\leq{}&
C
\left(
\int_\Omega
|\mathcal R_u-r_\infty|^{\frac{2n}{n+2}}
u^{\frac{2n}{n-2}}\,\ud x
\right)^{\frac{n+2}{n}}+
C\sum_{k=1}^{\ell}|\alpha_k-1|^2\\
&+
\begin{cases}
\displaystyle
C\sum_{k=1}^{\ell}\lambda_k^{2-n}+C\sum_{i\neq j}\Lambda_{i,j}^2,
&3\leq n\leq5,\\[3mm]
\displaystyle
C\sum_{k=1}^{\ell}
\lambda_k^{-\frac{n+2}{2}}
(\ln\lambda_k)^{\frac{n+2}{n}}+C\sum_{i\neq j}
\Lambda_{i,j}^{\frac{n+2}{n-2}}
\left(
\ln\Lambda_{i,j}^{-1}
\right)^{\frac{n+2}{n}},
&n\geq6,
\end{cases}
\end{aligned}
\ee
This, together with
\eqref{eq:amplitude-preliminary-positive}, proves \eqref{thm:w} and \eqref{thm:alpha-1}. 
\medskip

\noindent
\textbf{Step 6. Improved estimates when
\(Y_\Omega(u)\geq r_\infty\).}

\medskip

Suppose first that \(u_\infty\equiv0\). 
Combining \eqref{eq:energy-sigma-alpha} with
\eqref{eq:Ruup<}, and using the already established general
amplitude estimate to control the quadratic amplitude term, we
obtain, under \(Y_\Omega(u)\geq r_\infty\),
\be\label{eq:oneside-1-1-strengthened}
\begin{aligned}
\bm\Lambda+\|w\|^2
\leq{}&
C
\left(
\int_\Omega
|\mathcal R_u-r_\infty|^{\frac{2n}{n+2}}
u^{\frac{2n}{n-2}}\,\ud x
\right)^{\frac{n+2}{n}}+
C\sum_{k=1}^{\ell}\lambda_k^{2-n}.
\end{aligned}
\ee
By substituting \eqref{eq:oneside-1-1-strengthened} into the
preliminary estimates
\eqref{eq:amplitude-preliminary-zero} and \eqref{eq:w-preliminary-zero}, 
we obtain,
\[
\begin{aligned}
\sum_{k=1}^{\ell}|\alpha_k-1|
={}&
O\left(
\left(
\int_\Omega
|\mathcal R_u-r_\infty|^{\frac{2n}{n+2}}
u^{\frac{2n}{n-2}}\,\ud x
\right)^{\frac{n+2}{2n}}
\right)+
O\left(
\sum_{k=1}^{\ell}
\lambda_k^{
2-n}
\right),
\end{aligned}
\]
and
\[
\begin{aligned}
\|w\|^2
={}&
O\left(
\left(
\int_\Omega
|\mathcal R_u-r_\infty|^{\frac{2n}{n+2}}
u^{\frac{2n}{n-2}}\,\ud x
\right)^{\frac{n+2}{n}}
\right)\\
&+
\begin{cases}
\displaystyle
O\left(
\sum_{k=1}^{\ell}\lambda_k^{-n-2}(\ln\lambda_k)^\frac{n+2}{n}
\right),
& n\geq6,\\[2mm]
\displaystyle
O\left(
\sum_{k=1}^{\ell}\lambda_k^{4-2n}
\right),
& n=3,4,5.
\end{cases}
\end{aligned}
\]
These are the one-sided assertions in parts (i) and (ii) for
\(u_\infty\equiv0\).

Suppose now that \(u_\infty>0\). The estimate
\eqref{eq:Bruinfty>0} and the one-sided coercive energy expansion
\eqref{eq:Ruup<} imply, under \(Y_\Omega(u)\geq r_\infty\),
\be\label{eq:oneside-positive-strengthened}
\begin{aligned}
&
\sum_{k=1}^{\ell}
\lambda_k^{-\frac{n-2}{2}}
+
\bm\Lambda
+
\|w\|^2\leq
C
\left(
\int_\Omega
|\mathcal R_u-r_\infty|^{\frac{2n}{n+2}}
u^{\frac{2n}{n-2}}\,\ud x
\right)^{\frac{n+2}{2n}(1+\gamma)}.
\end{aligned}
\ee
Substitution into
\eqref{eq:amplitude-preliminary-positive} gives
\[
\sum_{k=1}^{\ell}|\alpha_k-1|
=
O\left(
\left(
\int_\Omega
|\mathcal R_u-r_\infty|^{\frac{2n}{n+2}}
u^{\frac{2n}{n-2}}\,\ud x
\right)^{\frac{n+2}{2n}}
\right).
\]

Finally, substituting \eqref{eq:oneside-positive-strengthened} into
\eqref{eq:w-preliminary-positive}, 
we obtain \eqref{eq:w4}.

This completes the proof of parts (i) and (ii).
\end{proof}

\begin{proof}[Proof of Theorem \ref{thm:LS}\,(iii)]

This part concerns only the case $u_\infty>0$. Let $L_1$ be the
largest integer such that
\[
\mu_l<\frac{n+2}{n-2}r_\infty,
\qquad 1\leq l\leq L_1.
\]
We set
\[
\bm z'=(z_1,\ldots,z_{L_1}),
\qquad
\bm z''=(z_{L_1+1},\ldots,z_L),
\]
with the convention $\bm z''=0$ if $L_1=L$.

By Lemma \ref{lem:regular-surface-expansions}, the expansions
\eqref{eq:nablaxiz} and \eqref{eq:xiz^p} hold.

Using 
\[
\xi_{\bm z}
=u_\infty+
\sum_{i=1}^{L}z_i\phi_i+h_{\bm z},
\qquad
\|h_{\bm z}\|=O(|\bm z|^2).
\]
and expanding the quotient $Y_\Omega(\xi_{\bm z})$ to second order, including the square of the
linear term in the denominator, gives
\be\label{eq:Y-xiz-quadratic}
\begin{aligned}
Y_\Omega(\xi_{\bm z})
={}&Y_\Omega(u_\infty)+
\frac{
\displaystyle
\sum_{i=2}^{L}
\left(
\mu_i-\frac{n+2}{n-2}r_\infty
\right)z_i^2
}{
\left(
\displaystyle
\int_\Omega u_\infty^{\frac{2n}{n-2}}\,\ud x
\right)^{\frac{n-2}{n}}
}
+o(|\bm z|^2).
\end{aligned}
\ee
For $2\leq i\leq L_1$,
\[
\mu_i-\frac{n+2}{n-2}r_\infty<0,
\]
whereas this coefficient vanishes for $L_1+1\leq i\leq L$.
Therefore, Proposition \ref{prop:regular-part} and
\eqref{eq:Y-xiz-quadratic} imply
\be\label{eq:negative-z-estimate}
\begin{aligned}
\sum_{i=2}^{L_1}z_i^2
={}&
o\left(
 z_1^2+\sum_{i=L_1+1}^{L}z_i^2
\right)+
O\left(
\left(
\int_\Omega
|\mathcal R_u-r_\infty|^{\frac{2n}{n+2}}
 u^{\frac{2n}{n-2}}\,\ud x
\right)^{\frac{n+2}{2n}(1+\gamma)}
\right)\\
&+
O\left(
\sum_{k=1}^{\ell}
\lambda_k^{-\frac{n-2}{2}(1+\gamma)}
\right).
\end{aligned}
\ee

It remains to estimate the scaling parameter $z_1$. Subtracting
$r_\infty$ times \eqref{eq:xiz^p} from \eqref{eq:nablaxiz} and
using Lemma \ref{lem:regular-energy-balance}, we obtain
\be\label{eq:z1-balance}
\begin{aligned}
&\frac4{n-2}r_\infty
\left(
\int_\Omega u_\infty^{\frac{2n}{n-2}}\,\ud x
\right)^{1/2}z_1+
\sum_{i=1}^{L}
\left(
\frac{n(n+2)}{(n-2)^2}r_\infty-\mu_i
\right)z_i^2\\
&=
O\left(
\left(
\int_\Omega
|\mathcal R_u-r_\infty|^{\frac{2n}{n+2}}
 u^{\frac{2n}{n-2}}\,\ud x
\right)^{\frac{n+2}{2n}}
\right)
+O\left(
\sum_{k=1}^{\ell}\lambda_k^{-\frac{n-2}{2}}
\right)
+o(|\bm z|^2).
\end{aligned}
\ee
Since $|\bm z|$ is small, the term involving $z_1^2$ on the
right-hand side of the resulting estimate for $|z_1|$ can be
absorbed. Combining \eqref{eq:z1-balance} with
\eqref{eq:negative-z-estimate} yields
\be\label{eq:z1-estimate}
\begin{aligned}
|z_1|
\leq{}&
C
\left(
\int_\Omega
|\mathcal R_u-r_\infty|^{\frac{2n}{n+2}}
u^{\frac{2n}{n-2}}\,\ud x
\right)^{\frac{n+2}{2n}}+
C\sum_{k=1}^{\ell}
\lambda_k^{-\frac{n-2}{2}}
+
C\sum_{i=L_1+1}^{L}z_i^2.
\end{aligned}
\ee
Here we used $1+\gamma<2$ and the smallness of all quantities.
Adding \eqref{eq:z1-estimate} to
\eqref{eq:negative-z-estimate} proves
\[
\begin{aligned}
|\bm z'|^2
={}&
o(|\bm z''|^2)+
O\left(
\left(
\int_\Omega
|\mathcal R_u-r_\infty|^{\frac{2n}{n+2}}
 u^{\frac{2n}{n-2}}\,\ud x
\right)^{\frac{n+2}{2n}(1+\gamma)}
\right)+
O\left(
\sum_{k=1}^{\ell}
\lambda_k^{-\frac{n-2}{2}(1+\gamma)}
\right).
\end{aligned}
\]
The little-$o$ term is uniform in the regime defining
$\V(u_\infty,\ell,\delta;\va)$ as $\va\to0$.

Under the additional assumption $Y_\Omega(u)\geq r_\infty$,
\eqref{eq:oneside-positive-strengthened} gives
\[
\sum_{k=1}^{\ell}\lambda_k^{-\frac{n-2}{2}}
\leq
C\left(
\int_\Omega
|\mathcal R_u-r_\infty|^{\frac{2n}{n+2}}
 u^{\frac{2n}{n-2}}\,\ud x
\right)^{\frac{n+2}{2n}(1+\gamma)}.
\]
The scale term above can therefore be absorbed into the residual
term, which proves the final assertion of part~(iii).
\end{proof}

\section{Energy-gap consequences}\label{sec:energy-gap}

We now combine the preliminary identities from
Section~\ref{sec:modulation} with Theorem~\ref{thm:LS}.

\subsection{The case without a regular component}

\begin{thm}\label{prop:energ-expand-a}
Suppose that $u_\infty\equiv0$. There exist $\va_1\in(0,1)$ and
$C>0$ such that, for every $\va\in(0,\va_1)$ and every
$
u\in\mathcal D\cap\V(0,\ell,\delta;\va)
$
satisfying
$
\|\mathcal R_u-r_\infty\|_{C^0(\Omega)}<\va
$
and
$
\int_\Omega u^{\frac{2n}{n-2}}\,\ud x=1,
$
we have
\[
\begin{aligned}
&-C\mathfrak G(u)^2
-C\bm\Lambda
-C\sum_{k=1}^{\ell}\lambda_k^{2-n}\leq
Y_\Omega(u)-r_\infty\leq
C
\mathfrak G(u)^2
+C\sum_{k=1}^{\ell}\lambda_k^{2-n},
\end{aligned}
\]
where $\mathfrak G(u)$ is defined in \eqref{eq:Gu}. If, in addition, $Y_\Omega(u)\geq r_\infty$, then
\be\label{eq:oneside-1-1}
\begin{aligned}
\bm\Lambda + \|w\|^2
\leq{}&
C\mathfrak G(u)^2
+C\sum_{k=1}^{\ell}\lambda_k^{2-n}.
\end{aligned}
\ee
\end{thm}

\begin{proof}
By Proposition~\ref{prop:quantif-1} and
\eqref{thm:alpha-1-0},
\[
\begin{aligned}
Y_\Omega(\sigma)-r_\infty
={}&
-C_1c_0r_\infty^{-\frac{n-2}{2}}\bm\Lambda+
O\left(\mathfrak G(u)^2\right)+o(\bm\Lambda)
+O\left(\sum_{k=1}^{\ell}\lambda_k^{2-n}\right).
\end{aligned}
\]
Moreover, \eqref{eq:mass-comparison} and Theorem~\ref{thm:LS}(ii) shows that \eqref{eq:estsigma} hold. Substituting into \eqref{eq:Ruup}, together with
\eqref{thm:w-0}, gives the lower bound of $Y_\Omega(u)$. Applying \eqref{eq:Ruup<} therefore gives
\be\label{eq:Yuupper}
\begin{aligned}
Y_\Omega(u)
\leq{}&r_\infty
+C\mathfrak G(u)^2+C\sum_{k=1}^{\ell}\lambda_k^{2-n}
-c\bm\Lambda-c\|w\|^2.
\end{aligned}
\ee
This proves the upper bound. If $Y_\Omega(u)\geq r_\infty$, then \eqref{eq:Yuupper} immediately yields \eqref{eq:oneside-1-1}.
\end{proof}

\subsection{The case with a regular component}

\begin{prop}\label{pro:Yom(u)}
Suppose that $u_\infty>0$. Under the assumptions of
Theorem~\ref{thm:LS},
\[
\begin{aligned}
& |Y_\Omega(\sigma)-r_\infty|\\
&=
O
\left( \mathfrak G(u)
\right)+
\begin{cases}
\displaystyle
O\left(\sum_{k=1}^{\ell}
\lambda_k^{-\frac{n+2}{4}}
(\ln\lambda_k)^{\frac{n+2}{2n}}\right)+O\left(\sum_{i\neq j}
\Lambda_{i,j}^{\frac{n+2}{2(n-2)}}
\left(
\ln\Lambda_{i,j}^{-1}
\right)^{\frac{n+2}{2n}}\right),
&n\geq6,\\[3mm]
\displaystyle
O\left(\sum_{k=1}^{\ell}\lambda_k^{\frac{2-n}{2}}\right)+O\left(\sum_{i\neq j}\Lambda_{i,j}\right),
&n=3,4,5,
\end{cases}
\end{aligned}
\]
\end{prop}

\begin{proof}
We first estimate
\[
\int_\Omega|\nabla\sigma|^2\,\ud x
-
r_\infty
\int_\Omega\sigma^{\frac{2n}{n-2}}\,\ud x.
\]
Recall that
\[
\sigma
=
\xi_{\bm z}
+
r_\infty^{-\frac{n-2}{4}}
\sum_{i=1}^{\ell}\alpha_iPU_i.
\]

We shall prove that
\be\label{eq:sigma-energy-balance-positive}
\begin{aligned}
&
\left|
\int_\Omega|\nabla\sigma|^2\,\ud x
-
r_\infty
\int_\Omega\sigma^{\frac{2n}{n-2}}\,\ud x
\right|\\
&\quad\leq
\left|
\int_\Omega|\nabla\xi_{\bm z}|^2\,\ud x
-
r_\infty
\int_\Omega\xi_{\bm z}^{\frac{2n}{n-2}}\,\ud x
\right|+
C\sum_{i=1}^{\ell}|\alpha_i-1|
+
C\bm\Lambda
+
C\sum_{i=1}^{\ell}
\lambda_i^{-\frac{n-2}{2}}.
\end{aligned}
\ee

We first expand the gradient term. Since
\[
-\Delta PU_i
=
U_i^{\frac{n+2}{n-2}}
\qquad\mbox{in }\Omega,
\]
integration by parts gives
\[
\begin{aligned}
\int_\Omega|\nabla\sigma|^2\,\ud x
=&\int_{\Omega}|\nabla \xi_{\boldsymbol{z}}|^2 \,\ud x+2r_\infty^{-\frac{n-2}{4}}
\sum_{i=1}^{\ell}\alpha_i
\int_\Omega
\xi_{\bm z}U_i^{\frac{n+2}{n-2}}\,\ud x +r_{\infty}^{-\frac{n-2}{2}}\int_{\Omega}\Big|\nabla\Big(\sum_{k=1}^{\ell} \alpha_k P U_k\Big)\Big|^2 \,\ud x\\
\end{aligned}
\]

Since $\xi_{\bm z}$ is uniformly bounded and
\[
\int_\Omega
U_i^{\frac{n+2}{n-2}}\,\ud x
=
O\left(
\lambda_i^{-\frac{n-2}{2}}
\right),
\]
we have
\[
\left|
\int_\Omega
\xi_{\bm z}U_i^{\frac{n+2}{n-2}}\,\ud x
\right|
\leq
C\lambda_i^{-\frac{n-2}{2}}.
\]

We next expand the mass. For $a,b\geq0$, one has
\[
\left|
(a+b)^{\frac{2n}{n-2}}
-a^{\frac{2n}{n-2}}
-b^{\frac{2n}{n-2}}
\right|
\leq
C\left(
a^{\frac{n+2}{n-2}}b
+
ab^{\frac{n+2}{n-2}}
\right).
\]
Applying this inequality with
\[
a=\xi_{\bm z},
\qquad
b=
r_\infty^{-\frac{n-2}{4}}
\sum_{i=1}^{\ell}\alpha_iPU_i,
\]
we obtain
\[
\begin{aligned}
&
\left|
\int_\Omega\sigma^{\frac{2n}{n-2}}\,\ud x
-
\int_\Omega\xi_{\bm z}^{\frac{2n}{n-2}}\,\ud x
-
r_\infty^{-\frac n2}
\int_\Omega
\left(
\sum_{i=1}^{\ell}\alpha_iPU_i
\right)^{\frac{2n}{n-2}}
\,\ud x
\right|\\
&\quad\leq
C\sum_{i=1}^{\ell}
\int_\Omega
\xi_{\bm z}^{\frac{n+2}{n-2}}PU_i\,\ud x+
C\int_\Omega
\xi_{\bm z}
\left(
\sum_{i=1}^{\ell}PU_i
\right)^{\frac{n+2}{n-2}}
\,\ud x\\
&\quad\le C\sum_{i=1}^{\ell}\int_\Omega PU_i\,\ud x\leq+C\sum_{i=1}^{\ell}
\int_\Omega
(PU_i)^{\frac{n+2}{n-2}}\,\ud x\\
&\quad \leq
C\sum_{i=1}^{\ell}
\lambda_i^{-\frac{n-2}{2}}.
\end{aligned}
\]
Hence,
\[
\begin{aligned}
&
\int_\Omega|\nabla\sigma|^2\,\ud x
-
r_\infty
\int_\Omega\sigma^{\frac{2n}{n-2}}\,\ud x\\
={}&
\int_\Omega|\nabla\xi_{\bm z}|^2\,\ud x
-
r_\infty
\int_\Omega\xi_{\bm z}^{\frac{2n}{n-2}}\,\ud x\\
&+
r_{\infty}^{-\frac{n-2}{2}}\int_{\Omega}\Big|\nabla\Big(\sum_{k=1}^{\ell} \alpha_k P U_k\Big)\Big|^2 \,\ud x-r_\infty^{-\frac{n-2}{2}}\int_\Omega
\left(
\sum_{i=1}^{\ell}\alpha_iPU_i
\right)^{\frac{2n}{n-2}}
\,\ud x
\\
&+
O\left(
\sum_{i=1}^{\ell}
\lambda_i^{-\frac{n-2}{2}}
\right)
+
O(\bm\Lambda).
\end{aligned}
\]
By \eqref{eq:sigma1} and \eqref{eq:sigma22}, we have
\[
\begin{aligned}
&\int_{\Omega}\Big|\nabla\Big(\sum_{k=1}^{\ell} \alpha_k P U_k\Big)\Big|^2 \,\ud x-\int_\Omega
\left(
\sum_{i=1}^{\ell}\alpha_iPU_i
\right)^{\frac{2n}{n-2}}
\,\ud x\\
&=C_2\sum_{i=1}^{\ell}\left(\alpha_i^2-\alpha_i^\frac{2n}{n-2}\right)+
O\left(
\sum_{i=1}^{\ell}
\lambda_i^{2-n}
\right)
+
O(\bm\Lambda).
\end{aligned}
\]
This proves \eqref{eq:sigma-energy-balance-positive}. Using Lemma \ref{lem:regular-energy-balance} and Theorem \ref{thm:LS}(i), we conclude that
\be\label{eq:sigma-energy-balance-final}
\begin{aligned}
&
\left|
\int_\Omega|\nabla\sigma|^2\,\ud x
-
r_\infty
\int_\Omega\sigma^{\frac{2n}{n-2}}\,\ud x
\right|\\
&\quad\leq
C
\left(
\int_\Omega
|\mathcal R_u-r_\infty|^{\frac{2n}{n+2}}
u^{\frac{2n}{n-2}}\,\ud x
\right)^{\frac{n+2}{2n}}+
C\bm\Lambda
+
C\sum_{i=1}^{\ell}
\lambda_i^{-\frac{n-2}{2}}.
\end{aligned}
\ee

We next convert \eqref{eq:sigma-energy-balance-final} into an
estimate of the Yamabe quotient. Using the exact identity
\[
\begin{aligned}
Y_\Omega(\sigma)-r_\infty
={}&
\frac{
\displaystyle
\int_\Omega|\nabla\sigma|^2\,\ud x
-
r_\infty
\displaystyle\int_\Omega
\sigma^{\frac{2n}{n-2}}\,\ud x
}{
\left(
\displaystyle\int_\Omega
\sigma^{\frac{2n}{n-2}}\,\ud x
\right)^{\frac{n-2}{n}}
}+
r_\infty
\left[
\left(
\int_\Omega
\sigma^{\frac{2n}{n-2}}\,\ud x
\right)^{\frac2n}
-1
\right],
\end{aligned}
\tag{14}
\]
Together with \eqref{eq:mass-comparison} and Theorem \ref{thm:LS}(ii),  this completes the proof the proposition.
\end{proof}

\begin{thm}\label{prop:energ-expand-b}
Suppose that $u_\infty>0$. There exist $\va_1\in(0,1)$ and $C>0$
such that, for every $\va\in(0,\va_1)$ and every
$
u\in\mathcal D\cap\V(u_\infty,\ell,\delta;\va)
$
satisfying
$
\|\mathcal R_u-r_\infty\|_{C^0(\Omega)}<\va
$
and
$
\int_\Omega u^{\frac{2n}{n-2}}\,\ud x=1,
$
we have
\[
\begin{aligned}
&Y_\Omega(u)-r_\infty\\
&\geq
-C\mathfrak G(u)-
\begin{cases}
\displaystyle
C\sum_{k=1}^{\ell}
\lambda_k^{-\frac{n+2}{4}}
(\ln\lambda_k)^{\frac{n+2}{2n}}+C\sum_{i\neq j}
\Lambda_{i,j}^{\frac{n+2}{2(n-2)}}
\left(
\ln\Lambda_{i,j}^{-1}
\right)^{\frac{n+2}{2n}},
&n\geq6,\\[3mm]
\displaystyle
C\sum_{k=1}^{\ell}\lambda_k^{\frac{2-n}{2}}+C\sum_{i\neq j}\Lambda_{i,j},
&n=3,4,5,
\end{cases}
\end{aligned}
\]
and
\[
Y_\Omega(u)-r_\infty
\leq
C\mathfrak G(u)^{1+\gamma}.
\]
If, in addition, $Y_\Omega(u)\geq r_\infty$, then
\be\label{eq:oneside-2}
\begin{aligned}
\sum_{k=1}^{\ell}\lambda_k^{-\frac{n-2}{2}}
+\bm\Lambda+\|w\|^2
\leq
C\mathfrak G(u)^{1+\gamma}.
\end{aligned}
\ee
\end{thm}

\begin{proof}
The lower bound follows from Proposition~\ref{pro:Yom(u)}. For the upper bound,
combine \eqref{eq:Bruinfty>0} with \eqref{eq:Ruup<}; after
shrinking $\va_0$, the mass factor in \eqref{eq:Ruup<} is bounded
above and below away from zero. This gives
\[
\begin{aligned}
Y_\Omega(u)
\leq{}&r_\infty
+C
\left(
\int_\Omega
|\mathcal R_u-r_\infty|^{\frac{2n}{n+2}}
 u^{\frac{2n}{n-2}}\,\ud x
\right)^{\frac{n+2}{2n}(1+\gamma)}-c\sum_{k=1}^{\ell}\lambda_k^{-\frac{n-2}{2}}
-c\bm\Lambda-c\|w\|^2.
\end{aligned}
\]
The upper bound follows immediately. Under the
additional assumption $Y_\Omega(u)\geq r_\infty$, the same
inequality gives \eqref{eq:oneside-2}.
\end{proof}

\bigskip

\noindent T. Jin

\noindent Department of Mathematics, The Hong Kong University of Science and Technology\\
Clear Water Bay, Kowloon, Hong Kong.  ~  \textsf{Email: tianlingjin@ust.hk}

\bigskip

\noindent J. Xiong

\noindent School of Mathematical Sciences, Beijing Normal University\\
Beijing 100875, China. ~ \textsf{Email: jx@bnu.edu.cn}

\bigskip

\noindent N. Zhou

\noindent Department of Mathematics and LPMC, Nankai University\\
Tianjin 300071, China. ~ \textsf{Email: zhouning@nankai.edu.cn}

\end{document}